\documentclass[11pt,leqno]{article}
\usepackage{amsfonts}
\usepackage{amsmath}
\usepackage{amssymb}
\usepackage{amsthm}
\usepackage{latexsym}
\usepackage{mathrsfs}
\usepackage{amsthm}
\usepackage{mathtools}
\usepackage{bm}
\usepackage{verbatim}
\usepackage{color}

\newtheorem{theorem}{\bf Theorem}[section]
\newtheorem{lemma}[theorem]{\bf Lemma}
\newtheorem{corollary}[theorem]{\bf Corollary}
\newtheorem{proposition}[theorem]{\bf Proposition}
\newtheorem{remark}[theorem]{\bf Remark}

\numberwithin{equation}{section}

\newcommand{\dt}{\partial_t}
\newcommand{\ds}{\partial_s}

\makeatletter
\newcommand{\opnorm}{\@ifstar\@opnorms\@opnorm}
\newcommand{\@opnorms}[1]{%
  \left|\mkern-1.5mu\left|\mkern-1.5mu\left|
   #1
  \right|\mkern-1.5mu\right|\mkern-1.5mu\right|
}
\newcommand{\@opnorm}[2][]{%
  \mathopen{#1|\mkern-1.5mu#1|\mkern-1.5mu#1|}
  #2
  \mathclose{#1|\mkern-1.5mu#1|\mkern-1.5mu#1|}
}

\begin{document}
\vspace*{0ex}
\begin{center}
{\Large\bf
A priori estimates for solutions to equations of motion \\[0.5ex]
of an inextensible hanging string
}
\end{center}

\begin{center}
Tatsuo Iguchi and Masahiro Takayama\footnote{Corresponding author}
\end{center}

\begin{abstract}
We consider the initial boundary value problem to equations of motion of an inextensible hanging string of finite length under the action of the gravity. 
We also consider the problem in the case without any external forces. 
In this problem, the tension of the string is also an unknown quantity. 
It is determined as a unique solution to a two-point boundary value problem, 
which is derived from the inextensibility of the string together with the equation of motion, and degenerates linearly at the free end. 
We derive a priori estimates for solutions to the initial boundary value problem in weighted Sobolev spaces under a natural stability condition. 
The necessity for the weights results from the degeneracy of the tension. 
Uniqueness of solutions is also proved. 
\end{abstract}

\section{Introduction}
We are concerned with the motion of a homogeneous and inextensible string of finite length $L$ under the action of the gravity and a tension of the string. 
Suppose that one end of the string is fixed and another one is free. 
Let $s$ be the arc length of the string measured from the free end of the string so that the string is descried as a curve 
\[
\bm{x}(s,t)=(x_1(s,t),x_2(s,t),x_3(s,t)), \qquad s\in [0,L]
\]
at time $t$. 
We can assume without loss of generality that the fixed end of the string is placed at the origin in $\mathbb{R}^3$. 
Let $\rho$ be a constant density of the string, $\bm{g}$ the acceleration of gravity vector, 
and $\tau(s,t)$ a scalar tension of the string at the point $\bm{x}(s,t)$ at time $t$. 
See Figure \ref{intro:hanging string}.
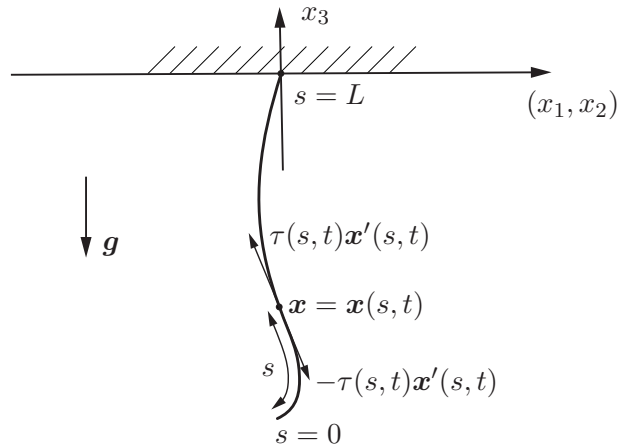
\begin{figure}[ht]
\setlength{\unitlength}{1pt}
\begin{picture}(0,0)
\put(320,-50){$(x_1,x_2)$}
\put(235,-15){$x_3$}
\put(232,-47){$s=L$}
\put(225,-175){$s=0$}
\put(220,-150){$s$}
\put(160,-102){$\bm{g}$}
\put(230,-126){$\bm{x}=\bm{x}(s,t)$}
\put(240,-155){$-\tau(s,t)\bm{x}'(s,t)$}
\put(223,-100){$\tau(s,t)\bm{x}'(s,t)$}
\end{picture}
\begin{center}
\includegraphics[width=0.45\linewidth]{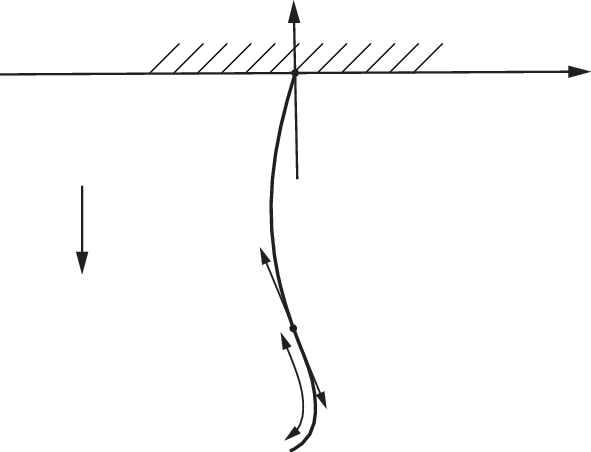}
\end{center}
\caption{Hanging String}
\label{intro:hanging string}
\end{figure}
Then, the motion of the string is described by the equations 
\[
\begin{cases}
 \rho\ddot{\bm{x}}-(\tau\bm{x}')'=\rho\bm{g} &\mbox{in}\quad (0,L)\times(0,T), \\
 |\bm{x}'|=1 &\mbox{in}\quad (0,L)\times(0,T),
\end{cases}
\]
where $\dot{\bm{x}}$ and $\bm{x}'$ denote the derivatives of $\bm{x}$ with respect to $t$ and $s$, respectively, 
so that $\bm{x}'$ is a unit tangential vector of the string. 
For a derivation of these equations, we refer, for example, to Reeken \cite{Reeken1977} and Yong \cite{Yong2006}. 
The first equation is the equation of motion. 
The peculiarity of this problem is that we do not assume any elasticity of the string so that the tension $\tau$ in the first equation 
is also an unknown quantity of the problem. 
In other words, we do not assume any constitutive equation for the tension $\tau$. 
However, we impose the second equation, which describes the fact that the string is inextensible. 
We also note that the tension is caused by the inextensibility of the string as we will see below. 
The boundary conditions at the both ends of the string are given by 
\[
\begin{cases}
 \bm{x}=\bm{0} &\mbox{on}\quad \{s=L\}\times(0,T), \\
 \tau=0 &\mbox{on}\quad \{s=0\}\times(0,T).
\end{cases}
\]
The first boundary condition represents that the one end of the string is fixed at the origin, and the second one represents that another end is free. 
In the case $\bm{g}\neq\bm{0}$, by making the change of the variables $s\to Ls$, $t\to\sqrt{L/g}t$, $\bm{x}\to L\bm{x}$, $\tau\to \rho gL\tau$, 
and $\bm{g}/g\to\bm{g}$ with $g=|\bm{g}|$, we may assume that $\rho=1$, $L=1$, and $|\bm{g}|=1$. 
Similarly, in the case $\bm{g}=\bm{0}$, by making the same change of the variables as above with any positive constant $g$, we may assume that $\rho=1$ and $L=1$. 
Therefore, in the following we consider the equations 
\begin{equation}\label{Eq}
\begin{cases}
 \ddot{\bm{x}}-(\tau\bm{x}')' = \bm{g} &\mbox{in}\quad (0,1)\times(0,T), \\
 |\bm{x}'|=1 &\mbox{in}\quad (0,1)\times(0,T),
\end{cases}
\end{equation}
under the boundary conditions 
\begin{equation}\label{BC}
\begin{cases}
 \bm{x}=\bm{0} &\mbox{on}\quad \{s=1\}\times(0,T), \\
 \tau=0 &\mbox{on}\quad \{s=0\}\times(0,T).
\end{cases}
\end{equation}
Here, $\bm{g}$ is a constant unit vector or the zero vector. 
Finally, we impose the initial conditions of the form 
\begin{equation}\label{IC}
(\bm{x},\dot{\bm{x}})|_{t=0}=(\bm{x}_0^\mathrm{in},\bm{x}_1^\mathrm{in}) \quad\mbox{in}\quad (0,1).
\end{equation}
This is the initial boundary value problem that we are going to consider in this paper. 
Here, we remark that the problem \eqref{Eq} and \eqref{BC} also arises in a minimization problem of the action function 
$J(\bm{x})=\int_0^T\!\!\int_0^1\bigl( \frac12|\dot{\bm{x}}(s,t)|^2+\bm{g}\cdot\bm{x}(s,t) \bigr)\mathrm{d}s\mathrm{d}t$ 
under the constraints $|\bm{x}(s,t)|\equiv1$ and $\bm{x}(1,t)=\bm{0}$. 
In this case, the tension $\tau$ appears as a Lagrangian multiplier. 
For more details on this variational principle, we refer, for example, to \c{S}eng\"{u}l and Vorotnikov \cite{SengulVorotnikov2017} and the references therein.

\medskip
As was explained above, the tension $\tau$ is also an unknown quantity. 
On the other hand, we assume that the string is inextensible so that we impose the constraint $|\bm{x}'|=1$, which causes a tension of the string. 
In other words, by using the constraint we can derive an equation for the tension $\tau$ as follows. 
Let $(\bm{x},\tau)$ be a solution to \eqref{Eq} and \eqref{BC}. 
Then, we see that $\tau$ satisfies the following two-point boundary value problem 
\begin{equation}\label{BVP}
\begin{cases}
 -\tau''+|\bm{x}''|^2\tau = |\dot{\bm{x}}'|^2 &\mbox{in}\quad (0,1)\times(0,T), \\
 \tau=0 &\mbox{on}\quad \{s=0\}\times(0,T), \\
 \tau'=-\bm{g}\cdot\bm{x}' &\mbox{on}\quad \{s=1\}\times(0,T),
\end{cases}
\end{equation}
where we regard the time $t$ as a parameter. 
This is a well-known fact and is easily verified; 
see, for example, Preston \cite[Section 2.1]{Preston2011} and \c{S}eng\"{u}l and Vorotnikov \cite[Section 2.4]{SengulVorotnikov2017}. 
In fact, by differentiating the constraint $|\bm{x}'|^2=1$ with respect to $s$ and $t$, we have 
$\bm{x}'\cdot\bm{x}''=0$, $\bm{x}'\cdot\bm{x}'''+|\bm{x}''|^2 = 0$, $\bm{x}'\cdot\dot{\bm{x}}'=0$, and $\bm{x}'\cdot\ddot{\bm{x}}'+|\dot{\bm{x}}'|^2 = 0$. 
Therefore, differentiating the first equation in \eqref{Eq} with respect to $s$ and then taking an inner product with $\bm{x}'$, 
we obtain the first equation in \eqref{BVP}. 
Taking an inner product of the first equation in \eqref{Eq} with $\bm{x}'$, taking its trace on $s=0$, 
and using the first boundary condition in \eqref{BC}, we obtain the last boundary condition in \eqref{BVP}.
It is easy to see that for each fixed time $t$, the two-point boundary value problem \eqref{BVP} can be solved uniquely, 
so that $\tau$ is determined by $\bm{x}'(\cdot,t)$ and $\dot{\bm{x}}'(\cdot,t)$. 
Unlike standard theories of nonlinear wave equaions, in our problem the tension $\tau$ depends nonlocally in space and time on $\bm{x}'$. 
Particularly, we need an information of the curvature vector $\bm{x}''(\cdot,t)$ and the deformation velocity $\dot{\bm{x}}'(\cdot,t)$ 
of the tangential vector of the string to determine the tension $\tau$.

\medskip
For the well-posedness of the initial boundary value problem, standard analysis on hyperbolic systems requires a positivity of the tension $\tau$. 
However, the positivity fails necessarily at the free end $s=0$ due to the boundary condition on $\tau$. 
Taking these into account, in place of assuming a strict positivity of $\tau$, we impose the following stability condition 
\begin{equation}\label{SolClass0_SC}
\frac{\tau(s,t)}{s}\geq c_0>0
\end{equation}
for $(s,t)\in(0,1)\times(0,T)$. 
If we consider a linearized problem around the rest state, then the corresponding stability condition is reduced to 
$-\bm{g}\cdot\bm{x}'(1,t) \geq c_0>0$ for $t\in(0,T)$. 
This last condition can be easily understood geometrically; see Figures \ref{fig:HS1} and \ref{fig:HS2}. 
\bigskip
\begin{figure}[ht]
\setlength{\unitlength}{1pt}
\begin{picture}(0,0)
\put(77,42){$\bm{g}$}
\put(107,22){$\bm{x}'(1,t)$}
\put(307,58){$\bm{g}$}
\put(337,12){$\bm{x}'(1,t)$}
\end{picture} 
 \begin{minipage}{0.48\hsize}
  \begin{center}
   \includegraphics[width=30mm]{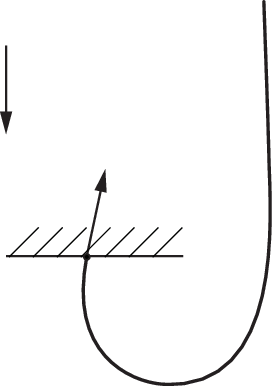}
  \end{center}
  \caption{The case $-\bm{g}\cdot\bm{x}'(1,t)>0$}
  \label{fig:HS1}
 \end{minipage}\quad
 \begin{minipage}{0.48\hsize}
  \begin{center}
   \includegraphics[width=30mm]{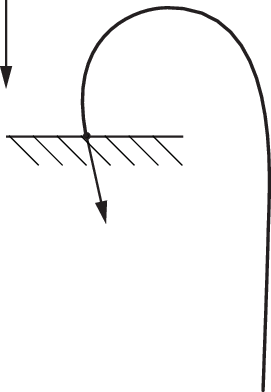}
  \end{center}
  \caption{The case $-\bm{g}\cdot\bm{x}'(1,t)<0$}
  \label{fig:HS2}
 \end{minipage}
\end{figure}
As we will see in Section \ref{sect:BVP}, under the condition $\int_0^1s|\bm{x}''(s,t)|^2\mathrm{d}s \lesssim 1$, the tension $\tau$ can be bounded from below as 
\begin{equation}\label{SC1}
\frac{\tau(s,t)}{s} \gtrsim -\bm{g}\cdot\bm{x}'(1,t) + \int_0^1s|\dot{\bm{x}}'(s,t)|^2\mathrm{d}s \exp\left( -\int_0^1s|\bm{x}''(s,t)|^2\mathrm{d}s \right),
\end{equation}
if the right-hand side is non-negative. 
This reveals a nonlinear stabilizing effect of the problem, and moreover, ensures the stability condition even in the case 
$\bm{g}=\bm{0}$ if $\dot{\bm{x}}(s,t)\not\equiv\bm{0}$. 
Our main objective is to show the local well-posedness of the initial boundary value problem \eqref{Eq}--\eqref{IC} in appropriate weighted Sobolev spaces 
under the stability condition \eqref{SolClass0_SC}. 
Toward this goal, in this paper we will derive a priori estimates for the solution $(\bm{x},\tau)$ to the problem, 
which are given in Theorem \ref{th:APE}. 
We also prove a uniqueness of solutions to the problem in the class where a priori estimates would be obtained. 
The uniqueness is given in Theorems \ref{th:unique_g<>0} and \ref{th:unique_g==0}.

\medskip
Even if a priori estimates for the solution $(\bm{x},\tau)$ would be obtained, it is not straightforward to construct an existence theory 
and we need more technical calculations than what will be done in this paper. 
Therefore, we postpone this existence part in our future work. 
Here, we just give a brief comment relative to the existence of a solution: 
In the derivation of the two-point boundary value problem \eqref{BVP}, we use essentially the constraint $|\bm{x}'|=1$ 
so that it is natural to expect that \eqref{BVP} contains an information of the constraint. 
In view of this, we will consider the initial boundary value problem to hyperbolic equations 
\begin{equation}\label{HP}
\begin{cases}
 \ddot{\bm{x}}-(\tau\bm{x}')' = \bm{g} &\mbox{in}\quad (0,1)\times(0,T), \\
 \bm{x}=\bm{0} &\mbox{on}\quad \{s=1\}\times(0,T),
\end{cases}
\end{equation}
for $\bm{x}$ under the initial condition \eqref{IC}, coupled with the two-point boundary value problem \eqref{BVP} for $\tau$, 
in place of the problem \eqref{Eq}--\eqref{IC}. 
One may think that a boundary condition on the free end $s=0$ is missing for the well-posedness of the problem for $\bm{x}$. 
However, it is not the case because the tension is degenerate at $s=0$. 
For more details, see Takayama \cite{Takayama2018}. 
We will use the initial boundary value problem to the hyperbolic and elliptic coupled system \eqref{HP}, \eqref{BVP}, and \eqref{IC} 
to construct the solution $(\bm{x},\tau)$. 
In order to show the equivalence of the problems, we need to show that the solution $(\bm{x},\tau)$ to the transformed system 
\eqref{HP}, \eqref{BVP}, and \eqref{IC} satisfies the constraint $|\bm{x}'|=1$ under appropriate conditions on the initial data. 
Here, we note that if there exists a smooth solution to the original problem \eqref{Eq}--\eqref{IC}, 
then the initial data have to satisfy the constraints 
$|\bm{x}_0^{\mathrm{in}\,\prime}|=1$ and $\bm{x}_0^{\mathrm{in}\,\prime}\cdot\bm{x}_1^{\mathrm{in}\,\prime}=0$ in $(0,1)$. 
Conversely, we will show in Theorem \ref{th:equiv_0} that if the initial data satisfy these constrains, then any regular solution 
$(\bm{x},\tau)$ to the transformed system \eqref{HP}, \eqref{BVP}, and \eqref{IC} satisfying the stability condition \eqref{SolClass0_SC} 
satisfies the constraint $|\bm{x}'|=1$. 
This will be carried out by using an energy estimate.

\medskip
Contrary to the studies on elastic strings, 
there are few results on the well-posedness of the initial boundary value problem \eqref{Eq}--\eqref{IC} to the motion of an inextensible straing. 
Reeken \cite{Reeken1979-1, Reeken1979-2} considered the motion of an inextensible string of \textit{infinite} length 
having one end fixed at the point $(0,0,\infty)$ in a gravity field. 
For technical reasons he assumed that the acceleration of gravity vector $\bm{g}$ is not constant. 
To be precise, he assumed that $\bm{g}=\bm{g}(s)\in C^\infty([0,\infty))$ is constant for $s\in[0,l]$ and grows linearly beyond $s=l$ for some positive $l$. 
Under this non-physical condition, he proved the existence locally in time and uniqueness of the solution 
provided that the initial data are sufficiently close to a trivial stationary solution in some weighted Sobolev spaces. 
The method that he used to solve the original problem \eqref{Eq}--\eqref{IC} is quite different from solving 
the transformed problem \eqref{HP}, \eqref{BVP}, and \eqref{IC}. 
He applied the hard implicit function theorem, which is also known as the Nash--Moser theorem, to construct the solution so that 
higher regularity must be imposed on the initial data and that a loss of derivatives was allowed. 
Preston \cite{Preston2011} considered the motion of an inextensible string of finite length in the case without any external forces, 
that is, in the case $\bm{g}=\bm{0}$. 
Under this particular situation, he proved the existence locally in time and uniqueness of the solution for arbitrary initial data in some weighted Sobolev spaces. 
Although the weighted Sobolev spaces used by Preston \cite{Preston2011} may seem to be different from those used by Reeken \cite{Reeken1979-1, Reeken1979-2}, 
their norms are equivalent so that their weighted Sobolev spaces are identical. 
In order to solve the original problem \eqref{Eq}--\eqref{IC}, he used the transformed problem \eqref{HP}, \eqref{BVP}, and \eqref{IC}. 
To be precise, in oreder to construct a solution he introduced a discretized problem with respect to $s$ and used uniform estimates for discrete solutions. 
Moreover, since the constraint $|\bm{x}'|=1$ could not be achieved if we use the discretization method described above, 
he guaranteed it by using the spherical coordinate such as $\bm{x}'(s,t)=(\cos\theta(s,t),\sin\theta(s,t))$ in the two-dimensional case. 
\c{S}eng\"{u}l and Vorotnikov \cite{SengulVorotnikov2017} considered exactly the same initial boundary value problem 
\eqref{Eq}--\eqref{IC} as ours and proved the existence of an admissible Young measure solution after transforming the problem 
into a system of conservation laws with a discontinuous flux. 
We note that the existence of such a generalized Young measure solution does not imply the classical well-posedness of the problem. 
To our knowledge, these are only results on the existence of a solution to the initial boundary value problem \eqref{Eq}--\eqref{IC}, 
so that its well-posedness has not been resolved so far. 
We aim to show the well-posedness of the problem \eqref{Eq}--\eqref{IC} 
in the weighted Sobolev spaces used by Reeken \cite{Reeken1979-1, Reeken1979-2} and Preston \cite{Preston2011}.

\medskip
As related topics on the motion of an inextensible string, Preston \cite{Preston2012} studied the geodesics on an infinite-dimensional manifold 
of inextensible curves in the $L^2$-metric and proved that the geodesics are determined by \eqref{Eq} and \eqref{BC} with $\bm{g}=\bm{0}$. 
Similarly, Preston and Saxton \cite{PrestonSaxton2013} studied the geodesic on this manifold in the $H^1$-metric 
and Shi and Vorotnikov \cite{ShiVorotnikov2019} studied the gradient flow of a potential energy on this manifold in the $L^2$-metric. 
Moreover, there are several results on the rotations of an inextensible hanging string about a vertical axis with one free end under the action of the gravity. 
We can observe stable configurations, in which its shape is not changing with time, when we force to rotate the string from the upper fixed end. 
These configurations are related to the angular velocity of the rotation. 
A representative result on this problem was given by Kolodner \cite{Kolodner1955}, 
who proved that the corresponding nonlinear eigenvalue problem with a constant angular velocity $\omega$ has exactly $n$ non-trivial solutions 
if and only if $\omega$ satisfies $\omega_n<\omega\leq\omega_{n+1}$ with $\omega_n\equiv\sigma_n\sqrt{|\bm{g}|/4L}$, 
where $\sigma_n$ is the $n$-th zero of the Bessel function $J_0(z)$, $\bm{g}$ is the acceleration of gravity vector, 
and $L$ is the length of the string. 
For more results on the rotating string, see references in Amore, Boyd, and M\'{a}rquez \cite{AmoreBoydMarquez2023}. 
The study of the motion of an inextensible string has applications: see Grothaus and Marheineke \cite{GrothausMarheineke2016} to textile industry; 
and Connell and Yue \cite{ConnellYue2007}, Lee, Huang, and Sung \cite{LeeHuangSung2014}, 
and Ryu, Park, Kim, and Sung \cite{RyuParkKimSung2015} to flapping dynamics of a flag.

\medskip
As ending this introduction, we explain that the weights in the norm that we will use in this paper arise naturally from the standard theory of hyperbolic systems. 
In the case where $\bm{g}$ is a unit constant vector, the problem \eqref{Eq} and \eqref{BC} has a trivial stationary solution 
$(\bm{x}_\mathrm{s}(s),\tau_\mathrm{s}(s))=((1-s)\bm{g},s)$. 
Linearizing \eqref{Eq} and \eqref{BC} around this stationary solution and picking up only the highest order terms, we obtain linear equations 
$\ddot{\bm{x}}=(s\bm{x}')'$ with the boundary condition $\bm{x}|_{s=1}=\bm{0}$. 
By introducing a new quantity $\bm{x}^\sharp(x_1,x_2,t)=\bm{x}(x_1^2+x_2^2,t)$, the linearized problem is transformed equivalently into 
\[
\begin{cases}
 \dt^2\bm{x}^\sharp=\frac14\Delta\bm{x}^\sharp &\mbox{in}\quad D\times(0,T), \\
 \bm{x}^\sharp=\bm{0} &\mbox{on}\quad \partial D\times(0,T),
\end{cases}
\]
where $\Delta$ is the two-dimensional Laplacian and $D$ is the unit disc in $\mathbb{R}^2$ with a center at the origin. 
It is well-known that the corresponding initial boundary value problem is well-posed in the class $\bigcap_{j=0}^mC^j([0,T];H^{m-j}(D))$, 
where $H^m(D)$ is the standard $L^2$ Sobolev space of order $m$ on $D$. 
By calculating the norm $\|\bm{x}^\sharp(t)\|_{H^m}$ in terms of $\bm{x}(\cdot,t)$, we are naturally led to a weighted Sobolev space $X^m$ 
with a norm $\|\cdot\|_{X^m}$ so that $\|\bm{x}^\sharp(t)\|_{H^m} \simeq\|\bm{x}(t)\|_{X^m}$ holds. 
This space $X^m$ is exactly the same as the space used by Reeken \cite{Reeken1979-1, Reeken1979-2} and that by Preston \cite{Preston2011}. 
In this paper, we evaluate the solution $\bm{x}$ of \eqref{Eq} in this weighted Sobolev space by using the energy method. 
This requires to evaluate the solution $\tau$ of the two-point boundary value problem \eqref{BVP} also in a weighted Sobolev space. 
To this end, we express the solution $\tau$ by using Green's function of the problem and evaluate it through precise pointwise estimates of Green's function.

\medskip
The contents of this paper are as follows.  
In Section \ref{sect:results} we begin with introducing a weighted Sobolev space $X^m$, which plays an important role in the problem, 
and then state our main results in this paper: 
a priori estimates for solutions in Theorem \ref{th:APE}, uniqueness of solutions in Theorems \ref{th:unique_g<>0} and \ref{th:unique_g==0}, 
and the equivalence of the original problem \eqref{Eq}--\eqref{IC} and the transformed problem \eqref{HP}, \eqref{BVP}, and \eqref{IC} in Theorem \ref{th:equiv_0}. 
In Section \ref{sect:BVP} we analyze Green's function related to the two-point boundary value problem \eqref{BVP} 
to derive precise pointwise estimates for the solution in terms of norms of the weighted Sobolev space $X^m$ for the coefficients. 
In Section \ref{sect:FS} we present basic properties of the space $X^m$ and related calculus inequalities. 
We also introduce another weighted Sobolev space $Y^m$ for convenience in estimating higher order derivatives of the solution 
to the two-point boundary value problem \eqref{BVP}. 
In Section \ref{sect:equiv} we prove Theorem \ref{th:equiv_0}. 
In Section \ref{sect:EE} we analyze a linearized system to the problem \eqref{Eq}, \eqref{BC}, and \eqref{BVP}, and derive an energy estimate for the solution. 
In Section \ref{sect:unique} we prove Theorems \ref{th:unique_g<>0} and \ref{th:unique_g==0} by applying the energy estimate obtained in Section \ref{sect:EE} 
with a slight modification. 
In Section \ref{sect:ETau} we derive estimates for the tension $\tau$ assuming some bounds of norms for $\bm{x}$ and the stability condition \eqref{SolClass0_SC}. 
In the critical case of the regularity index, we use a weaker norm for $\tau$ than those in the other cases. 
In Section \ref{sect:EstID} we evaluate initial values for time derivatives of $\bm{x}$ and $\tau$ 
in terms of the initial data $(\bm{x}_0^\mathrm{in},\bm{x}_1^\mathrm{in})$. 
Finally, in Section \ref{sect:APE} we prove Theorem \ref{th:APE}.

\medskip
\noindent
{\bf Notation}. \ 
For $1\leq p\leq\infty$, we denote by $L^p$ the Lebesgue space on the open interval $(0,1)$. 
For non-negative integer $m$, we denote by $H^m$ the $L^2$ Sobolev space of order $m$ on $(0,1)$. 
The norm of a Banach space $B$ is denoted by $\|\cdot\|_B$. 
The inner product in $L^2$ is denoted by $(\cdot,\cdot)_{L^2}$. 
We put $\dt=\frac{\partial}{\partial t}$ and $\ds=\frac{\partial}{\partial s}$. 
The norm of a weighted $L^p$ space with a weight $s^\alpha$ is denoted by $\|s^\alpha u\|_{L^p}$, so that 
$\|s^\alpha u\|_{L^p}^p=\int_0^1s^{\alpha p}|u(s)|^p \mathrm{d}s$ for $1\leq p<\infty$. 
It is sometimes denoted by $\|\sigma^\alpha u\|_{L^p}$, too. 
This would cause no confusion. 
$[P,Q]=PQ-QP$ denotes the commutator. 
We denote by $C(a_1, a_2, \ldots)$ a positive constant depending on $a_1, a_2, \ldots$. 
$f\lesssim g$ means that there exists a non-essential positive constant C such that $f\leq Cg$ holds. 
$f\simeq g$ means that $f\lesssim g$ and $g\lesssim f$ hold. 
$a_1 \vee a_2 = \max\{a_1,a_2\}$.

\medskip
\noindent
{\bf Acknowledgement} \\
T. I. is partially supported by JSPS KAKENHI Grant Number JP22H01133.

\section{Main results}\label{sect:results}
In order to state our main results, we first introduce function spaces that we are going to use in this paper. 
For a non-negative integer $m$ we define a weighted Sobolev space $X^m$ as a set of all function $u=u(s)\in L^2$ equipped with a norm 
$\|\cdot\|_{X^m}$ defined by 
\begin{equation}\label{WSS}
\|u\|_{X^m}^2 =
\begin{cases}
 \displaystyle
 \|u\|_{H^k}^2 + \sum_{j=1}^k\|s^j\ds^{k+j}u\|_{L^2}^2 &\mbox{for}\quad m=2k, \\
 \displaystyle
 \|u\|_{H^k}^2 + \sum_{j=1}^{k+1}\|s^{j-\frac12}\ds^{k+j}u\|_{L^2}^2 &\mbox{for}\quad m=2k+1.
\end{cases}
\end{equation}
For a function $u=u(s,t)$ depending also on time $t$, we introduce norms 
$\opnorm{\cdot}_m$ and $\opnorm{\cdot}_{m,*}$ by 
\[
\opnorm{u(t)}_m^2 = \sum_{j=0}^m \|\dt^j u(t)\|_{X^{m-j}}^2, \qquad
\opnorm{u(t)}_{m,*}^2 = \sum_{j=0}^{m-1} \|\dt^j u(t)\|_{X^{m-j}}^2. 
\]
The first norm $\opnorm{\cdot}_m$ will be used to evaluate $\bm{x}$, 
whereas the second norm $\opnorm{\cdot}_{m,*}$ will be used to evaluate $\tau$. 
However, in the critical case on the regularity index $m$, 
we need to use a weaker norm than $\opnorm{\cdot}_{m,*}$. 
For $\epsilon>0$, we introduce norms $\|\cdot\|_{X_\epsilon^k}$ for $k=1,2,3$ as 
\[
\|u\|_{X_\epsilon^k}^2 =
\begin{cases}
 \|s^\epsilon u\|_{L^\infty}^2 + \|s^{\frac12+\epsilon}u'\|_{L^2}^2 &\mbox{for}\quad k=1, \\
 \|u\|_{L^\infty}^2 + \|s^\epsilon u'\|_{L^2}^2 + \|s^{1+\epsilon} u''\|_{L^2}^2 &\mbox{for}\quad k=2, \\
 \|u\|_{L^\infty}^2 + \|s^\epsilon u'\|_{L^\infty}^2 + \|s^{\frac12+\epsilon}u''\|_{L^2}^2 + \|s^{\frac32+\epsilon}u'''\|_{L^2}^2
  &\mbox{for}\quad k=3,
\end{cases}
\]
and put 
\[
\opnorm{u(t)}_{3,*,\epsilon}^2 = \|u(t)\|_{X_\epsilon^3}^2 + \|\dt u(t)\|_{X_\epsilon^2}^2 + \|\dt^2 u(t)\|_{X_\epsilon^1}^2. 
\]
The following theorem is one of main theorems in this paper and gives a priori estimates for the solution $(\bm{x},\tau)$ to the problem \eqref{Eq}--\eqref{IC}.

\begin{theorem}\label{th:APE}
For any integer $m\geq4$ and any positive constants $M_0$ and $c_0$, 
there exist a sufficiently small positive time $T$ and a large constant $C$ such that if the initial data satisfy 
\begin{equation}\label{CondID}
\begin{cases}
 \|\bm{x}_0^\mathrm{in}\|_{X^m}+\|\bm{x}_1^\mathrm{in}\|_{X^{m-1}} \leq M_0, \\
 \frac{\tau_0^\mathrm{in}(s)}{s} \geq 2c_0 \quad\mbox{for}\quad 0<s<1, 
\end{cases}
\end{equation}
where $\tau_0^\mathrm{in}(s)=\tau(s,0)$ is the initial tension, 
then any regular solution $(\bm{x},\tau)$ to the initial boundary value problem \eqref{Eq}--\eqref{IC} satisfies the stability condition \eqref{SolClass0_SC}, 
$\opnorm{\bm{x}(t)}_m \leq C$, and 
\[
\begin{cases}
 C^{-1}s \leq \tau(s,t) \leq Cs, \quad \sum_{j=1}^{m-3}|\dt^j\tau(s,t)|\leq Cs, \quad |\dt^{m-2}\tau(s,t)|\leq Cs^\frac12, \\
 \opnorm{\tau'(t)}_{m-1,*} \leq C \qquad\ \mbox{in the case $m\geq5$}, \\
 \opnorm{\tau'(t)}_{3,*,\epsilon} \leq C(\epsilon) \qquad\mbox{in the case $m=4$}
\end{cases}
\]
for $0\leq t\leq T$ and $\epsilon>0$, where the constant $C(\epsilon)$ depends also on $\epsilon$. 
\end{theorem}

\begin{remark}\label{remark:CI}
\begin{enumerate}
\item[\rm(1)]
Since $\tau(\cdot,t)$ is uniquely determined from $(\bm{x}(\cdot,t),\dot{\bm{x}}(\cdot,t))$ as a solution of the two-point boundary value problem \eqref{BVP}, 
the initial tension $\tau_0^\mathrm{in}$ is also uniquely determined from the initial data $(\bm{x}_0^\mathrm{in},\bm{x}_1^\mathrm{in})$. 
Moreover, by Lemma \ref{lem:EstSolBVP1}, under the condition $\|\bm{x}_0^\mathrm{in}\|_{X^m}\leq M_0$, we have 
\[
\frac{\tau_0^\mathrm{in}(s)}{s} \gtrsim -\bm{g}\cdot\bm{x}_0^{\mathrm{in}\prime}(1)
 + \int_0^1s|\bm{x}_1^{\mathrm{in}\prime}(s)|^2\mathrm{d}s \exp\left( -\int_0^1s|\bm{x}_0^{\mathrm{in}\prime\prime}(s)|^2\mathrm{d}s \right).
\]
Therefore, if the initial string is in fact hanging from the fixed end $s=1$, that is, if $-\bm{g}\cdot\bm{x}_0^{\mathrm{in}\prime}(1)>0$, 
then the second condition in \eqref{CondID} is satisfied. 
Moreover, even in the case $\bm{g}=\bm{0}$, if the initial deformation velocity $\bm{x}_1^\mathrm{in}$ is not identically zero, 
then the second condition in \eqref{CondID} is satisfied, too. 
\item[\rm(2)]
Lemma \ref{lem:EstSolBVP1} also implies that if $\bm{x}_1^\mathrm{in}=\bm{0}$, 
then we have $\frac{\tau_0^\mathrm{in}(s)}{s} \simeq -\bm{g}\cdot\bm{x}_0^{\mathrm{in}\prime}(1)$. 
Therefore, if, in addition, $-\bm{g}\cdot\bm{x}_0^{\mathrm{in}\prime}(1)<0$, 
then the initial tension $\tau_0^\mathrm{in}$ is negative everywhere except at the free end $s=0$, 
so that the equation of motion in \eqref{Eq}  becomes elliptic in space and time. 
As a result, the initial boundary value problem becomes ill-posed. 
\item[\rm(3)]
The requirement $m\geq4$ corresponds to the quasilinear regularity in the sense that $m=4$ is the minimal integer regularity index $m$ 
that ensures the embedding $C^0([0,T];X^m)\cap C^1([0,T];X^{m-1}) \hookrightarrow C^1([0,1]\times[0,T])$; see Remark \ref{remark:embedding}. 
Therefore, $m=4$ is a critical regularity index in the classical sense. 
\end{enumerate}
\end{remark}

We then consider the uniqueness of the solution $(\bm{x},\tau)$ to the initial boundary value problem \eqref{Eq}--\eqref{IC}. 
To this end, we need to specify a class that the solutions belong to. 
Here, we consider the solutions satisfying 
\begin{equation}\label{SolClass0}
\bm{x}' \in L^\infty(0,T;X^2)\cap W^{1,\infty}(0,T;X^1). 
\end{equation}
We note that if $\bm{x} \in L^\infty(0,T;X^4)\cap W^{1,\infty}(0,T;X^3)$, 
then it also satisfies \eqref{SolClass0}; see Remark \ref{remark:embedding}. 
As we will see in Lemma \ref{lem:EstSol3}, under the conditions \eqref{SolClass0}, the solutions satisfy also $\bm{x}'\in W^{2,\infty}(0,T;L^2)$. 
In view of these and the boundary condition $\bm{x}|_{s=1}=\bm{0}$, 
we may assume without loss of generality that $\bm{x},\bm{x}' \in C^0([0,T];X^1) \cap C^1([0,T];L^2)$. 
Therefore, the initial conditions \eqref{IC} can be understood in the classical sense.

\begin{theorem}\label{th:unique_g<>0}
The solution to the initial boundary value problem \eqref{Eq}--\eqref{IC} is unique in the class \eqref{SolClass0} 
satisfying the stability condition \eqref{SolClass0_SC}. 
\end{theorem}

In the case $\bm{g}=\bm{0}$, if the initial deformation velocity $\bm{x}_1^\mathrm{in}$ is identically zero, 
then the initial boundary value problem \eqref{Eq}--\eqref{IC} has a trivial solution $(\bm{x}(s,t),\tau(s,t))=(\bm{x}_0^\mathrm{in}(s),0)$. 
Since this solution does not satisfy the stability condition \eqref{SolClass0_SC}, 
we cannot apply directly Theorem \ref{th:unique_g<>0} to ensure the uniqueness of solutions in this case. 
Nevertheless, by Lemma \ref{lem:SpecialCase} we see that this trivial solution is the only one that does not satisfy the stability condition \eqref{SolClass0_SC} 
in th case $\bm{g}=\bm{0}$. 
As a result, we have the following uniqeness theorem without assuming a priori the stability condition.

\begin{theorem}\label{th:unique_g==0}
In the case $\bm{g}=\bm{0}$, the solution to the initial boundary value problem \eqref{Eq}--\eqref{IC} is unique in the class \eqref{SolClass0}. 
\end{theorem}

The following theorem ensures the equivalence of the original problem \eqref{Eq}--\eqref{IC} and the transformed problem 
\eqref{HP}, \eqref{BVP}, and \eqref{IC}.

\begin{theorem}\label{th:equiv_0}
Let $(\bm{x},\tau)$ be a solution to the transformed problem \eqref{HP}, \eqref{BVP}, and \eqref{IC} in the class \eqref{SolClass0} satisfying 
the stability condition \eqref{SolClass0_SC}. 
Suppose that the initial data satisfy $|\bm{x}_0^{\mathrm{in}\,\prime}(s)|\equiv1$ and 
$\bm{x}_0^{\mathrm{in}\,\prime}(s)\cdot\bm{x}_1^{\mathrm{in}\,\prime}(s)\equiv0$. 
Then, we have $|\bm{x}'(s,t)|\equiv1$. 
\end{theorem}

We prove Theorems \ref{th:APE}, \ref{th:unique_g<>0}, \ref{th:unique_g==0}, and \ref{th:equiv_0} 
in Sections \ref{sect:APE}, \ref{sect:unique}, \ref{sect:unique}, and \ref{sect:equiv}, respectively.

\section{Two-point boundary value problem}\label{sect:BVP}
In view of \eqref{BVP} we will consider the two-point boundary value problem 
\begin{equation}\label{TBVP}
\begin{cases}
 -\tau''+|\bm{x}''|^2\tau = h \quad\mbox{in}\quad (0,1), \\
 \tau(0)=0, \quad \tau'(1)=a,
\end{cases}
\end{equation}
where $\bm{x}(s)$ and $h(s)$ are given functions and $a$ is a constant.

\subsection{Green's function}
As is well-known, Green's function to the boundary value problem \eqref{TBVP} can be constructed as follows. 
Let $\varphi$ and $\psi$ be unique solutions to the initial value problems 
\begin{equation}\label{IVPphi}
\begin{cases}
 -\varphi''+|\bm{x}''|^2\varphi = 0 \quad\mbox{in}\quad (0,1), \\
 \varphi(0)=0, \quad \varphi'(0)=1,
\end{cases}
\end{equation}
and 
\begin{equation}\label{IVPpsi}
\begin{cases}
 -\psi''+|\bm{x}''|^2\psi = 0 \quad\mbox{in}\quad (0,1), \\
 \psi(1)=1, \quad \psi'(1)=0,
\end{cases}
\end{equation}
respectively. 
The Wronskian $W(s;\varphi,\psi)=\varphi(s)\psi'(s)-\varphi'(s)\psi(s)$ is a non-zero constant 
since the uniqueness of solutions to the boundary value problem \eqref{TBVP} is easily verified. 
Particularly, we have $W(s;\varphi,\psi)\equiv-\varphi'(1)$. 
A sharp estimate for $\varphi'(1)$ will be given below; see Lemma \ref{lem:EstPhi}. 
In terms of these fundamental solutions, Green's function to the boundary value problem \eqref{TBVP} is given by 
\begin{equation}\label{GF}
G(s,r)=
\begin{cases}
 \dfrac{\varphi(s)\psi(r)}{\varphi'(1)} &\mbox{for}\quad 0\leq s\leq r, \\[2ex]
 \dfrac{\psi(s)\varphi(r)}{\varphi'(1)} &\mbox{for}\quad r\leq s\leq 1.
\end{cases}
\end{equation}
Particularly, the unique solution to the boundary value problem \eqref{TBVP} can be expressed as 
\begin{equation}\label{SolBVP}
\tau(s)=a\frac{\varphi(s)}{\varphi'(1)}
 +\frac{\psi(s)}{\varphi'(1)}\int_0^s\varphi(\sigma)h(\sigma)\mathrm{d}\sigma
 +\frac{\varphi(s)}{\varphi'(1)}\int_s^1\psi(\sigma)h(\sigma)\mathrm{d}\sigma.
\end{equation}
We proceed to evaluate these fundamental solutions $\varphi$ and $\psi$.

\begin{lemma}\label{lem:EstPhi}
Let $\varphi$ be a unique solution to \eqref{IVPphi}. 
Then, for any $s\in[0,1]$ we have 
\[
\begin{cases}
 1 \leq \varphi'(s) \leq \exp(\|\sigma^{\frac12}\bm{x}''\|_{L^2}^2), \\
 s \leq \varphi(s) \leq s\exp(\|\sigma^{\frac12}\bm{x}''\|_{L^2}^2).
\end{cases}
\]
\end{lemma}

\begin{proof}
It is sufficient to show the first estimate because the second one can easily follow from the first one by integrating it over $[0,s]$ 
and by using the initial condition $\varphi(0)=0$.

We first show that $\varphi(s)>0$ for all $s\in(0,1]$. 
In view of the initial conditions at $s=0$, we have $\varphi(s)>0$ for $0<s\ll1$. 
Now, suppose that there exists $s_*\in(0,1]$ such that $\varphi(s_*)=0$. 
We can assume without loss of generality that $\varphi(s)>0$ for $0<s<s_*$, so that $\varphi'(s_*)\leq0$. 
Then, we have $\varphi''(s)=|\bm{x}''(s)|^2\varphi(s) \geq 0$ for $0<s<s_*$. 
This implies that $\varphi'(s)$ is non-decreasing in the interval $[0,s_*]$, so that $\varphi'(s_*)\geq\varphi'(0)=1$. 
This contradicts with $\varphi'(s_*)\leq0$. 
Therefore, $\varphi(s)>0$ holds for all $s\in(0,1]$. 
Particularly, $\varphi'(s)$ is non-decreasing in the whole interval $[0,1]$, so that we obtain $\varphi'(s)\geq\varphi'(0)=1$ for all $s\in[0,1]$.

We proceed to show the upper bound of $\varphi'(s)$. 
Since $\varphi'(s)$ is a non-decreasing function, we have $\varphi(s)=\int_0^s\varphi'(\sigma)\mathrm{d}\sigma \leq s\varphi'(s)$. 
Therefore, we see that 
\begin{align*}
\varphi'(s)
&= 1+\int_0^s\varphi''(\sigma) \mathrm{d}\sigma \\
&= 1+\int_0^s|\bm{x}''(\sigma)|^2\varphi(\sigma) \mathrm{d}\sigma \\
&\leq 1+\int_0^s \sigma|\bm{x}''(\sigma)|^2\varphi'(\sigma) \mathrm{d}\sigma,
\end{align*}
which together with Gronwall's inequality yields 
$\varphi'(s) \leq \exp(\int_0^s\sigma|\bm{x}''(\sigma)|^2 \mathrm{d}\sigma)$. 
This gives the desired estimate. 
\end{proof}

\begin{lemma}\label{lem:EstPsi}
Let $\psi$ be a unique solution to \eqref{IVPpsi}. 
Then, for any $s\in[0,1]$ and any $\alpha\geq0$ we have 
\[
\begin{cases}
 1 \leq \psi(s) \leq \exp(\|\sigma^{\frac12}\bm{x}''\|_{L^2}^2), \\
 0 \geq s^\alpha\psi'(s) \geq -\|\sigma^{\frac{\alpha}{2}}\bm{x}''\|_{L^2}^2\exp(\|\sigma^{\frac12}\bm{x}''\|_{L^2}^2).
\end{cases}
\]
\end{lemma}

\begin{proof}
We first show that $\psi(s)>0$ for all $s\in[0,1]$. 
In view of the initial condition $\psi(1)=1$, we have $\psi(s)>0$ for $0\leq 1-s\ll1$. 
Now, suppose that there exists $s_*\in[0,1)$ such that $\psi(s_*)=0$. 
We can assume without loss of generality that $\psi(s)>0$ for $s_*<s\leq1$, so that $\psi'(s_*)\geq0$. 
Then, we have $\psi''(s)=|\bm{x}''(s)|^2\psi(s) \geq 0$ for $s_*<s\leq1$. 
This implies that $\psi'(s)$ is non-decreasing in the interval $[s_*,1]$, so that $\psi'(s)\leq\psi'(1)=0$ for all $s\in[s_*,1]$. 
This implies that $\psi(s)$ is non-increasing in the interval $[s_*,1]$, so that $\psi(s_*)\geq\psi(1)=1$. 
This contradicts with $\psi(s_*)=0$. 
Therefore, $\psi(s)>0$ holds for all $s\in[0,1]$. 
Particularly, $\psi''(s)\geq0$ holds for all $s\in[0,1]$, which implies in turn that $\psi'(s)\leq0$ and $\psi(s)\geq1$ for all $s\in[0,1]$.

We then show the upper bound of $\psi(s)$. 
Noting that $\psi(s)$ is a non-increasing function and that $\psi'(1)=0$, we see that 
\begin{align}\label{PfEstPsi}
\psi'(s) 
&= -\int_s^1\psi''(\sigma) \mathrm{d}\sigma \\
&= -\int_s^1|\bm{x}''(\sigma)|^2\psi(\sigma) \mathrm{d}\sigma \nonumber \\
&\geq -\int_s^1|\bm{x}''(\sigma)|^2 \mathrm{d}\sigma \psi(s), \nonumber
\end{align}
which together with Gronwall's inequality and $\psi(1)=1$ yields 
\begin{align*}
\psi(s) 
&\leq \exp\left(\int_s^1\int_\sigma^1|\bm{x}''(\tilde{\sigma})|^2 \mathrm{d}\tilde{\sigma}\mathrm{d}\sigma\right) \\
&= \exp\left( \int_s^1(\sigma-s)|\bm{x}''(\sigma)|^2 \mathrm{d}\sigma\right) \\
&\leq \exp\left( \int_0^1\sigma|\bm{x}''(\sigma)|^2 \mathrm{d}\sigma\right).
\end{align*}
This shows the desired upper bound.

We finally show the lower bound of $\psi'(s)$. 
It follows from \eqref{PfEstPsi} that 
\[
s^\alpha \psi'(s) \geq -\int_s^1\sigma^\alpha|\bm{x}''(\sigma)|^2 \mathrm{d}\sigma \psi(s),
\]
which together with the upper bound of $\psi(s)$ gives the desired one. 
\end{proof}

When the function $\bm{x}$ depends also on the time $t$, the fundamental solution $\varphi$ depends on the time $t$, too. 
In the following lemma, we will give estimates for the time derivative $\dot{\varphi}$.

\begin{lemma}\label{lem:EstDtPhi}
Let $\varphi$ be a unique solution to \eqref{IVPphi}. 
Then, for any $(s,t)\in[0,1]\times[0,T]$ we have 
\[
\begin{cases}
 |\dot{\varphi}'(s,t)| \leq \|\sigma^{\frac12}\dot{\bm{x}}''(t)\|_{L^2}\|\sigma^{\frac12}\bm{x}''(t)\|_{L^2}
  \exp(2\|\sigma^{\frac12}\bm{x}''(t)\|_{L^2}^2), \\
 |\dot{\varphi}(s,t)| \leq s\|\sigma^{\frac12}\dot{\bm{x}}''(t)\|_{L^2}\|\sigma^{\frac12}\bm{x}''(t)\|_{L^2}
  \exp(2\|\sigma^{\frac12}\bm{x}''(t)\|_{L^2}^2).
\end{cases}
\]
\end{lemma}

\begin{proof}
It is sufficient to show the first estimate because the second one can easily follow from the first one by integrating it over $[0,s]$ 
and by using the initial condition $\dot{\varphi}(0,t)=0$. 
We note that $\dot{\varphi}$ is a solution to the initial value problem 
\[
\begin{cases}
 -\dot{\varphi}''+|\bm{x}''|^2\dot{\varphi} = h \quad\mbox{in}\quad (0,1), \\
 \dot{\varphi}(0)=\dot{\varphi}'(0)=0,
\end{cases}
\]
where $h=-2(\dot{\bm{x}}''\cdot\bm{x}'')\varphi$ and we regard the time $t$ as a parameter, 
so that we omit $t$ from the notation in the following of this proof. 
Integrating the equation over $[0,s]$ and using the initial condition $\dot{\varphi}(0)=0$, we have 
\[
\dot{\varphi}'(s) = \int_0^s|\bm{x}''(\sigma)|^2\dot{\varphi}(\sigma)\mathrm{d}\sigma - \int_0^s h(\sigma)\mathrm{d}\sigma.
\]
Since $|\dot{\varphi}(\sigma)|=|\int_0^\sigma \dot{\varphi}'(\tilde{\sigma})\mathrm{d}\tilde{\sigma}|
\leq \sigma\sup_{0\leq r\leq \sigma}|\dot{\varphi}'(r)|$, we obtain 
\[
\sup_{0\leq r\leq s}|\dot{\varphi}'(r)| 
\leq \int_0^s \sigma|\bm{x}''(\sigma)|^2 \sup_{0\leq r\leq \sigma}|\dot{\varphi}'(r)|\mathrm{d}\sigma + \int_0^s|h(\sigma)|\mathrm{d}\sigma.
\]
Therefore, Gronwall's inequality yields 
$\sup_{0\leq r\leq s}|\dot{\varphi}'(r)| \leq \exp(\int_0^s  \sigma|\bm{x}''(\sigma)|^2\mathrm{d}\sigma)\int_0^s|h(\sigma)|\mathrm{d}\sigma$. 
It follows from Lemma \ref{lem:EstPhi} that 
$\|h\|_{L^1} \leq \|\sigma^{\frac12}\dot{\bm{x}}''\|_{L^2}\|\sigma^{\frac12}\bm{x}''\|_{L^2} \exp(\|\sigma^{\frac12}\bm{x}''\|_{L^2}^2)$, 
so that we obtain the desired estimate. 
\end{proof}

\subsection{Estimate of solutions}
In view of \eqref{BVP} we first consider the case where $h(s)$ is non-negative.

\begin{lemma}\label{lem:EstSolBVP1}
Let $\tau$ be a unique solution to the boundary value problem \eqref{TBVP}. 
Suppose that $h(s)\geq 0$ and $a+\|\sigma h\|_{L^1}\exp(-\|\sigma^{\frac12}\bm{x}''\|_{L^2}^2)\geq 0$. 
Then, for any $s\in[0,1]$ we have 
\[
\begin{cases}
 s\{ a +\|\sigma h\|_{L^1}\exp(-\|\sigma^{\frac12}\bm{x}''\|_{L^2}^2) \}
 \exp(-\|\sigma^{\frac12}\bm{x}''\|_{L^2}^2) 
 \leq \tau(s) \leq s(a+\|h\|_{L^1}), \\
 a-(a+\|h\|_{L^1})\|\sigma^{\frac12}\bm{x}''\|_{L^2}^2 \leq \tau'(s) \leq a+\|h\|_{L^1}.
\end{cases}
\]
\end{lemma}

\begin{proof}
We remind that the solution $\tau$ is expressed by Green's function as \eqref{SolBVP}. 
Under the assumptions, by Lemmas \ref{lem:EstPhi} and \ref{lem:EstPsi} we see that 
\begin{align*}
\tau(s)
& \geq \frac{\varphi(s)}{\varphi'(1)}
\left\{ a+\exp(-\|\sigma^{\frac12}\bm{x}''\|_{L^2}^2)
 \int_0^s\varphi(\sigma)h(\sigma)\mathrm{d}\sigma
 +\int_s^1\psi(\sigma)h(\sigma)\mathrm{d}\sigma \right\} \\
& \geq \frac{\varphi(s)}{\varphi'(1)}
\left\{ a+\exp(-\|\sigma^{\frac12}\bm{x}''\|_{L^2}^2)
 \int_0^s \sigma h(\sigma)\mathrm{d}\sigma
 +\exp(-\|\sigma^{\frac12}\bm{x}''\|_{L^2}^2)
 \int_s^1 \sigma h(\sigma)\mathrm{d}\sigma \right\} \\
& = \frac{\varphi(s)}{\varphi'(1)}
\{ a+\exp(-\|\sigma^{\frac12}\bm{x}''\|_{L^2}^2)\|\sigma h\|_{L^1} \} \\
& \geq s\exp(-\|\sigma^{\frac12}\bm{x}''\|_{L^2}^2)
\{ a+\exp(-\|\sigma^{\frac12}\bm{x}''\|_{L^2}^2)\|\sigma h\|_{L^1} \}, 
\end{align*}
which implies the lower bound of $\tau(s)$. 
Integrating the equation for $\tau$ over $[s,1]$, we have 
\[
\tau'(s)+\int_s^1|\bm{x}''(\sigma)|^2\tau(\sigma)\mathrm{d}\sigma = a + \int_s^1h(\sigma)\mathrm{d}\sigma.
\]
Since the positivity of $\tau(s)$ is already guaranteed, this implies the upper bound of $\tau'(s)$, and then that of $\tau(s)$. 
As for the lower bound of $\tau'(s)$, we see that 
\begin{align*}
\tau'(s)
&\geq a - \int_s^1|\bm{x}''(\sigma)|^2\tau(\sigma)\mathrm{d}\sigma \\
&\geq a - (a+\|h\|_{L^1}) \int_s^1 \sigma|\bm{x}''(\sigma)|^2\mathrm{d}\sigma.
\end{align*}
This gives the desired estimate. 
\end{proof}

We proceed to give estimate for the solution $\tau$ to the problem \eqref{TBVP} without assuming the non-negativity of $h(s)$ and $a$. 
Such estimates will be used to evaluate the derivatives of $\tau$ with respect to $t$.

\begin{lemma}\label{lem:EstSolBVP2}
For any $M>0$ there exists a constant $C=C(M)>0$ such that if $\|\sigma^{\frac12}\bm{x}''\|_{L^2} \leq M$, then the solution $\tau$ 
to the boundary value problem \eqref{TBVP} satisfies 
\[
\begin{cases}
 |\tau(s)| \leq C(|a|s+\|\sigma^\alpha h\|_{L^1}s^{1-\alpha}), \\
 s^\alpha|\tau'(s)| \leq C(|a|s^\alpha + \|\sigma^\alpha h\|_{L^1})
\end{cases}
\]
for any $s\in[0,1]$ and any $\alpha\in[0,1]$. 
\end{lemma}

\begin{proof}
It follows from Lemmas \ref{lem:EstPhi} and \ref{lem:EstPsi} that $\varphi(s)\simeq s$, $\varphi'(s)\simeq1$, $\psi(s)\simeq1$, 
and $s|\psi'(s)| \lesssim 1$. 
Since the solution $\tau(s)$ is expressed as \eqref{SolBVP}, we have 
\begin{align*}
|\tau(s)|
&\lesssim |a|s + \int_0^s \sigma|h(\sigma)| \mathrm{d}\sigma + s\int_s^1 |h(\sigma)| \mathrm{d}\sigma \\
&\lesssim |a|s + s^{1-\alpha}\int_0^1 \sigma^\alpha|h(\sigma)| \mathrm{d}\sigma.
\end{align*}
This gives the first estimate of the lemma. 
In view of 
\begin{equation}\label{Sol'BVP}
\tau'(s)=a\frac{\varphi'(s)}{\varphi'(1)}+\frac{\psi'(s)}{\varphi'(1)}\int_0^s\varphi(\sigma)h(\sigma)\mathrm{d}\sigma
 +\frac{\varphi'(s)}{\varphi'(1)}\int_s^1\psi(\sigma)h(\sigma)\mathrm{d}\sigma, 
\end{equation}
we see that 
\begin{align*}
s^\alpha|\tau'(s)|
&\lesssim |a|s^\alpha + s^\alpha|\psi'(s)| \int_0^s \sigma|h(\sigma)| \mathrm{d}\sigma
 + s^\alpha \int_s^1 |h(\sigma)| \mathrm{d}\sigma \\
&\lesssim |a|s^ \alpha + \int_0^1 \sigma^\alpha|h(\sigma)| \mathrm{d}\sigma.
\end{align*}
This gives the second estimate of the lemma. 
\end{proof}

In general, if $h(s)$ has a singularity at $s=0$, then so is $\tau'(s)$. 
To evaluate the singularity in terms of $L^p$-norm, the above pointwise estimate does not give a sharp one. 
Next, we will derive a sharp $L^p$ estimate for $\tau'(s)$. 
To this end, we prepare the following calculus inequality.

\begin{lemma}\label{lem:CalIneq}
Let $1\leq p\leq\infty$ and $\alpha+\frac{1}{p}>0$, and put $H(s)=\int_s^1 h(\sigma) \mathrm{d}\sigma$. 
Then, we have 
\[
\|s^\alpha H\|_{L^p} \leq \frac{1}{(\alpha+\frac{1}{p})^\frac{1}{p}} \|s^{\alpha+\frac{1}{p}}h\|_{L^1}.
\]
\end{lemma}

\begin{proof}
The case $p=\infty$ is trivial, so that we assume $p<\infty$. 
We may also assume without loss of generality that $h(s)$ is non-negative. 
By integration by parts, we see that 
\begin{align*}
\|s^\alpha H\|_{L^p}^p
&= \int_0^1s^{\alpha p}\left(\int_s^1 h(\sigma) \mathrm{d}\sigma \right)^p \mathrm{d}s \\
&= \left[ \frac{1}{\alpha p+1}s^{\alpha p+1}\left(\int_s^1 h(\sigma) \mathrm{d}\sigma \right)^p \right]_0^1
 + \frac{p}{\alpha p+1}\int_0^1 s^{\alpha p+1}h(s) \left(\int_s^1 h(\sigma) \mathrm{d}\sigma \right)^{p-1} \mathrm{d}s \\
&= \frac{p}{\alpha p+1} \int_0^1 s^{\alpha+\frac{1}{p}}h(s) \left(s^{\alpha+\frac{1}{p}}\int_s^1 h(\sigma) \mathrm{d}\sigma \right)^{p-1}
 \mathrm{d}s \\
&\leq \frac{p}{\alpha p+1} \left( \int_0^1 s^{\alpha+\frac{1}{p}}h(s) \mathrm{d}s \right)^p.
\end{align*}
Therefore, we obtain the desired estimate. 
\end{proof}

\begin{lemma}\label{lem:EstSolBVP3}
For any $M>0$ there exists a constant $C=C(M)>0$ such that if $\|\sigma^{\frac12}\bm{x}''\|_{L^2} \leq M$, then the solution $\tau$ 
to the boundary value problem \eqref{TBVP} satisfies 
\[
\|s^\alpha\tau'\|_{L^p} \leq C(|a|+\|s^{\alpha+\frac{1}{p}}h\|_{L^1})
\]
for any $p\in[1,\infty]$ and any $\alpha\geq0$ satisfying $\alpha+\frac{1}{p}\leq1$. 
\end{lemma}

\begin{proof}
The case $p=\infty$ has already been proved in Lemma \ref{lem:EstSolBVP2}, so that we assume $p<\infty$. 
It follows from \eqref{Sol'BVP} and Lemmas \ref{lem:EstPhi} and \ref{lem:EstPsi} that 
\begin{align*}
s^\alpha|\tau'(s)|
&\lesssim |a|s^\alpha + s^\alpha|\psi'(s)|\int_0^s\sigma|h(\sigma)| \mathrm{d}\sigma
 + s^\alpha\int_s^1|h(\sigma)| \mathrm{d}\sigma \\
&\lesssim |a|s^\alpha + s^{1-\frac{1}{p}}|\psi'(s)| \int_0^s \sigma^{\alpha+\frac{1}{p}}|h(\sigma)| \mathrm{d}\sigma
 + s^\alpha\int_s^1|h(\sigma)| \mathrm{d}\sigma.
\end{align*}
Here, in view of \eqref{PfEstPsi} we have $|\psi'(s)| \lesssim \int_s^1 |\bm{x}''(\sigma)|^2 \mathrm{d}\sigma$, 
so that by Lemma \ref{lem:CalIneq} we get 
\[
\|s^{1-\frac{1}{p}}\psi'\|_{L^p} \lesssim \|s|\bm{x}''|^2\|_{L^1} = \|s^{\frac12}\bm{x}''\|_{L^2}^2 \lesssim 1. 
\]
Therefore, by Lemma \ref{lem:CalIneq} again we obtain 
\[
\|s^\alpha\tau'\|_{L^p}
\lesssim |a| + \biggl( 1+ \frac{1}{(\alpha+\frac{1}{p})^\frac{1}{p}} \biggr) \|s^{\alpha+\frac{1}{p}}h\|_{L^1}. 
\]
Since $(\alpha+\frac{1}{p})^\frac{1}{p} \geq (\frac{1}{p})^\frac{1}{p} \geq \exp(\exp(-1))$, we obtain the desired estimate. 
\end{proof}

\section{Function spaces}\label{sect:FS}
In this section, we present several properties related to the weighted Sobolev space $X^m$, 
which is equipped with the norm $\|\cdot\|_{X^m}$ defined by \eqref{WSS}.

\subsection{Weighted Sobolev space $X^m$}
The norm $\|\cdot\|_{X^m}$ is essentially the same one introduced by Reeken \cite{Reeken1979-1}. 
By the definition of the norm, it holds obviously that $\|u\|_{X^m} \leq \|u\|_{X^{m+1}}$ for $m=0,1,2,\ldots$. 
Moreover, this space is characterized as follows. 
Let $D$ be the unit disc in $\mathbb{R}^2$ and $H^m(D)$ the $L^2$ Sobolev space of order $m$ on $D$. 
For a function $u$ defined in the open interval $(0,1)$, we define $u^\sharp(x,y)=u(x^2+y^2)$ which is a function on $D$.

\begin{lemma}[{\cite[Proposition 3.2]{Takayama2018}}]\label{lem:NormEq}
Let $m$ be a non-negative integer. 
The map $X^m\ni u \mapsto u^\sharp \in H^m(D)$ is bijective and it holds that $\|u\|_{X^m} \simeq \|u^\sharp\|_{H^m(D)}$ 
for any $u\in X^m$. 
\end{lemma}

\begin{lemma}\label{lem:EstA2}
For a non-negative integer $m$, we have $\|(su')'\|_{X^m} \lesssim \|u\|_{X^{m+2}}$. 
\end{lemma}

\begin{proof}
Put $(A_2u)(s)=-(su'(s))'$. 
Since $\Delta u^\sharp = -4 (A_2u)^\sharp$, by Lemma \ref{lem:NormEq} we have 
$\|A_2u\|_{X^m} \simeq \|\Delta u^\sharp\|_{H^m(D)} \leq \|u^\sharp\|_{H^{m+2}(D)} \simeq \|u\|_{X^{m+2}}$, 
so that we obtain the desired estimate. 
\end{proof}

\begin{lemma}\label{lem:embedding}
For any $q\in[1,\infty)$ and any $\epsilon>0$ there exist positive constants $C_q=C(q)$ and $C_\epsilon=C(\epsilon)$ such that 
for any $u\in X^1$ we have $\|u\|_{L^q} \leq C_q\|u\|_{X^1}$ and $\|s^\epsilon u\|_{L^\infty} \leq C_\epsilon\|u\|_{X^1}$. 
\end{lemma}

\begin{proof}
It is easy to see that $\|u^\sharp\|_{L^q(D)}=\pi^\frac1q\|u\|_{L^q}$ for any $q\in[1,\infty]$. 
Therefore, the Sobolev embedding theorem $H^1(D) \hookrightarrow L^q(D)$ and Lemma \ref{lem:NormEq} gives the first estimate of the lemma. 
As for the second one, we let $q\geq\frac32$. 
For $s\in[0,1]$ we see that 
\begin{align*}
s|u(s)|^q
&= \int_0^s(\sigma|u(\sigma)|^q)'\mathrm{d}\sigma \\
&\leq \int_0^1( |u(\sigma)|^q + q|u(\sigma)|^{q-1}\sigma|u'(\sigma)| )\mathrm{d}\sigma \\
&\leq \|u\|_{L^{2(q-1)}}^{q-1}(\|u\|_{L^2}+q\|\sigma u'\|_{L^2}) \\
&\lesssim \|u\|_{X^1}^q,
\end{align*}
where we used the Cauchy--Schwarz inequality and the first estimate of the lemma. 
This gives $s^\frac1q|u(s)| \lesssim \|u\|_{X^1}$. 
Since $q\geq\frac32$ is arbitrary, we obtain the second estimate of the lemma. 
\end{proof}

\begin{remark}\label{remark:embedding}
\begin{enumerate}
\item[\rm(1)]
In Lemma \ref{lem:embedding} we cannot take $q=\infty$ or equivalently $\epsilon=0$. 
In other words, the embedding $X^1 \hookrightarrow L^\infty(0,1)$ does not hold. 
A counter-example is given by $u(s)=\log(\log(\frac{\mathrm{e}}{s}))$. 
\item[\rm(2)]
Unlike the standard Sobolev spaces, $u\in X^{m+1}$ does not necessarily imply $u'\in X^m$. 
A counter-example is given by $u(s)=\int_0^s\log(\log(\frac{\mathrm{e}}{\sigma}))\mathrm{d}\sigma$, which is in $X^3$. 
However, in view of the embedding $X^2 \hookrightarrow L^\infty(0,1)$ we easily check that its first derivative $u'$ is not in $X^2$. 
\item[\rm(3)]
Alternatively, we have $\|su'\|_{X^m} \leq \|u\|_{X^{m+1}}$ and $\|u'\|_{X^m} \leq \|u\|_{X^{m+2}}$. 
\end{enumerate}
\end{remark}

\begin{lemma}\label{lem:embedding2}
For a non-negative integer $j$, we have $\|\ds^j u\|_{L^\infty} \lesssim \|u\|_{X^{2j+2}}$. 
\end{lemma}

\begin{proof}
Since $X^2 \hookrightarrow H^1 \hookrightarrow L^\infty$, we have $\|u\|_{L^\infty} \lesssim \|u\|_{X^2}$. 
This together with $\|u'\|_{X^m} \leq \|u\|_{X^{m+2}}$ yields the desired estimate. 
\end{proof}

The following lemma gives a weighted $L^p$ estimate for derivatives of functions defined in $(0,1)$ in terms of $X^m$ norm.

\begin{lemma}\label{lem:CalIneqLp1}
For a positive integer $k$ there exists a positive constant $C$ such that for any $p\in[2,\infty]$ and any $j=1,2,\ldots,k$ we have 
\[
\begin{cases}
 \|s^{j-\frac12-\frac1p}\ds^{k+j-1}u\|_{L^p}
  \leq C\bigl( \|s^{j-1}\ds^{k+j-1}u\|_{L^2} + \|s^{j-1}\ds^{k+j-1}u\|_{L^2}^{\frac12+\frac1p} 
   \|s^{j}\ds^{k+j}u\|_{L^2}^{\frac12-\frac1p} \bigr), \\[1ex]
 \|s^{j-\frac1p}\ds^{k+j}u\|_{L^p}
  \leq C\bigl( \|s^{j-\frac12}\ds^{k+j}u\|_{L^2}
   + \|s^{j-\frac12}\ds^{k+j}u\|_{L^2}^{\frac12+\frac1p} \|s^{j+\frac12}\ds^{k+j+1}u\|_{L^2}^{\frac12-\frac1p} \bigr).
\end{cases}
\]
Particularly, 
\[
\begin{cases}
 \displaystyle
 \sum_{j=1}^k\|s^{j-\frac12-\frac1p}\ds^{k+j-1}u\|_{L^p} \leq C\|u\|_{X^{2k}}, \\
 \displaystyle
 \sum_{j=1}^k\|s^{j-\frac1p}\ds^{k+j}u\|_{L^p} \leq C\|u\|_{X^{2k+1}}.
\end{cases}
\]
\end{lemma}

\begin{proof}
It is sufficient to show the first two estimates. 
\begin{align*}
s^{2j-1}|\ds^{k+j-1}u(s)|^2
&= \int_0^s (\sigma^{2j-1}|\partial_\sigma^{k+j-1}u(\sigma)|^2)' \mathrm{d}\sigma \\
&\lesssim \int_0^1\{ \sigma^{2(j-1)}|\partial_\sigma^{k+j-1}u(\sigma)|^2
 + \sigma^{j-1}|\partial_\sigma^{k+j-1}u(\sigma)| \sigma^j|\partial_\sigma^{k+j}u(\sigma)| \}\mathrm{d}\sigma \\
&\leq \|\sigma^{j-1}\partial_\sigma^{k+j-1}u\|_{L^2}^2
 + \|\sigma^{j-1}\partial_\sigma^{k+j-1}u\|_{L^2}\|\sigma^{j}\partial_\sigma^{k+j}u\|_{L^2},
\end{align*}
which shows the first estimate in the case $p=\infty$. 
The case $p=2$ is trivial. 
Therefore, interpolating the estimates in the cases $p=2$ and $p=\infty$ we obtain the first estimate in the case $2\leq p\leq \infty$. 
The second estimate can be proved in the same way so that we omit the proof. 
\end{proof}

\begin{lemma}\label{lem:Algebra}
Let $m$ be a non-negative integer. 
Then, we have 
\[
\begin{cases}
 \|uv\|_{L^2} \lesssim \|u\|_{X^1} \|v\|_{X^1}, \\
 \|uv\|_{X^m} \lesssim \|u\|_{X^{m \vee 2}} \|v\|_{X^m}.
\end{cases}
\]
\end{lemma}

\begin{proof}
These estimates follows directly from Lemma \ref{lem:NormEq} together with the well-known inequalities 
$\|UV\|_{L^2(D)} \lesssim \|U\|_{H^1(D)} \|V\|_{H^1(D)}$ and $\|UV\|_{H^m(D)} \lesssim \|U\|_{H^{m\vee2}(D)} \|V\|_{H^m(D)}$. 
\end{proof}

\begin{lemma}\label{lem:EstCompFunc1}
Let $m$ be an non-negative integer, $\Omega$ an open set in $\mathbb{R}^N$, and $F\in C^m(\Omega)$. 
There exists a positive constant $C=C(m,N)$ such that if $u\in X^m$ takes its value in a compact set $K$ in $\Omega$, then we have 
\[
\|F(u)\|_{X^m} \leq C\|F\|_{C^m(K)}(1+\|u\|_{X^m})^m.
\]
If, in addition, $u$ depends also on time $t$, then we have also 
\[
\begin{cases}
 \opnorm{ F(u(t)) }_m \leq C\|F\|_{C^m(K)}(1+\opnorm{ u(t) }_m )^m, \\
 \opnorm{ F(u(t)) }_{m,*} \leq C\|F\|_{C^m(K)}(1+\opnorm{ u(t) }_{m,*} )^m.
\end{cases}
\]
\end{lemma}

\begin{proof}
It is well-known that if $U\in H^m(D)$ takes its value in a compact set $K$, then we have 
$\|F(U)\|_{H^m(D)} \leq C\|F\|_{C^m(K)}(1+\|U\|_{H^m(D)})^m$. 
This together with Lemma \ref{lem:NormEq} implies the first estimate of the lemma. 
Similarly, we can obtain the later estimates. 
\end{proof}

\begin{lemma}\label{lem:commutator}
Let $j$ be a non-negative integer. 
It holds that 
\begin{align*}
& \|s^\frac{j}{2}[\ds^{j+1},u]v\|_{L^2} \lesssim
\begin{cases}
 \min\{ \|u'\|_{L^\infty}\|v\|_{L^2}, \|u'\|_{L^2}\|v\|_{L^\infty} \} &\mbox{for}\quad j=0, \\
 \min\{ \|u'\|_{X^2}\|v\|_{X^1}, \|u'\|_{X^1}\|v\|_{X^2} \} &\mbox{for}\quad j=1, \\
 \|u'\|_{X^j}\|v\|_{X^j} &\mbox{for}\quad j\geq2.
\end{cases}
\end{align*}
\end{lemma}

\begin{proof}
The case $j=0$ is trivial. 
By Lemmas \ref{lem:CalIneqLp1} and \ref{lem:embedding2} with $j=0$, we see that 
\begin{align*}
\|s^\frac12[\ds^2,u]v\|_{L^2}
&\leq \|s^\frac12u''\|_{L^\infty}\|v\|_{L^2} + 2\|u'\|_{L^\infty}\|s^\frac12v'\|_{L^2} \\
&\lesssim \|u'\|_{X^2} \|v\|_{X^1}, \\
\|s^\frac12[\ds^2,u]v\|_{L^2}
&\leq \|s^\frac12u''\|_{L^2}\|v\|_{L^\infty} + 2\|u'\|_{L^2}\|s^\frac12v'\|_{L^\infty} \\
&\lesssim \|u'\|_{X^1} \|v\|_{X^2},
\end{align*}
which yield the estimate in the case $j=1$. 
In the case $j\geq2$, we evaluate it as 
\[
\|s^\frac{j}{2}[\ds^{j+1},u]v\|_{L^2}
\lesssim \|s^\frac{j}{2}\ds^j u'\|_{L^2} \|v\|_{L^\infty} 
 +\sum_{j_1+j_2=j, j_1\leq j-1} \|s^{\alpha_1}\ds^{j_1}u'\|_{L^\infty} \|s^{\alpha_2}\ds^{j_2}v\|_{L^2},
\]
where $\alpha_1$ and $\alpha_2$ should be taken so that $\alpha_1+\alpha_2 \leq \frac{j}{2}$. 
We choose these indices as follows: 
\begin{enumerate}
\item[(1)]
The case $j=2k$: 
\[
\alpha_1 = 
\begin{cases}
 0 &\mbox{if}\quad j_1\leq k-1, \\
 j_1-k+\frac12 &\mbox{if}\quad j_1\geq k, 
\end{cases}
\qquad
\alpha_2 = 
\begin{cases}
 0 &\mbox{if}\quad j_2\leq k, \\
 j_2-k &\mbox{if}\quad j_2\geq k+1.
\end{cases}
\]
\item[(2)]
The case $j=2k+1$: 
\[
\alpha_1 = 
\begin{cases}
 0 &\mbox{if}\quad j_1\leq k-1, \\
 \frac12 &\mbox{if}\quad j_1=k, \\
 j_1-k &\mbox{if}\quad j_1\geq k+1, 
\end{cases}
\qquad
\alpha_2 = 
\begin{cases}
 0 &\mbox{if}\quad j_2\leq k, \\
 j_2-k-\frac12 &\mbox{if}\quad j_2\geq k+1.
\end{cases}
\]
\end{enumerate}
Then, we see that these indices satisfy the desired property. 
Here, in the case $j=2k+1$ by Lemmas \ref{lem:embedding} with $\epsilon=\frac12$ we have 
$\|s^\frac12\ds^ku'\|_{L^\infty} \lesssim \|\ds^ku'\|_{X^1} \leq \|u'\|_{X^j}$, 
which together with Lemma \ref{lem:CalIneqLp1} implies $\|s^{\alpha_1}\ds^{j_1}u'\|_{L^\infty} \lesssim \|u'\|_{X^j}$ and 
$\|s^{\alpha_2}\ds^{j_2}v\|_{L^2} \lesssim \|v\|_{X^j}$. 
Therefore, we obtain the the estimate in the case $j\geq2$. 
\end{proof}

\subsection{Weighted Sobolev space $Y^m$}
For a non-negative integer $m$ we define another weighted Sobolev space $Y^m$ as the set of all function $u$ defined in the open interval $(0,1)$ 
equipped with a norm $\|\cdot\|_{Y^m}$ defined by 
\[
\|u\|_{Y^m}^2 =
\begin{cases}
 \|s^\frac12 u\|_{L^2}^2 &\mbox{for}\quad m=0, \\
 \displaystyle
 \|u\|_{H^k}^2 + \sum_{j=1}^{k+1}\|s^j\ds^{k+j}u\|_{L^2}^2 &\mbox{for}\quad m=2k+1, \\
 \displaystyle
 \|u\|_{H^k}^2 + \sum_{j=1}^{k+2}\|s^{j-\frac12}\ds^{k+j}u\|_{L^2}^2 &\mbox{for}\quad m=2k+2.
\end{cases}
\]
Obviously, it holds that $\|u\|_{Y^m} \leq \|u\|_{Y^{m+1}}$ and $\|u\|_{Y^m} \leq \|u\|_{X^m} \leq \|u\|_{Y^{m+1}}$ for $m=0,1,2,\ldots$. 
This function space $Y^m$ is introduced so that the identity 
\begin{equation}\label{Ym1}
\|u\|_{X^{m+1}}^2 = \|u\|_{L^2}^2 + \|u'\|_{Y^m}^2
\end{equation}
holds for any $m=0,1,2,\ldots$. 
Particularly, $u\in X^{m+1}$ implies $u'\in Y^m$. 
Note also that the identity 
\begin{equation}\label{Ym2}
\|u\|_{Y^{m+1}}^2 = \|u\|_{X^m}^2 + \|s^{\frac{m+2}{2}}\ds^{m+1}u\|_{L^2}^2
\end{equation}
holds for any $m=0,1,2,\ldots$.

\begin{lemma}\label{lem:CalIneq1}
It holds that 
\[
\|u'v'\|_{Y^m} \lesssim
\begin{cases}
 \|u\|_{X^{m+2}} \|v\|_{X^{m+2}} &\mbox{for}\quad m=0,1, \\
 \|u\|_{X^{m+1\vee4}} \|v\|_{X^{m+1}} &\mbox{for}\quad m=0,1,2,\ldots.
\end{cases}
\]
\end{lemma}

\begin{proof}
By Lemmas \ref{lem:embedding} and \ref{lem:CalIneqLp1}, we see that 
\begin{align*}
\|u'v'\|_{Y^0}
&\leq \|s^\frac14u'\|_{L^4} \|s^\frac14v'\|_{L^4} \\
&\lesssim \|u\|_{X^2} \|v\|_{X^2}, 
\end{align*}
\begin{align*}
\|u'v'\|_{Y^1}
&\leq \|u'\|_{L^4} \|v'\|_{L^4} + \|s^\frac12u'\|_{L^\infty}\|s^\frac12v''\|_{L^2} + \|s^\frac12u''\|_{L^2}\|s^\frac12v''\|_{L^\infty} \\
&\lesssim \|u'\|_{X^1} \|v'\|_{X^1} + \|u'\|_{X^1}\|v\|_{X^3} + \|u\|_{X^3}\|v'\|_{X^1} \\
&\lesssim \|u\|_{X^3} \|v\|_{X^3},
\end{align*}
which give the first estimate of the lemma.

Similarly, by Lemmas \ref{lem:embedding2} and \ref{lem:CalIneqLp1} we see that 
\begin{align*}
\|u'v'\|_{Y^0}
&\leq \|u'\|_{L^\infty} \|s^\frac12v'\|_{L^2} \\
&\lesssim \|u\|_{X^4} \|v\|_{X^1}, 
\end{align*}
\begin{align*}
\|u'v'\|_{Y^1}
&\leq \|u'\|_{L^\infty}( \|v'\|_{L^2} + \|sv''\|_{L^2}) + \|s^\frac12u''\|_{L^\infty}\|v'\|_{L^2} \\
&\lesssim \|u\|_{X^4} \|v\|_{X^2},
\end{align*}
\begin{align*}
\|u'v'\|_{Y^2}
&\leq \|u'\|_{L^\infty}( \|v'\|_{L^2} + \|s^\frac12v''\|_{L^2} + \|sv'''\|_{L^2} ) \\
&\quad\;
 + \|s^\frac12u''\|_{L^\infty}( \|v'\|_{L^2} + \|s^\frac12v''\|_{L^2} ) + \|s^\frac32u'''\|_{L^\infty}\|v'\|_{L^2}\\
&\lesssim \|u\|_{X^4} \|v\|_{X^3},
\end{align*}
which give the second estimate for $m=0,1,2$.

We then consider the case $m=2k+1$ with $k\geq1$. 
By \eqref{Ym2} and Lemmas \ref{lem:Algebra} and \ref{lem:CalIneqLp1}, we see that 
\begin{align*}
\|u'v'\|_{Y^{2k+1}}
&\leq \|u'v'\|_{X^{2k}} + \|s^{k+1}\ds^{2k+1}(u'v')\|_{L^2} \\
&\lesssim \|u'\|_{X^{2k}} \|v'\|_{X^{2k}}
 + \sum_{j=0}^k \bigl( \|s^j\ds^{j+1}u\|_{L^\infty} \|s^{k+1-j}\ds^{2k+2-j}v\|_{L^2} \\
&\qquad
  + \|s^j\ds^{j+1}v\|_{L^\infty} \|s^{k+1-j}\ds^{2k+2-j}u\|_{L^2} \bigr) \\
&\lesssim \|u\|_{X^{2k+2}} \|v\|_{X^{2k+2}},
\end{align*}
which gives the second estimate for $m=2k+1$.

Finally, we consider the case $m=2k$ with $k\geq2$. 
Similarly as above, we see that 
\begin{align*}
\|u'v'\|_{Y^{2k}}
&\lesssim \|u'v'\|_{X^{2k-1}} + \|s^{k+\frac12}\ds^{2k}(u'v')\|_{L^2} \\
&\lesssim \|u'\|_{X^{2k-1}} \|v'\|_{X^{2k-1}}
 + \sum_{j=0}^k \bigl( \|s^j\ds^{j+1}u\|_{L^\infty} \|s^{k+\frac12-j}\ds^{2k+1-j}v\|_{L^2} \\
&\qquad
 + \|s^j\ds^{j+1}v\|_{L^\infty} \|s^{k+\frac12-j}\ds^{2k+1-j}u\|_{L^2} \bigr) \\
&\lesssim \|u\|_{X^{2k+1}} \|v\|_{X^{2k+1}},
\end{align*}
which gives the second estimate for $m=2k$. 
The proof is complete. 
\end{proof}

\begin{lemma}\label{lem:CalIneq2}
If $\tau|_{s=0}=0$, then we have 
\[
\|\tau u''v''\|_{Y^m} \lesssim
\begin{cases}
 \|\tau'\|_{L^2} \|u\|_{X^4} \|v\|_{X^3} &\mbox{for}\quad m=0, \\
 \|\tau'\|_{L^\infty} \min\{ \|u\|_{X^4} \|v\|_{X^2}, \|u\|_{X^3} \|v\|_{X^3} \} &\mbox{for}\quad m=0, \\
 \min\{ \|\tau'\|_{L^2} \|u\|_{X^4},\|\tau'\|_{L^\infty} \|u\|_{X^3}\} \|v\|_{X^4} &\mbox{for}\quad m=1, \\
 \|\tau'\|_{L^\infty \cap X^{m-1}} \|u\|_{X^{m+2}} \|v\|_{X^{m+2}} &\mbox{for}\quad m\geq2.
\end{cases}
\]
\end{lemma}

\begin{proof}
We first note that under the assumption $\tau|_{s=0}=0$ we have $\tau(s)=\int_0^s\tau'(\sigma)\mathrm{d}\sigma$, 
so that $|\tau(s)| \leq s^{1-\frac1p}\|\tau'\|_{L^p}$ for $1\leq p\leq\infty$. 
Therefore, we see that 
\[
\|\tau u''v''\|_{Y^0} \leq \min\{ \|\tau'\|_{L^2}\|s u''v''\|_{L^2},\|\tau'\|_{L^\infty}\|s^\frac32 u''v''\|_{L^2} \}
\]
and that by Lemma \ref{lem:CalIneqLp1} 
\begin{align*}
\|s u''v''\|_{L^2} 
&\leq \|s^\frac12u''\|_{L^\infty} \|s^\frac12v''\|_{L^2} \\
&\lesssim \|u\|_{X^4}\|v\|_{X^3}
\end{align*}
and 
\begin{align*}
\|s^\frac32 u''v''\|_{L^2} 
&\leq \min\{ \|s^\frac12u''\|_{L^\infty} \|sv''\|_{L^2}, \|su''\|_{L^\infty} \|s^\frac12v''\|_{L^2} \} \\
&\lesssim \min\{ \|u\|_{X^4} \|v\|_{X^2}, \|u\|_{X^3} \|v\|_{X^3} \},
\end{align*}
which give the estimates of the lemma in the case $m=0$. 
Similarly, we see that 
\begin{align*}
\|\tau u''v''\|_{Y^1}
&\leq \|\tau'\|_{L^2}( \|s^\frac12u''v''\|_{L^2} + \|s^\frac32 u'''v''\|_{L^2} + \|s^\frac32u''v'''\|_{L^2} ) \\
&\leq \|\tau'\|_{L^2}( \|s^\frac12u''\|_{L^\infty} \|v''\|_{L^2} 
 + \|s^\frac32u'''\|_{L^\infty}\|v''\|_{L^2} + \|u''\|_{L^2}\|s^\frac32v'''\|_{L^\infty} ) \\
&\lesssim \|\tau'\|_{L^2} \|u\|_{X^4} \|v\|_{X^4}, 
\end{align*}
and 
\begin{align*}
\|\tau u''v''\|_{Y^1}
&\leq \|\tau'\|_{L^\infty}( \|su''v''\|_{L^2} + \|s^2u'''v''\|_{L^2} + \|s^2u''v'''\|_{L^2} ) \\
&\leq  \|\tau'\|_{L^\infty}( \|s^\frac12u''\|_{L^2} \|s^\frac12v''\|_{L^\infty} 
 + \|s^\frac32u'''\|_{L^2}\|s^\frac12v''\|_{L^\infty} + \|s^\frac12u''\|_{L^2}\|s^\frac32v'''\|_{L^\infty} ) \\
&\lesssim \|\tau'\|_{L^\infty} \|u\|_{X^3} \|v\|_{X^4}, 
\end{align*}
which give the estimates of the lemma in the case $m=1$.

We then consider the case $m=2k$ with $k\geq1$. 
We have 
\begin{align*}
\|\tau u''v''\|_{Y^{2k}}
&\lesssim \sum_{j_0+j_1+j_2\leq k-1}\|(\ds^{j_0}\tau)(\ds^{j_1+2}u)(\ds^{j_2+2}v)\|_{L^2} \\
&\quad\;
 + \sum_{j=1}^{k+1}\sum_{j_0+j_1+j_2=k+j-1} \|s^{j-\frac12}(\ds^{j_0}\tau)(\ds^{j_1+2}u)(\ds^{j_2+2}v)\|_{L^2}.
\end{align*}
We first evaluate $I_1(j_0,j_1,j_2)=\|(\ds^{j_0}\tau)(\ds^{j_1+2}u)(\ds^{j_2+2}v)\|_{L^2}$, where $j_0+j_1+j_2\leq k-1$. 
In the following calculations, we use frequently Lemma \ref{lem:CalIneqLp1}. 
\begin{enumerate}
\item[(i)]
The case $(j_0,j_1,j_2)=(0,k-1,0),(0,0,k-1)$. 
\begin{align*}
I_1(0,k-1,0)
&\leq \|\tau'\|_{L^2} \|\ds^{k+1}u\|_{L^2}\|s^\frac12v''\|_{L^\infty} \\
&\lesssim \|\tau'\|_{L^2} \|u\|_{H^{k+1}} \|v\|_{X^4} \\
&\lesssim \|\tau'\|_{L^2} \|u\|_{X^{2k+2}} \|v\|_{X^{2k+2}}.
\end{align*}
Similar estimate holds for $I_1(0,0,k-1)$. 

\item[(ii)]
The other cases. 
Since $j_1,j_2 \leq k-2$, we have 
\begin{align*}
I_1(j_0,j_1,j_2)
&\lesssim \|\ds^{j_0}\tau\|_{L^\infty} \|\ds^{j_1+2}u\|_{H^1} \|\ds^{j_2+2}v\|_{H^1} \\
&\lesssim \|\tau\|_{W^{k-1,\infty}} \|u\|_{H^{k+1}} \|v\|_{H^{k+1}} \\
&\lesssim \|\tau'\|_{X^{2(k-1)}} \|u\|_{X^{2k+2}} \|v\|_{X^{2k+2}}.
\end{align*}
\end{enumerate}
In any of these cases, we have 
$I_1(j_0,j_1,j_2) \lesssim \|\tau'\|_{X^{2(k-1)}} \|u\|_{X^{2k+2}} \|v\|_{X^{2k+2}}$.

We then evaluate $I_2(j_0,j_1,j_2;j)=\|s^{j-\frac12}(\ds^{j_0}\tau)(\ds^{j_1+2}u)(\ds^{j_2+2}v)\|_{L^2}$, 
where $1\leq j\leq k+1$ and $j_0+j_1+j_2=k+j-1$. 
In the following calculations, we will use Lemmas \ref{lem:embedding} and \ref{lem:CalIneqLp1}. 
\begin{enumerate}
\item[(i)]
The case $j=k+1$ and $j_1=j_2=0$. 
\begin{align*}
I_2(2k,0,0;k+1)
&\leq \|s^{k-\frac12}\ds^{2k}\tau\|_{L^2} \|s^\frac12u''\|_{L^\infty} \|s^\frac12v''\|_{L^\infty} \\
&\lesssim \|\tau'\|_{X^{2k-1}} \|u\|_{X^4} \|v\|_{X^4} \\
&\leq \|\tau'\|_{X^{2k-1}} \|u\|_{X^{2(k+1)}} \|v\|_{X^{2k+2}}.
\end{align*}
\item[(ii)]
The case $(j_0,j_1,j_2)=(0,k+j-1,0),(0,0,k+j-1)$. 
\begin{align*}
I_2(0,k+j-1,0;j)
&\leq \|\tau'\|_{L^\infty} \|s^j\ds^{k+j+1}u\|_{L^2} \|s^\frac12v''\|_{L^\infty} \\
&\lesssim \|\tau'\|_{L^\infty} \|u\|_{X^{2k+2}} \|v\|_{X^{2k+2}}.
\end{align*}
Similar estimate holds for $I_2(0,0,k+j-1;j)$. 
\item[(iii)]
The case $j_0=0$ and $j_1,j_2\ne k+j-1$. 
We evaluate it as 
\[
I_2(0,j_1,j_2;j) \leq \|\tau'\|_{L^\infty} \|s^{\alpha_1}\ds^{j_1+2}u \|_{L^\infty} \|s^{\alpha_2}\ds^{j_2+2}v\|_{L^2},
\]
where $\alpha_1$ and $\alpha_2$ should be chosen so that $\alpha_1+\alpha_2\leq j+\frac12$. 
We choose these indices as 
\[
\alpha_1 = 
\begin{cases}
 0 &\mbox{if}\quad j_1+2 \leq k, \\
 j_1-k+\frac32 &\mbox{if}\quad j_1+2 \geq k+1,
\end{cases}
\qquad
\alpha_2 = 
\begin{cases}
 0 &\mbox{if}\quad j_2+2 \leq k+1, \\
 j_2-k+1 &\mbox{if}\quad j_2+2 \geq k+2.
\end{cases} 
\]
Then, we see that $\alpha_1$ and $\alpha_2$ satisfy in fact $\alpha_1+\alpha_2\leq j+\frac12$. 
Moreover, we have $\|s^{\alpha_1}\ds^{j_1+2}u \|_{L^\infty} \lesssim \|u\|_{X^{2k+2}}$ and 
$\|s^{\alpha_2}\ds^{j_2+2}v\|_{L^2} \lesssim \|v\|_{X^{2k+2}}$.

\item[(iv)]
The other cases. 
In view of $1\leq j_0\leq 2k-1$ and $j_1,j_2\leq k+j-2$, we evaluate it as 
\[
I_2(0,j_1,j_2;j) \leq \|s^{\alpha_0}\ds^{j_0}\tau\|_{L^\infty} \|s^{\alpha_1}\ds^{j_1+2}u \|_{L^\infty} \|s^{\alpha_2}\ds^{j_2+2}v\|_{L^2},
\]
where $\alpha_0$, $\alpha_1$, and $\alpha_2$ should be chosen so that $\alpha_0+\alpha_1+\alpha_2\leq j-\frac12$. 
We choose the indices $\alpha_0$ as 
\[
\alpha_0 = 
\begin{cases}
 0 &\mbox{if}\quad j_0=1 \mbox{ or } 1\leq j_0\leq k-1, \\
 \frac12 &\mbox{if}\quad j_0=k\geq2, \\
 j_0-k &\mbox{if}\quad j_0\geq k+1,
\end{cases}
\]
$\alpha_1$ and $\alpha_2$ as in the above case (iii). 
Then, we see that these indices satisfy in fact $\alpha_0+\alpha_1+\alpha_2\leq j-\frac12$. 
Moreover, we have $\|s^{\alpha_0}\ds^{j_0}\tau\|_{L^\infty} \lesssim \|\tau'\|_{L^\infty \cap X^{2k-1}}$, 
$\|s^{\alpha_1}\ds^{j_1+2}u \|_{L^\infty} \lesssim \|u\|_{X^{2(k+1)}}$ and 
$\|s^{\alpha_2}\ds^{j_2+2}v\|_{L^2} \lesssim \|v\|_{X^{2k+2}}$. 
\end{enumerate}
In any of these cases, we have 
$I_2(j_0,j_1,j_2;j) \lesssim \|\tau'\|_{L^\infty \cap X^{m-1}} \|u\|_{X^{2k+2}} \|v\|_{X^{2k+2}}$. 
To summarize, we obtain the desired estimate of the lemma in the case $m=2k$.

The case $m=2k+1$ with $k\geq1$ can be proved in the same way as above so we omit the proof in the case. 
\end{proof}

\subsection{Averaging operator $\mathscr{M}$}\label{sect:AOM}
For a function $u$ defined in the open interval $(0,1)$ we define an averaging operator $\mathscr{M}$ by 
\begin{equation}\label{defM}
(\mathscr{M}u)(s) = \frac{1}{s}\int_0^su(\sigma) \mathrm{d}\sigma = \int_0^1u(sr) \mathrm{d}r.
\end{equation}
Then, we have 
\begin{equation}\label{derM}
\ds^j(\mathscr{M}u)(s)
 = \int_0^1r^j(\partial_\sigma^ju)(sr) \mathrm{d}r
 = \frac{1}{s^{j+1}}\int_0^s \sigma^j\partial_\sigma^ju(\sigma) \mathrm{d}\sigma.
\end{equation}
We will evaluate this in a weighted $L^p$ space.

\begin{lemma}\label{lem:EstM}
Let $1\leq p\leq\infty$. 
Suppose that $\alpha$ and $\beta$ satisfy $\alpha+1>\beta+\frac1p$. 
For a function $u$ defined in $(0,1)$ we put 
\[
U_\alpha(s)=\frac{1}{s^{\alpha+1}}\int_0^s\sigma^\alpha u(\sigma) \mathrm{d}\sigma. 
\]
Then, we have 
\[
\|s^\beta U_\alpha\|_{L^p} \leq \frac{1}{\alpha-\beta+1-\frac1p}\|s^\beta u\|_{L^p}.
\]
\end{lemma}

\begin{proof}
Since $s^\beta U_\alpha(s) = \frac{1}{s^{\alpha-\beta+1}}\int_0^s \sigma^{\alpha-\beta}(\sigma^\beta u(\sigma)) \mathrm{d}\sigma$, 
it is sufficient to show the estimate in the case $\beta=0$. 
We may assume also that $u(s)$ is non-negative. 
We first consider the case $p=\infty$, so that we assume $\alpha+1>0$. 
\begin{align*}
|U_\alpha(s)|
&\leq \frac{1}{s^{\alpha+1}}\int_0^s \sigma^\alpha \mathrm{d}\sigma \|u\|_{L^\infty} = \frac{1}{\alpha+1}\|u\|_{L^\infty},
\end{align*}
which shows the estimate in the case $p=\infty$. 
Therefore, we suppose that $1\leq p<\infty$. 
By using integration by parts and noting the condition $(\alpha+1)p-1>0$, we see that 
\begin{align*}
\|U_\alpha\|_{L^p}^p
&= \left[ -\frac{1}{(\alpha+1)p-1}\frac{1}{s^{(\alpha+1)p-1}}\biggl( \int_0^s \sigma^\alpha u(\sigma) \mathrm{d}\sigma \biggr)^p \right]_0^1 \\
&\quad\;
 + \frac{p}{(\alpha+1)p-1}\int_0^1\frac{s^\alpha u(s)}{s^{(\alpha+1)p+1}} 
 \biggl( \int_0^s \sigma^\alpha u(\sigma) \mathrm{d}\sigma \biggr)^{p-1} \mathrm{d}s \\
&\leq \frac{p}{(\alpha+1)p-1}\int_0^1 u(s)(U_\alpha(s))^{p-1} \mathrm{d}s \\
&\leq \frac{p}{(\alpha+1)p-1}\|u\|_{L^p}\|U_\alpha\|_{L^p}^{p-1},
\end{align*}
where we used H\"older's inequality. 
This shows the desired estimate. 
\end{proof}

\begin{corollary}\label{cor:WEM1}
Let $j$ be a non-negative integer, $1\leq p\leq \infty$, and $\beta<j+1-\frac1p$. 
Then, we have 
\[
\|s^\beta\ds^j(\mathscr{M}u)\|_{L^p} \leq \frac{1}{j+1-\beta-\frac1p}\|s^\beta\ds^j u\|_{L^p}.
\]
Particularly, $\|\mathscr{M}u\|_{X^m} \leq 2\|u\|_{X^m}$ for $m=0,1,2,\ldots$. 
\end{corollary}

\begin{proof}
In view of \eqref{derM}, the estimate follows directly from Lemma \ref{lem:EstM}. 
\end{proof}

\section{Equivalence of the systems}\label{sect:equiv}
In this section we prove Theorem \ref{th:equiv_0}, which ensures the equivalence of the original system \eqref{Eq}--\eqref{IC} 
and the transformed system \eqref{HP}, \eqref{BVP}, and \eqref{IC} under the stability condition \eqref{SolClass0_SC}.

\subsection{Auxiliary estimates for solutions}
Taking into account the class \eqref{SolClass0}, we assume that 
\begin{equation}\label{SolClassEst0}
 \|\bm{x}'(t)\|_{X^2}+\|\dot{\bm{x}}'(t)\|_{X^1} \leq M \quad\mbox{for}\quad 0\leq t\leq T
\end{equation}
with a positive constant $M$. 
Particularly, we have 
\[
\|\bm{x}''(t)\|_{L^2}, \|s\bm{x}'''(t)\|_{L^2}, \|\dot{\bm{x}}'(t)\|_{L^2}, \|s^\frac12\dot{\bm{x}}''(t)\|_{L^2} \leq M
 \quad\mbox{for}\quad 0\leq t\leq T. 
\]

\begin{lemma}\label{lem:EstSol1}
Let $\bm{x}$ satisfy \eqref{SolClassEst0}. 
Then, we have 
\[
\begin{cases}
 \|s^{\frac12-\frac1p}\bm{x}''(t)\|_{L^p} \leq C(M) &\mbox{for}\quad 2\leq p\leq\infty, \\
 \|\dot{\bm{x}}'(t)\|_{L^q} \leq C(q,M) &\mbox{for}\quad 1\leq q<\infty, \\
 \|s^\epsilon\dot{\bm{x}}'(t)\|_{L^\infty} \leq C(\epsilon,M) &\mbox{for}\quad \epsilon>0.
\end{cases}
\]
\end{lemma}

\begin{proof}
The first estimate follows from Lemma \ref{lem:CalIneqLp1}, and the last two estimates follow from Lemma \ref{lem:embedding}. 
\end{proof}

\begin{lemma}\label{lem:EstSol2}
Let $M$ be a positive constant and $q\in[1,\infty)$. 
There exist positive constants $C_1=C(M)$ and $C_q=C(M,q)$ such that if $\bm{x}$ satisfies \eqref{SolClassEst0}, 
then the solution $\tau$ to the boundary value problem \eqref{BVP} satisfies 
\[
\tau(s,t) \leq C_1s, \quad |\tau'(s,t)| \leq C_1, \quad  \|\tau''(t)\|_{L^q} \leq C_q
\]
for any $(s,t)\in[0,1]\times[0,T]$. 
\end{lemma}

\begin{proof}
The first two estimates follow from Lemma \ref{lem:EstSolBVP1}. 
As for the last one, we see that 
\begin{align*}
\|\tau''(t)\|_{L^q}
&\leq \|\tau(t)|\bm{x}''(t)|^2\|_{L^q} + \||\dot{\bm{x}}'(t)|^2\|_{L^q} \\
&\leq C_1\|s^\frac12\bm{x}''(t)\|_{L^{2q}}^2 + \|\dot{\bm{x}}'(t)\|_{L^{2q}}^2,
\end{align*}
which together with Lemma \ref{lem:EstSol1} gives the last estimate. 
\end{proof}

\begin{lemma}\label{lem:EstSol3}
For any positive constant $M$ there exists a positive constant $C_1=C(M)$ such that 
for any solution $(\bm{x},\tau)$ to the problem \eqref{Eq}--\eqref{BC} satisfying \eqref{SolClassEst0} 
we have $\|\ddot{\bm{x}}'(t)\|_{L^2} \leq C_1$ for $0\leq t\leq T$. 
\end{lemma}

\begin{proof}
We remind here that $\tau$ is also the solution to the boundary value problem \eqref{BVP}. 
It follows from the hyperbolic equations for $\bm{x}$ that $\ddot{\bm{x}}'=(\tau\bm{x}')'' = \tau\bm{x}'''+2\tau'\bm{x}''+\tau''\bm{x}'$, 
so that by Lemma \ref{lem:EstSol2} we have 
$|\ddot{\bm{x}}'| \lesssim s|\bm{x}'''|+|\bm{x}''|+|\tau''|$, which yields the desired estimate. 
\end{proof}

\begin{lemma}\label{lem:EstSol4}
For any positive constant $M$ there exists a positive constant $C_1=C(M)$ such that 
for any solution $(\bm{x},\tau)$ to the problem \eqref{Eq}--\eqref{BC} satisfying \eqref{SolClassEst0} we have 
\[
|\dot{\tau}(s,t)| \leq C_1s, \quad |\dot{\tau}'(s,t)| \leq C_1
\]
for any $(s,t)\in[0,1]\times[0,T]$. 
\end{lemma}

\begin{proof}
Note that $\dot{\tau}$ is a solution to the boundary value problem 
\[
\begin{cases}
 -\dot{\tau}''+|\bm{x}''|^2\dot{\tau} = h_1 \quad\mbox{in}\quad (0,1), \\
 \dot{\tau}(0)=0, \quad \dot{\tau}'(1)=a_1,
\end{cases}
\]
where $h_1=2(\dot{\bm{x}}'\cdot\ddot{\bm{x}}')-2(\bm{x}''\cdot\dot{\bm{x}}'')\tau$ and $a_1=-\bm{g}\cdot\dot{\bm{x}}'|_{s=1}$. 
By Lemmas \ref{lem:EstSol2} and \ref{lem:EstSol3}, we see that 
$\|h_1(t)\|_{L^1} \lesssim \|\dot{\bm{x}}'(t)\|_{L^2}\|\ddot{\bm{x}}'(t)\|_{L^2} + \|\bm{x}''(t)\|_{L^2}\|s\dot{\bm{x}}''(t)\|_{L^2} \lesssim1$. 
Moreover, by the standard Sobolev embedding theorem we have $|a_1|\lesssim \|\dot{\bm{x}}'(t)\|_{X^1} \lesssim 1$. 
Therefore, the desired estimates follow from Lemma \ref{lem:EstSolBVP2} in the case $\alpha=0$. 
\end{proof}

\subsection{Proof of Theorem \ref{th:equiv_0}}
We now give a proof of Theorem \ref{th:equiv_0}. 
Suppose that $(\bm{x},\tau)$ is a solution to the transformed system \eqref{HP}, \eqref{BVP}, and \eqref{IC} in the class \eqref{SolClass0} 
satisfying the stability condition \eqref{SolClass0_SC}. 
Then, there exist a positive constant $M$ such that \eqref{SolClassEst0} holds. 
Therefore, we can use the estimates for $\tau$ and $\dot{\tau}$ obtained in Lemmas \ref{lem:EstSol2} and \ref{lem:EstSol4}. 
We will use such estimates freely in the following.

Put $h(s,t)=|\bm{x}'(s,t)|^2-1$. 
It is sufficient to show that $h(s,t)\equiv0$. 
By a straightforward calculation, we have 
\[
\begin{cases}
 \ddot{h} = \tau h'' + 2\tau' h' + 2\tau''h &\mbox{in}\quad (0,1)\times(0,T), \\
 \tau h' + 2\tau'h = 0 &\mbox{on}\quad \{s=1\}\times(0,T), \\
 (h,\dot{h})|_{t=0} = (0,0) &\mbox{in}\quad (0,1).
\end{cases}
\]
Therefore, 
\begin{align*}
\frac{\mathrm{d}}{\mathrm{d}t}\int_0^1(\tau|\dot{h}|^2+\tau^2|h'|^2) \mathrm{d}s
&= \int_0^1(4\tau\tau''h\dot{h}+\dot{\tau}|\dot{h}|^2+2\tau\dot{\tau}|h'|^2) \mathrm{d}s
 + 2(\tau^2h'\dot{h})|_{s=1},
\end{align*}
where we used the hyperbolic equation for $h$, integration by parts, and the boundary condition $\tau|_{s=0}=0$. 
As for the boundary term, by using the boundary condition for $h$ at $s=1$ we see that 
\begin{align*}
2(\tau^2h'\dot{h})|_{s=1}
&= -4(\tau\tau'h\dot{h})|_{s=1} \\
&= -2\frac{\mathrm{d}}{\mathrm{d}t}(\tau\tau'h^2)|_{s=1} + 2((\dot{\tau}\tau'+\tau\dot{\tau}')h^2)|_{s=1}. 
\end{align*}
Putting $E(t)=\int_0^1( \lambda s|h|^2+\tau|\dot{h}|^2+\tau^2|h'|^2 ) \mathrm{d}s + 2(\tau\tau'h^2)|_{s=1}$, 
where $\lambda>0$ is a parameter to be chosen below, we obtain 
\begin{align*}
\frac{\mathrm{d}}{\mathrm{d}t}E(t)
&= \int_0^1( 2\lambda sh\dot{h} + 4\tau\tau''h\dot{h}+\dot{\tau}|\dot{h}|^2+2\tau\dot{\tau}|h'|^2) \mathrm{d}s + 2((\dot{\tau}\tau'+\tau\dot{\tau}')h^2)|_{s=1}. 
\end{align*}
As in the proof of Lemma \ref{lem:CalIneqLp1}, we have 
$|h(1)|^2 \leq \|sh\|_{L^\infty}^2 \lesssim \|s^\frac12h\|_{L^2}^2 + \|s^\frac12h\|_{L^2}\|sh\|_{L^2}$, so that by taking $\lambda$ sufficiently large, 
we have the equivalence $E(t) \simeq \|s^\frac12 h(t)\|_{L^2}^2 + \|s^\frac12 \dot{h}(t)\|_{L^2}^2 + \|s h'(t)\|_{L^2}^2$. 
Therefore, we see easily that 
\[
\int_0^1( 2\lambda sh\dot{h}+\dot{\tau}|\dot{h}|^2+2\tau\dot{\tau}|h'|^2) \mathrm{d}s + 2((\dot{\tau}\tau'+\tau\dot{\tau}')h^2)|_{s=1} \lesssim E(t).
\]
To evaluate the integral of $\tau\tau''h\dot{h}$, we observe that 
\begin{align*}
|h(s,t)| 
&\leq |h(1,t)|+\int_s^1|h'(\sigma,t)|\mathrm{d}\sigma \\
&\leq |h(1,t)|+\left(\int_s^1\sigma^{-2}\mathrm{d}\sigma\right)^{\frac12}\left(\int_s^1\sigma^2|h'(\sigma,t)|^2\mathrm{d}\sigma\right)^{\frac12} \\
&\lesssim \|\sigma^{\frac12}h\|_{L^2} + s^{-\frac12}\|\sigma h'\|_{L^2},
\end{align*}
where we used $|h(1)|^2 \lesssim \|\sigma^\frac12h\|_{L^2} + \|\sigma h\|_{L^2}$. 
Therefore, 
\begin{align*}
\int_0^1|\tau\tau''h\dot{h}| \mathrm{d}s
&\lesssim \int_0^1 |\tau''|(s^{\frac12}|h|)(s^{\frac12}|\dot{h}|) \mathrm{d}s \\
&\lesssim \|\tau''\|_{L^2}(\|s^{\frac12}h\|_{L^2}+\|s h'\|_{L^2})\|s^{\frac12}\dot{h}\|_{L^2} \\
&\lesssim E(t).
\end{align*}
Summarizing the above estimates, we get $\frac{\mathrm{d}}{\mathrm{d}t}E(t) \lesssim E(t)$, 
which together with the initial condition $E(0)=0$ and Gronwall's inequality yields the desired result. 
\hfill$\Box$

\section{Energy estimate for a linearized system}\label{sect:EE}
%
\subsection{Differential operator $\mathscr{A}_\tau$}
We introduce a variable coefficient differential operator $\mathscr{A}_\tau$ by 
\begin{equation}\label{defA}
\mathscr{A}_\tau\bm{x}=-(\tau\bm{x}')',
\end{equation}
so that \eqref{Eq} can be written as 
\[
\begin{cases}
 \ddot{\bm{x}}+\mathscr{A}_\tau\bm{x}=\bm{g} &\mbox{in}\quad (0,1)\times(0,T), \\
 |\bm{x}'|=1 &\mbox{in}\quad (0,1)\times(0,T).
\end{cases}
\]

\begin{lemma}\label{lem:EstA}
For any $C_1\geq1$ there exists a constant $C\geq1$ such that if $\tau(s)$ satisfies 
\[
 C_1^{-1}s \leq \tau(s) \leq C_1s, \quad |\tau'(s)| \leq C_1
\]
for any $s\in[0,1]$, then we have an equivalence of the norms 
\[
C^{-1}( \|s\bm{x}''\|_{L^2}+\|\bm{x}'\|_{L^2} ) \leq \|\mathscr{A}_\tau\bm{x}\|_{L^2} \leq C( \|s\bm{x}''\|_{L^2}+\|\bm{x}'\|_{L^2}). 
\]
\end{lemma}

\begin{proof}
It is sufficient to show the first estimate. 
Moreover, in view of $\|s\bm{x}''\|_{L^2} \leq C_1\|\tau\bm{x}''\|_{L^2} \leq C_1 (\|(\tau\bm{x}')'\|_{L^2}+\|\tau'\bm{x}'\|_{L^2})$, 
it is sufficient to show $\|\bm{x}'\|_{L^2} \lesssim\|\mathscr{A}_\tau\bm{x}\|_{L^2}$. 
Integrating the definition \eqref{defA} of $\mathscr{A}_\tau$ over $[0,s]$, we have 
\[
\tau(s)\bm{x}'(s) = -\int_0^s (\mathscr{A}_\tau\bm{x})(\sigma) \mathrm{d}\sigma = - s\mathscr{M}(\mathscr{A}_\tau\bm{x})(s),
\]
where $\mathscr{M}$ is the averaging operator defined in Section \ref{sect:AOM}. 
Therefore, $|\bm{x}'(s)| \leq C_1|\mathscr{M}(\mathscr{A}_\tau\bm{x})(s)|$ for any $s\in[0,1]$. 
Now, the desired estimate follows from Corollary \ref{cor:WEM1} in the case $\beta=j=0$ and $p=2$. 
\end{proof}

\subsection{Linearized system}
In this subsection we derive an energy estimate for solutions to a linearized system for \eqref{Eq}, \eqref{BC}, and \eqref{BVP}. 
We denote variations of $(\bm{x},\tau)$ by $(\bm{y},\nu)$ in the linearization. 
Then, the linearized system has the form 
\begin{equation}\label{LEq}
\begin{cases}
 \ddot{\bm{y}}+\mathscr{A}_\tau\bm{y}+\mathscr{A}_\nu\bm{x} = \bm{f} &\mbox{in}\quad (0,1)\times(0,T), \\
 \bm{x}'\cdot\bm{y}'=f &\mbox{in}\quad (0,1)\times(0,T), \\
 \bm{y}=\bm{0} &\mbox{on}\quad \{s=1\}\times(0,T),
\end{cases}
\end{equation}
and 
\begin{equation}\label{LBVP}
\begin{cases}
 -\nu''+|\bm{x}''|^2\nu = 2\dot{\bm{x}}'\cdot\dot{\bm{y}}' - 2(\bm{x}''\cdot\bm{y}'')\tau + h &\mbox{in}\quad (0,1)\times(0,T), \\
 \nu = 0 &\mbox{on}\quad \{s=0\}\times(0,T), \\
 \nu' = -\bm{g}\cdot\bm{y}' &\mbox{on}\quad \{s=1\}\times(0,T),
\end{cases}
\end{equation}
where $\bm{f}$, $f$, and $h$ can be regarded as given functions. 
As for $\bm{x}$, in addition to \eqref{SolClassEst0} we assume that 
\begin{equation}\label{SolClassEst2}
\|\bm{x}(t)\|_{X^3} + \|\dot{\bm{x}}(t)\|_{X^2} \leq M_1.
\end{equation}
Note that this condition is less restrictive than \eqref{SolClassEst0}. 
An advantage to assume this condition is that we can take the constant $M_1$ smaller than the constant $M$ in \eqref{SolClassEst0} in applications. 
We are going to evaluate the functional $E(t)$ defined by 
\begin{equation}\label{DefE}
E(t) = \|\dot{\bm{y}}(t)\|_{X^1}^2 + \|\bm{y}(t)\|_{X^2}^2.
\end{equation}

\begin{proposition}\label{prop:EE}
For any positive constants $M_1$, $M$, and $c_0$ and any $\epsilon \in (0,\frac12)$, 
there exist positive constants $C_1=C(M_1,c_0)$ and $C_2(\epsilon)=C(M,M_1,c_0,\epsilon)$ such that if $(\bm{x},\tau)$ is a solution to the problem 
\eqref{Eq}--\eqref{BC} satisfying \eqref{SolClassEst0}, \eqref{SolClassEst2}, and the stability condition \eqref{SolClass0_SC}, 
then for any solution $(\bm{y},\nu)$ to \eqref{LEq}--\eqref{LBVP} we have 
\[
E(t) \leq C_1 \mathrm{e}^{C_2(\epsilon) t}\left( E(0) + S_1(0) + C_2(\epsilon)\int_0^t S_2(t')\mathrm{d}t' \right), 
\]
where 
\begin{equation}\label{DefS12}
\begin{cases}
 S_1(t) = \|\bm{f}\|_{L^2}^2 + \|sh\|_{L^1}^2,\\
 S_2(t) = \|\dot{\bm{f}}\|_{L^2}^2 + \|s^{\frac12-\epsilon}\dot{f}\|_{L^2}^2 + |\dot{f}|_{s=1}|^2
 + \|s\dot{h}\|_{L^1}^2 + \|s^{\frac12+\epsilon}h\|_{L^2}^2.
\end{cases}
\end{equation}
\end{proposition}

\begin{remark}\label{re:EE}
The linearized system \eqref{LEq} is overdetermined due to the second equation in \eqref{LEq}. 
Therefore, it is natural to ask if one could obtain a similar energy estimate as above without using the second equation. 
The answer is affirmative if we impose further regularities on $\bm{x}$ in addition to \eqref{SolClassEst0} and \eqref{SolClassEst2}. 
In other words, thanks to the almost orthogonality condition $\bm{x}'\cdot\bm{y}'=f$ we can minimize the regularity imposed on $\bm{x}$. 
For more details, we refer to Iguchi and Takayama \cite{IguchiTakayama2023}. 
\end{remark}

\begin{proof}[Proof of Proposition \ref{prop:EE}]
By taking $L^2$-inner product of the first equation in \eqref{LEq} with $\mathscr{A}_{\tau}\dot{\bm{y}}$ and using integration by parts, 
we have 
\begin{align*}
\frac{\mathrm{d}}{\mathrm{d}t}\{ \|\tau^\frac12\dot{\bm{y}}'\|_{L^2}^2 + \|\mathscr{A}_{\tau}\bm{y}\|_{L^2}^2 \}
&= (\dot{\tau}\dot{\bm{y}}',\dot{\bm{y}}')_{L^2} + 2(\mathscr{A}_{\dot{\tau}}\bm{y},\mathscr{A}_{\tau}\bm{y})_{L^2} 
  + 2(\mathscr{A}_{\tau}\dot{\bm{y}},\bm{f})_{L^2} - 2(\mathscr{A}_{\tau}\dot{\bm{y}}, \mathscr{A}_{\nu}\bm{x})_{L^2},
\end{align*}
where the boundary terms are vanished due to the boundary conditions. 
In view of Lemmas \ref{lem:EstA}, \ref{lem:EstSol2}, and \ref{lem:EstSol4}, the first two terms in the right-hand side can be easily 
handled so that we will focus on the last two terms. 
We evaluate the third term as 
$(\mathscr{A}_{\tau}\dot{\bm{y}},\bm{f})_{L^2} = \frac{\rm d}{{\rm d}t}(\mathscr{A}_{\tau}\bm{y},\bm{f})_{L^2}
 - (\mathscr{A}_{\dot{\tau}}\bm{y},\bm{f})_{L^2} - (\mathscr{A}_{\tau}\bm{y},\dot{\bm{f}})_{L^2}$. 
As for the last term, by integration by parts, we see that 
\begin{align*}
(\mathscr{A}_{\tau}\dot{\bm{y}}, \mathscr{A}_{\nu}\bm{x})_{L^2}
&= ((\tau\dot{\bm{y}}')',(\nu\bm{x}')')_{L^2} \\
&= -(\tau\dot{\bm{y}}',(\nu\bm{x}')'')_{L^2} + ( \tau\dot{\bm{y}}'\cdot(\nu\bm{x}')' )|_{s=1} \\
&= -(\tau\dot{\bm{y}}',\nu''\bm{x}'+2\nu'\bm{x}''+\nu\bm{x}''')_{L^2}
 + ( \tau\dot{\bm{y}}'\cdot(\nu\bm{x}''+\nu'\bm{x}') )|_{s=1}.
\end{align*}
Here, the term $(\tau\dot{\bm{y}}',\nu''\bm{x}')_{L^2}$ would be troublesome if we evaluate it directly. 
To bypass the trouble, we make use of the second equation in \eqref{LEq}. 
Differentiating it with respect to $t$ we have 
\begin{equation}\label{OR}
\bm{x}'\cdot\dot{\bm{y}}'=\dot{f}-\dot{\bm{x}}'\cdot\bm{y}', 
\end{equation}
so that the term can be written as 
$(\tau\dot{\bm{y}}',\nu''\bm{x}')_{L^2}=(\tau\dot{f},\nu'')_{L^2}-(\tau\bm{y}',\nu''\dot{\bm{x}}')_{L^2}$, 
which can now be easily handled. 
Thanks to the relation \eqref{OR} again, we have 
$( (\tau\dot{\bm{y}}')\cdot(\nu'\bm{x}') )|_{s=1} = ( \tau\nu'(\dot{f}-\bm{y}'\cdot\dot{\bm{x}}'))|_{s=1}$, 
which can also be easily handled. 
Therefore, we obtain 
\begin{equation}\label{LEE1}
\frac{\mathrm{d}}{\mathrm{d}t}\{ \|\tau^\frac12\dot{\bm{y}}'\|_{L^2}^2 + \|\mathscr{A}_{\tau}\bm{y}\|_{L^2}^2
 - 2(\mathscr{A}_\tau\bm{y},\bm{f})_{L^2} \}
= I_1 - 2( \nu\tau\bm{x}''\cdot\dot{\bm{y}}' )|_{s=1}, 
\end{equation}
where 
\begin{align*}
I_1 
&= (\dot{\tau}\dot{\bm{y}}',\dot{\bm{y}}')_{L^2} + 2(\mathscr{A}_{\dot{\tau}}\bm{y},\mathscr{A}_{\tau}\bm{y})_{L^2}
 - 2(\mathscr{A}_{\dot{\tau}}\bm{y},\bm{f})_{L^2} - 2(\mathscr{A}_{\tau}\bm{y},\dot{\bm{f}})_{L^2} \\
&\quad\;
 - 2(\tau\bm{y}',\nu''\dot{\bm{x}}')_{L^2} + 2(\tau\dot{\bm{y}}',2\nu'\bm{x}''+\nu\bm{x}''')_{L^2} + 2(\tau\nu'',\dot{f})_{L^2} \\
&\quad\;
 + 2( \tau\nu'(\dot{\bm{x}}'\cdot\bm{y}'-\dot{f}))|_{s=1}.
\end{align*}
The only remaining term that we have to evaluate is $( \nu\tau\bm{x}''\cdot\dot{\bm{y}}' )|_{s=1}$, 
to which we need more careful analysis.

\begin{lemma}\label{lem:BT}
For any positive constants $M$ and $c_0$, 
there exists a positive constant $C_2=C(M,c_0)$ such that 
if $(\bm{x},\tau)$ is a solution to the problem \eqref{Eq}--\eqref{BC} satisfying \eqref{SolClassEst0} and the stability condition \eqref{SolClass0_SC}, 
then for any solution $(\bm{y},\nu)$ to \eqref{LEq}--\eqref{LBVP} we have 
\begin{equation}\label{LEE2}
- 2( \nu\tau\bm{x}''\cdot\dot{\bm{y}}' )|_{s=1}
= \frac{\mathrm{d}}{\mathrm{d}t}\left( \frac{\varphi'}{\varphi}\nu^2 \right)\biggr|_{s=1} + I_2,
\end{equation}
where $\varphi$ is the fundamental solution defined by \eqref{IVPphi} and $I_2$ satisfies 
\begin{align*}
|I_2| \leq C_2( \|\tau^\frac12\dot{\bm{y}}'\|_{L^2}^2 + \|\mathscr{A}_{\tau}\bm{y}\|_{L^2}^2
 + \|\nu\|_{L^\infty}^2 + \|\nu'\|_{L^2}^2  + \|\bm{f}\|_{L^2}^2 + |\dot{f}|_{s=1}|^2 + \|s\dot{h}\|_{L^1}^2 ).
\end{align*}
\end{lemma}

\begin{proof}
By taking the trace of the hyperbolic equations for $\bm{x}$ on $s=1$ and using the boundary condition, 
we have $(\tau\bm{x}''+\tau'\bm{x}')|_{s=1}+\bm{g}=\bm{0}$. 
Therefore, by \eqref{OR} again 
\begin{equation}\label{LI21}
\bm{g}\cdot\dot{\bm{y}}'+\tau(\bm{x}''\cdot\dot{\bm{y}}')=I_{2,1} \quad\mbox{on}\quad s=1,
\end{equation}
where $I_{2,1} = (\tau'( \dot{\bm{x}}'\cdot\bm{y}' - \dot{f} ))|_{s=1}$. 
Differentiating \eqref{LBVP} with respect to $t$, we see that $\dot{\nu}$ satisfies 
\[
\begin{cases}
 -\dot{\nu}''+|\bm{x}''|^2\dot{\nu} = h_I + h_{II}' &\mbox{in}\quad (0,1)\times(0,T), \\
 \dot{\nu} = 0 &\mbox{on}\quad \{s=0\}\times(0,T), \\
 \dot{\nu}' = -\bm{g}\cdot\dot{\bm{y}}' &\mbox{on}\quad \{s=1\}\times(0,T),
\end{cases}
\]
where 
\begin{align*}
h_I &= 2(\ddot{\bm{x}}'\cdot\dot{\bm{y}}') - 2(\dot{\bm{x}}''\cdot\bm{y}'')\tau - 2(\bm{x}''\cdot\bm{y}'')\dot{\tau}
 - 2(\bm{x}''\cdot\dot{\bm{x}}'')\nu - 2(\dot{\bm{x}}''\cdot\ddot{\bm{y}}) + 2(\tau\bm{x}'')'\cdot\dot{\bm{y}}' + \dot{h}, \\
h_{II} &= 2(\dot{\bm{x}}'\cdot\ddot{\bm{y}}) - 2(\bm{x}''\cdot\dot{\bm{y}}')\tau.
\end{align*}
Now, we use the solution formula \eqref{SolBVP} to $\dot{\nu}$ at $s=1$ to obtain 
\[
\dot{\nu}=\frac{\varphi}{\varphi'}(-\bm{g}\cdot\dot{\bm{y}}'+h_{II})
 + \frac{1}{\varphi'}\int_0^1( \varphi(\sigma)h_I(\sigma)-\varphi'(\sigma)h_{II}(\sigma))\mathrm{d}\sigma
 \quad\mbox{on}\quad s=1,
\]
where we used integration by parts. 
Since $h_{II}=- 2(\bm{x}''\cdot\dot{\bm{y}}')\tau$ on $s=1$, we get 
\begin{equation}\label{LI22}
\bm{g}\cdot\dot{\bm{y}}'+2\tau(\bm{x}''\cdot\dot{\bm{y}}')=-\frac{\varphi'}{\varphi}\dot{\nu} + I_{2,2}
 \quad\mbox{on}\quad s=1,
\end{equation}
where 
\[
I_{2,2} = \frac{1}{\varphi|_{s=1}}\int_0^1( \varphi_1(\sigma)h_I(\sigma)-\varphi_1'(\sigma)h_{II}(\sigma))\mathrm{d}\sigma.
\]
By \eqref{LI21} and \eqref{LI22}, we obtain \eqref{LEE2} with 
\[
I_2 = \left[ -\left\{ \frac{\mathrm{d}}{\mathrm{d}t}\left( \frac{\varphi'}{\varphi} \right) \right\}\nu^2
 + 2\nu(I_{2,1}-I_{2,2}) \right]\biggr|_{s=1}.
\]

It remains to show the estimate for $I_2$ of the lemma. 
Since $\bm{x}$ is supposed to satisfy \eqref{SolClassEst0}, we will use freely the estimates in 
Lemmas \ref{lem:EstPhi}, \ref{lem:EstDtPhi}, \ref{lem:EstSol1}--\ref{lem:EstSol4} without any comment. 
Then, we have $|I_2| \lesssim \|\nu\|_{L^\infty}^2 + \|\nu\|_{L^\infty}( |I_{2,1}|+|I_{2,2}| )$. 
By the standard Sobolev embedding theorem, we have 
$|I_{2,1}| \lesssim |\bm{y}'|_{s=1}| + |\dot{f}|_{s=1}| \lesssim \|\bm{y}'\|_{L^2}+\|s\bm{y}''\|_{L^2}+|\dot{f}|_{s=1}|$. 
We have also 
\begin{align*}
|I_{2,2}| 
&\lesssim \|s h_I\|_{L^1}+\|h_{II}\|_{L^1} \\
&\lesssim \|s^\frac12\dot{\bm{y}}'\|_{L^2}+\|s\bm{y}''\|_{L^2} + \|\ddot{\bm{y}}\|_{L^2} + \|\nu\|_{L^\infty} + \|s\dot{h}\|_{L^1}.
\end{align*}
It follows from the first equation in \eqref{LEq} for $\bm{y}$ that 
$\|\ddot{\bm{y}}\|_{L^2} \leq \|\mathscr{A}_{\tau}\bm{y}\|_{L^2} + \|\mathscr{A}_{\nu}\bm{x}\|_{L^2} + \|\bm{f}\|_{L^2}$. 
Here, we have 
$\|\mathscr{A}_{\nu}\bm{x}\|_{L^2} \leq \|\nu\|_{L^\infty}\|\bm{x}''\|_{L^2} + \|\nu'\|_{L^2}
 \lesssim \|\nu\|_{L^\infty} + \|\nu'\|_{L^2}$. 
These estimates together with Lemma \ref{lem:EstA} yield the desired estimate. 
\end{proof}

We need to evaluate $\nu$ in terms of $\bm{y}$, which will be given in the next lemma.

\begin{lemma}\label{lem:EstNu}
Under the same hypotheses and the same notations in Lemma \ref{lem:BT}, we have 
\[
\begin{cases}
 \|s^{-\frac12}\nu\|_{L^\infty} + \|s^\frac12\nu'\|_{L^\infty} + \|\nu'\|_{L^2}
  \leq C_1( \|\tau^\frac12\dot{\bm{y}}'\|_{L^2} + \|\mathscr{A}_{\tau}\bm{y}\|_{L^2} + \|s^\frac12h\|_{L^1} ), \\
 \|s^{\frac12+\epsilon}\nu''\|_{L^2}
  \leq C_1(\epsilon)( \|\tau^\frac12\dot{\bm{y}}'\|_{L^2} + \|\mathscr{A}_{\tau}\bm{y}\|_{L^2} ) + \|s^{\frac12+\epsilon}h\|_{L^2},
\end{cases}
\]
where $\epsilon\in(0,\frac12)$ is arbitrary and 
$C_1(\epsilon)=C(M,c_0,\epsilon)$. 
Particularly, we have 
\[
\|\ddot{\bm{y}}\|_{L^2} \leq C_1( \|\tau^\frac12\dot{\bm{y}}'\|_{L^2} + \|\mathscr{A}_{\tau}\bm{y}\|_{L^2} 
 + \|\bm{f}\|_{L^2} + \|s^\frac12h\|_{L^1}). 
\]
\end{lemma}

\begin{proof}
We note that $\nu$ is the solution to the boundary value problem \eqref{LBVP} and that we have 
\[
\begin{cases}
 \|s^\frac12( \dot{\bm{x}}'\cdot\dot{\bm{y}}'-(\bm{x}''\cdot\bm{y}'')\tau )\|_{L^1}
  \lesssim \|s^\frac12\dot{\bm{y}}'\|_{L^2}+\|s\bm{y}''\|_{L^2}, \\
 |\bm{g}\cdot\bm{y}'|_{s=1}| \lesssim \|\bm{y}'\|_{L^2}+\|s\bm{y}''\|_{L^2}.
\end{cases}
\]
Therefore, the first estimate of the lemma follows from Lemmas \ref{lem:EstSolBVP2}, \ref{lem:EstSolBVP3}, and \ref{lem:EstA}. 
Moreover, by the first equation in \eqref{LBVP} we have 
\[
\|s^{\frac12+\epsilon}(\nu''+h)\|_{L^2}
\lesssim \|s^\frac14\bm{x}''\|_{L^4}^2\|s^{-\frac12}\nu\|_{L^\infty}
 + \|s^\epsilon\dot{\bm{x}}'\|_{L^\infty}\|s^\frac12\dot{\bm{y}}'\|_{L^2} + \|s^\frac12\bm{x}''\|_{L^\infty}\|s\bm{y}''\|_{L^2},
\]
which together with the first estimate, Lemmas \ref{lem:EstSol1} and \ref{lem:EstA}, 
and $\|s^\frac12 h\|_{L^1} \leq (1-2\epsilon)^{-\frac12}\|s^{\frac12+\epsilon}h\|_{L^2}$ gives the second one. 
\end{proof}

In view of \eqref{LEE1} and \eqref{LEE2}, we define an energy functional $\mathscr{E}(t)$ by 
\begin{align*}
\mathscr{E}(t) &= \|\tau^\frac12\dot{\bm{y}}'\|_{L^2}^2 + \|\mathscr{A}_{\tau}\bm{y}\|_{L^2}^2 - 2(\mathscr{A}_\tau\bm{y},\bm{f})_{L^2}
 - \left( \frac{\varphi'}{\varphi}\nu^2 \right)\biggr|_{s=1}
 + \lambda( \|\dot{\bm{y}}\|_{L^2}^2 + \|\bm{y}\|_{X^1}^2 ),
\end{align*}
where the parameter $\lambda>0$ will be chosen so large that the following lemma holds.

\begin{lemma}\label{lem:EquiNorm2}
For any positive constants $M_1$ and $c_0$, there exist $\lambda_0=\lambda(M_1,c_0)>0$ and $C_1=C(M_1,c_0)\geq1$ 
such that if $\lambda=\lambda_0$, then we have 
\[
\mathscr{E}(t) \leq C_1( E(t) + S_1(t) ), \qquad E(t) \leq C_1( \mathscr{E}(t) +S_1(t)), 
\]
where $E(t)$ and $S_1(t)$ are defined by \eqref{DefE} and \eqref{DefS12}. 
\end{lemma}

\begin{proof}
It is sufficient to evaluate $(\nu|_{s=1})^2$. 
Since $\nu$ is a solution to the boundary value problem \eqref{LBVP}, the solution formula \eqref{SolBVP} yields 
\begin{align*}
\nu 
&= -\frac{\varphi}{\varphi'}( \bm{g}\cdot\bm{y}' + 2\tau(\bm{x}''\cdot\bm{y}') )
 + \frac{1}{\varphi'} \int_0^1 ( 2(\varphi\tau\bm{x}'')'\cdot\bm{y}' - 2(\varphi\dot{\bm{x}}')'\cdot\dot{\bm{y}} + \varphi h)\mathrm{d}\sigma
 \quad\mbox{on}\quad s=1,
\end{align*}
where we used integration by parts. 
This together with Lemmas \ref{lem:EstPhi} and \ref{lem:EstSolBVP1} gives 
\begin{align*}
|\nu|_{s=1}|
&\leq C_1(1+|\bm{x}''|_{s=1}|)|\bm{y}'|_{s=1}|
 + C_1(\|s ^\frac12\bm{x}''\|_{L^2}+\|s^\frac32\bm{x}'''\|_{L^2})\|s^\frac12\bm{y}'\|_{L^2} \\
&\quad\;
 + C_1(\|\dot{\bm{x}}'\|_{L^2}+\|s\dot{\bm{x}}''\|_{L^2})\|\dot{\bm{y}}\|_{L^2} + \|sh\|_{L^1} \\
&\leq C_1(\|s^\frac12\bm{y}'\|_{L^2}^\frac12\|s\bm{y}''\|_{L^2}^\frac12 + \|s^\frac12\bm{y}'\|_{L^2} + \|\dot{\bm{y}}\|_{L^2} + \|sh\|_{L^1},
\end{align*}
where we used the Sobolev embedding theorem $\|u\|_{L^\infty} \lesssim \|u\|_{L^2}^\frac12\|u'\|_{L^2}^\frac12 + \|u\|_{L^2}$. 
Therefore, by Lemma \ref{lem:EstA} we obtain the desired equivalence. 
\end{proof}

We go back to the proof of Proposition \ref{prop:EE}. 
We fix the parameter $\lambda$ in the energy functional $\mathscr{E}(t)$ as $\lambda=\lambda_0$. 
It follows form \eqref{LEE1} and \eqref{LEE2} that $\frac{\mathrm{d}}{\mathrm{d}t}\mathscr{E}(t)
 = I_1+I_2 + 2\lambda_0\{(\dot{\bm{y}},\ddot{\bm{y}})_{L^2}+(s^\frac12\bm{y}',s^\frac12\dot{\bm{y}}')_{L^2}+(\bm{y},\dot{\bm{y}})_{L^2} \}$. 
Since $\bm{x}$ is supposed to satisfy \eqref{SolClassEst0}, 
we will use freely the estimates in Lemmas \ref{lem:EstPhi}, \ref{lem:EstDtPhi}, \ref{lem:EstSol1}--\ref{lem:EstSol4} without any comment. 
Then, by Lemmas \ref{lem:EstA} and \ref{lem:EstNu} we see that 
\begin{align*}
|I_1|
&\lesssim \|\tau^\frac12\dot{\bm{y}}'\|_{L^2}^2 + \|\mathscr{A}_\tau\bm{y}\|_{L^2}^2 + \|(\bm{f},\dot{\bm{f}})\|_{L^2}^2
 + \|\bm{y}'\|_{L^2} \|s^\frac34\nu''\|_{L^2} \|s^\frac14\dot{\bm{x}}'\|_{L^\infty} \\
&\quad\;
 + \|\tau^\frac12\dot{\bm{y}}'\|_{L^2}( \|s^\frac12\nu'\|_{L^\infty}\|\bm{x}''\|_{L^2} + \|s^{-\frac12}\nu\|_{L^\infty}\|s\bm{x}'''\|_{L^2}) \\
&\quad\;
 + \|s^{\frac12+\epsilon}\nu''\|_{L^2}\|s^{\frac12-\epsilon}\dot{f}\|_{L^2}
 + \|s^\frac12\nu'\|_{L^\infty}(|\bm{y}'|_{s=1}| + |\dot{f}|_{s=1}| ) \\
&\lesssim \|\tau^\frac12\dot{\bm{y}}'\|_{L^2}^2 + \|\mathscr{A}_\tau\bm{y}\|_{L^2}^2 + S_1(t) + S_2(t),
\end{align*}
where we used again $\|s^\frac12h\|_{L^1} \leq (1-2\epsilon)^{-\frac12}\|s^{\frac12+\epsilon}h\|_{L^2}$. 
Similarly, by Lemma \ref{lem:BT} we obtain $|I_2| \leq C_2( E(t) + S_1(t) + S_2(t))$, 
so that $\frac{\mathrm{d}}{\mathrm{d}t}\mathscr{E}(t) \leq C_2(\epsilon)( E(t)+S_1(t)+S_2(t))$. 
Moreover, we see easily that $S_1(t) \leq 2 S_1(0) + 4t\int_0^tS_2(t')\mathrm{d}t'$. 
These together with Lemma \ref{lem:EquiNorm2} and Gronwall's inequality gives the desired estimate. 
\end{proof}

\section{Uniqueness of solutions}\label{sect:unique}
In this section we prove Theorems \ref{th:unique_g<>0} and \ref{th:unique_g==0}, 
which ensures the uniqueness of solutions to the initial boundary value problem \eqref{Eq}--\eqref{IC}.

\subsection{Proof of Theorem \ref{th:unique_g<>0}}
In this subsection, we will show the uniqueness of solutions to the problem \eqref{Eq}--\eqref{IC} under the stability condition \eqref{SolClass0_SC}. 
Suppose that $(\bm{x}_1,\tau_1)$ and $(\bm{x}_2,\tau_2)$ are solutions to the problem in the class \eqref{SolClass0} 
satisfying the stability condition \eqref{SolClass0_SC}. 
Then, we can assume that both $\bm{x}_1$ and $\bm{x}_2$ satisfy \eqref{SolClassEst0} with a positive constant $M$. 
Therefore, we will use freely the estimates in Lemmas \ref{lem:EstPhi}, \ref{lem:EstDtPhi}, \ref{lem:EstSol1}--\ref{lem:EstSol4} without any comment. 
Putting $\bm{y}=\bm{x}_1-\bm{x}_2$, $\nu=\tau_1-\tau_2$, $\bm{x}=\frac12(\bm{x}_1+\bm{x}_2)$, $\tau=\frac12(\tau_1+\tau_2)$, 
we are going to show that $\bm{y}=\bm{0}$ and $\nu=0$. 
We see that $\bm{y}$ satisfies 
\begin{equation}\label{EqDiff1}
\begin{cases}
 \ddot{\bm{y}}+\mathscr{A}_{\tau}\bm{y}+\mathscr{A}_\nu\bm{x}=\bm{0} &\mbox{in}\quad (0,1)\times(0,T), \\
 \bm{x}'\cdot\bm{y}'=0 &\mbox{in}\quad (0,1)\times(0,T), \\
 \bm{y}=\bm{0} &\mbox{on}\quad \{s=1\}\times(0,T), \\
 (\bm{y},\dot{\bm{y}})=(\bm{0},\bm{0}) &\mbox{on}\quad (0,1)\times\{t=0\},
\end{cases}
\end{equation}
and that $\nu$ satisfies 
\begin{equation}\label{EqDiff2}
\begin{cases}
 -\nu''+|\bm{x}''|^2\nu = 2(\dot{\bm{x}}'\cdot\dot{\bm{y}}')-2(\bm{x}''\cdot\bm{y}'')\tau + h &\mbox{in}\quad (0,1)\times(0,T), \\
 \nu = 0 &\mbox{on}\quad \{s=0\}\times(0,T), \\
 \nu' = -\bm{g}\cdot\bm{y}' &\mbox{on}\quad \{s=1\}\times(0,T),
\end{cases}
\end{equation}
where $h=-\frac14|\bm{y}''|^2\nu$. 
Although $\bm{x}$ satisfies \eqref{SolClassEst0}, $(\bm{x},\tau)$ is not a solution to the problem \eqref{Eq}--\eqref{IC} in general. 
Therefore, we cannot apply the energy estimate obtained in Proposition \ref{prop:EE} directly and we need to modify the estimate. 
In this case, in place of \eqref{LEE1} we have 
\[
\frac{\mathrm{d}}{\mathrm{d}t}\{ \|\tau^\frac12\dot{\bm{y}}'\|_{L^2}^2 + \|\mathscr{A}_{\tau}\bm{y}\|_{L^2}^2 \}
= I_1 - 2( \nu\tau\bm{x}''\cdot\dot{\bm{y}}' )|_{s=1}, 
\]
where 
\begin{align*}
I_1 
&= (\dot{\tau}\dot{\bm{y}}',\dot{\bm{y}}')_{L^2} + 2(\mathscr{A}_{\dot{\tau}}\bm{y},\mathscr{A}_{\tau}\bm{y})_{L^2} \\
&\quad\;
 - 2(\tau\bm{y}',\nu''\dot{\bm{x}}')_{L^2} + 2(\tau\dot{\bm{y}}',2\nu'\bm{x}''+\nu\bm{x}''')_{L^2}
 + 2( \tau\nu'(\dot{\bm{x}}'\cdot\bm{y}'))|_{s=1}.
\end{align*}
To evaluate the boundary term $( \tau\nu'(\dot{\bm{x}}'\cdot\bm{y}'))|_{s=1}$, we showed Lemma \ref{lem:BT}, 
where we used essentially the fact that $(\bm{x},\tau)$ is a solution to the problem \eqref{Eq}--\eqref{IC}. 
Therefore, we need to modify Lemma \ref{lem:BT} as follows.

\begin{lemma}\label{lem:MBT}
It holds that 
\begin{equation}\label{LEE3}
-2( \tau\nu\bm{x}''\cdot\dot{\bm{y}}' )|_{s=1}
= 2\frac{\mathrm{d}}{\mathrm{d}t}\left( \frac{\tau^2}{\tau_1^2+\tau_2^2}\frac{\varphi'}{\varphi}\nu^2 \right)\biggr|_{s=1} + I_2,
\end{equation}
where $\varphi$ is the fundamental solution defined by \eqref{IVPphi} for $\bm{x}$ and $I_2$ satisfies 
\begin{align*}
|I_2| \leq C_2( \|\tau^\frac12\dot{\bm{y}}'\|_{L^2}^2 + \|\mathscr{A}_{\tau}\bm{y}\|_{L^2}^2
 + \|\nu\|_{L^\infty}^2+\|\nu'\|_{L^2}^2 + \|s\dot{h}\|_{L^1}^2).
\end{align*}
\end{lemma}

\begin{proof}
By taking the trace of the hyperbolic equations for $\bm{x}_j$ on $s=1$ and using the boundary condition, we have 
$-(\tau_j\bm{x}_j''+\tau_j'\bm{x}_j')|_{s=1}=\bm{g}$, so that 
\begin{align*}
2\bm{x}''
&= -\left(\frac{1}{\tau_1}+\frac{1}{\tau_2}\right)\bm{g} - \left(\frac{\tau_1'}{\tau_1}+\frac{\tau_2'}{\tau_2}\right)\bm{x}'
 - \frac12\left(\frac{\tau_1'}{\tau_1}-\frac{\tau_2'}{\tau_2}\right)\bm{y}' \quad\mbox{on}\quad s=1.
\end{align*}
Differentiating the second equation in \eqref{EqDiff1} with respect to $t$ we have $\bm{x}'\cdot\dot{\bm{y}}'=-\dot{\bm{x}}'\cdot\bm{y}'$. 
These two identities yield 
\begin{equation}\label{I21}
\frac{\tau}{2}\left(\frac{1}{\tau_1}+\frac{1}{\tau_2}\right)\bm{g}\cdot\dot{\bm{y}}' + \tau(\bm{x}''\cdot\dot{\bm{y}}') = I_{2,1}
 \quad\mbox{on}\quad s=1,
\end{equation}
where 
\[
I_{2,1}=\left\{\frac{\tau}{2}\left(\frac{\tau_2'}{\tau_2}\dot{\bm{x}}_1'+\frac{\tau_1'}{\tau_1}\dot{\bm{x}}_2'\right)
 \cdot\bm{y}'\right\}\biggr|_{s=1}.
\]
On the other hand, \eqref{LI22} is still valid. 
By \eqref{I21} and \eqref{LI22}, we obtain \eqref{LEE3} with 
\[
I_2 = \left[ -\left\{\frac{\mathrm{d}}{\mathrm{d}t}\left( \frac{2\tau^2}{\tau_1^2+\tau_2^2}\frac{\varphi'}{\varphi} \right) \right\}\nu^2
 + 2\nu\left( \frac{2\tau_1\tau_2}{\tau_1^2+\tau_2^2} I_{2,1}-\frac{2\tau^2}{\tau_1^2+\tau_2^2} I_{2,2} \right) \right]\biggr|_{s=1}.
\]
The estimate for $I_2$ is exactly the same as in the proof of Lemma \ref{lem:BT}, so we omit it. 
\end{proof}

Thanks to this lemma, we see that the energy estimate obtained in Proposition \ref{prop:EE} is still valid in this case. 
We use the estimate with, for example, $\epsilon=\frac14$ and obtain 
\[
E(t) \leq C_1\mathrm{e}^{C_2t}\left( E(0) + S_1(0) + C_2\int_0^tS_2(t')\mathrm{d}t' \right),
\]
where $E(t)=\|\dot{\bm{y}}(t)\|_{X^1}^2 + \|\bm{y}(t)\|_{X^2}^2$, $S_1(t)=\|sh(t)\|_{L^1}^2$, and 
$S_2(t)=\|s\dot{h}(t)\|_{L^1}^2+\|s^\frac34h(t)\|_{L^2}^2$. 
Here, we remind that $h$ was defined by $h=-\frac14|\bm{y}''|^2\nu$. 
By the initial conditions in \eqref{EqDiff1}, we have $E(0)=S_1(0)=0$. 
In view of 
\[
\begin{cases}
 |h| \lesssim s(|\bm{x}_1''|+|\bm{x}_2''|)|\bm{y}''|, \\
 |\dot{h}| \lesssim s(|\dot{\bm{x}}_1''|+|\dot{\bm{x}}_2''|+|\bm{x}_1''|+|\bm{x}_2''|)|\bm{y}''|,
\end{cases}
\]
we have 
\[
\begin{cases}
 \|s^\frac12h\|_{L^2} \lesssim ( \|s^\frac12\bm{x}_1''\|_{L^\infty} + \|s^\frac12\bm{x}_2''\|_{L^\infty} ) \|s\bm{y}''\|_{L^2}, \\
 \|s^\frac12\dot{h}\|_{L^1} \lesssim ( \|s^\frac12\dot{\bm{x}}_1''\|_{L^2} + \|s^\frac12\dot{\bm{x}}_2''\|_{L^2}
  + \|\bm{x}_1''\|_{L^2} + \|\bm{x}_2''\|_{L^2} )\|s\bm{y}''\|_{L^2}. 
\end{cases}
\]
Therefore, we obtain $E(t) \lesssim \int_0^tE(t')\mathrm{d}t'$, 
which together with Gronwall's inequality implies $E(t)\equiv0$ so that $\dot{\bm{y}}=\bm{0}$ and that $\bm{y}=\bm{0}$ due to the initial conditions. 
Then, by the uniqueness of solutions to the boundary value problem \eqref{EqDiff2} implies $\nu=0$. 
Therefore, the proof of Theorem \ref{th:unique_g<>0} is complete. 
\hfill$\Box$

\subsection{Proof of Theorem \ref{th:unique_g==0}}
In this subsection, we will show the uniqueness of solutions to the problem \eqref{Eq}--\eqref{IC} in the case $\bm{g}=\bm{0}$ 
without assuming a priori the stability condition \eqref{SolClass0_SC}. 
The following lemma ensures that in the case $\bm{g}=\bm{0}$ if a solution to \eqref{Eq}--\eqref{IC} does not satisfy the stability condition, 
then it is necessarily the trivial one.

\begin{lemma}\label{lem:SpecialCase}
Let $\bm{g}=\bm{0}$ and $(\bm{x},\tau)$ be a solution to the problem \eqref{Eq}--\eqref{IC} in the class \eqref{SolClass0}. 
Suppose that the stability condition \eqref{SolClass0_SC} is not satisfied, that is, 
\begin{equation}\label{SolClass0_non-SC}
\inf_{(s,t)\in(0,1)\times(0,T)}\frac{\tau(s,t)}{s}\leq 0. 
\end{equation}
Then, we have $\bm{x}(s,t)\equiv\bm{x}_0^{\mathrm{in}}(s)$ and $\tau(s,t)\equiv 0$. 
Particularly, $\bm{x}_1^{\mathrm{in}}(s)\equiv\bm{0}$ must hold. 
\end{lemma}

\begin{proof}
By Lemma \ref{lem:EstSolBVP1} we see that 
$\frac{\tau(s,t)}{s} \geq \|\sigma^{\frac12}\dot{\bm{x}}'(t)\|_{L^2}^2 \exp(-2\|\sigma^{\frac12}\bm{x}''(t)\|_{L^2}^2)$ for any $(s,t)\in(0,1]\times[0,T]$. 
Then, continuity of the right-hand side with respect to $t$ together with \eqref{SolClass0_non-SC} yields that there exists a $t_0\in[0,T]$ such that 
$\dot{\bm{x}}'(s,t_0)\equiv\bm{0}$, which together with the boundary condition \eqref{BC} implies $\dot{\bm{x}}(s,t_0)\equiv\bm{0}$. 
On the other hand, by a standard argument, we see easily that $\frac{\mathrm{d}}{\mathrm{d}t}\int_0^1|\dot{\bm{x}}(s,t)|^2 \mathrm{d}s=0$, 
where we used $\bm{x}'\cdot\dot{\bm{x}}'=0$, which comes from the second equation in \eqref{Eq}. 
Therefore, for any $t\in[0,T]$, it holds that $\|\dot{\bm{x}}(t)\|_{L^2}=\|\dot{\bm{x}}(t_0)\|_{L^2}=0$. 
This implies $\dot{\bm{x}}(s,t)\equiv\bm{0}$, and hence we have $\bm{x}(s,t)\equiv \bm{x}_0^{\mathrm{in}}(s)$ and $\bm{x}_1^{\mathrm{in}}(s)\equiv\bm{0}$. 
Moreover, by Lemma \ref{lem:EstSolBVP1} we see that $0\leq\tau(s,t)\leq s\|\dot{\bm{x}}'(t)\|_{L^2}^2=0$, which implies $\tau(s,t)\equiv 0$. 
\end{proof}

As a corollary of Lemma \ref{lem:SpecialCase} we have the following.

\begin{corollary}\label{cor:SpecialCase}
Let $\bm{g}=\bm{0}$ and $(\bm{x},\tau)$ be a solution to the problem \eqref{Eq}--\eqref{IC} in the class \eqref{SolClass0}. 
Suppose that the initial data satisfy $\bm{x}_1^{\mathrm{in}}(s)\equiv\bm{0}$. 
Then, we have $\bm{x}(s,t)\equiv\bm{x}_0^{\mathrm{in}}(s)$ and $\tau(s,t)\equiv 0$. 
\end{corollary}

\begin{proof}
By Lemma \ref{lem:EstSolBVP1} we see that $0\leq\tau(s,0) \leq s\|\dot{\bm{x}}'(0)\|_{L^2}^2 = s\|\bm{x}_1^{\mathrm{in}\,\prime}\|_{L^2}^2=0$, 
and hence $\tau(s,t)$ satisfies \eqref{SolClass0_non-SC}. 
Therefore, by Lemma \ref{lem:SpecialCase} we have the desired result. 
\end{proof}

We now give a proof of Theorem \ref{th:unique_g==0}. 
Suppose that $(\bm{x}_1,\tau_1)$ and $(\bm{x}_2,\tau_2)$ are solutions to the problem in the class \eqref{SolClass0}. 
If $\bm{x}_1^{\mathrm{in}}(s)\equiv\bm{0}$, then Corollary \ref{cor:SpecialCase} implies $(\bm{x}_1,\tau_1)=(\bm{x}_2,\tau_2)=(\bm{x}_0^\mathrm{in},0)$. 
Otherwise, it follows from Lemma \ref{lem:SpecialCase} that 
$\tau_1$ and $\tau_2$ satisfy the stability condition \eqref{SolClass0_SC}. 
Therefore, the result follows from Theorem \ref{th:unique_g<>0}. 
\hfill$\Box$

\section{Estimates for the tension $\tau$}\label{sect:ETau}
\begin{lemma}\label{lem:EstTau1}
Let $M$ be a positive constant and $m$ and $j$ integers such that $m\geq5$ and $0\leq j\leq m-2$. 
There exists a positive constant $C=C(M,m)$ such that if $\bm{x}$ satisfies 
\[
\sum_{l=0}^{j+1} \|\dt^l\bm{x}(t)\|_{X^{m-l}} \leq M, 
\]
then the solution $\tau$ to the boundary value problem \eqref{BVP} satisfies the following estimates: 
\[
\begin{cases}
 \|\dt^j\tau'(t)\|_{L^\infty \cap X^{m-1-j}} \leq C &\mbox{in the case $j\leq m-3$}, \\
 \|\dt^{m-2}\tau'(t)\|_{X^1} \leq C &\mbox{in the case $j=m-2$}.
\end{cases}
\]
\end{lemma}

\begin{proof}
We prove this lemma by induction on $j$. 
Assuming that 
\begin{equation}\label{IndAss1}
\sum_{l=0}^{j-1} \|\dt^l\tau'(t)\|_{L^\infty \cap X^{m-1-l}} \lesssim 1
\end{equation}
holds in the case $j\geq1$, we are going to evaluate $\dt^j\tau$. 
We note that 
the above induction hypothesis together with the boundary condition $\dt^l\tau|_{s=0}=0$ implies that $|\dt^l\tau(s,t)| \lesssim s$ for $l=0,1,\ldots,j-1$ 
and that 
$\dt^j\tau$ is a solution to the boundary value problem 
\begin{equation}\label{BVPdtj}
\begin{cases}
 -\dt^j\tau''+|\bm{x}''|^2\dt^j\tau = h_j &\mbox{in}\quad (0,1)\times(0,T), \\
 \dt^j\tau=0 &\mbox{on}\quad \{s=0\}\times(0,T), \\
 \dt^j\tau'= a_j&\mbox{on}\quad \{s=1\}\times(0,T),
\end{cases}
\end{equation}
where $h_j=\dt^j|\dot{\bm{x}}'|^2-[\dt^j,|\bm{x}''|^2]\tau$ and $a_j=-\bm{g}\cdot\dt^j\bm{x}'|_{s=1}$. 
Here, in view of $2\leq m-j$ and the standard Sobolev embedding theorem we see easily that 
$|a_j(t)| \lesssim \|\dt^j\bm{x}(t)\|_{X^{m-j}} \lesssim 1$.

\medskip
\noindent
{\bf Step 1.} 
We first derive pointwise estimates for $\dt^j\tau$. 
By \eqref{IndAss1}, we see that 
\[
|h_j| \lesssim \sum_{j_1+j_2=j, j_1\leq j_2}|\dt^{j_1+1}\bm{x}'\cdot\dt^{j_2+1}\bm{x}'|
 + \sum_{j_1+j_2\leq j, j_1\leq j_2}s|\dt^{j_1}\bm{x}''\cdot\dt^{j_2}\bm{x}''|.
\]

(i) The case $j\leq m-3$. 
In this case, we have $j_1\leq j_2\leq m-3$ so that $2\leq m-(j_2+1) \leq m-(j_2+1)$ and that $3\leq m-j_2 \leq m-j_1$. 
Therefore, we obtain 
\begin{align*}
\|\dt^{j_1+1}\bm{x}'\cdot\dt^{j_2+1}\bm{x}'\|_{L^1}
&\leq \|\dt^{j_1+1}\bm{x}'\|_{L^2} \|\dt^{j_2+1}\bm{x}'\|_{L^2} \\
&\leq \|\dt^{j_1+1}\bm{x}\|_{X^{m-(j_1+1)}} \|\dt^{j_2+1}\bm{x}\|_{X^{m-(j_2+1)}}
\end{align*}
and 
\begin{align*}
\|s\dt^{j_1}\bm{x}''\cdot\dt^{j_2}\bm{x}''\|_{L^1}
&\leq \|s^\frac12\dt^{j_1}\bm{x}''\|_{L^2} \|s^\frac12\dt^{j_2}\bm{x}''\|_{L^2} \\
&\leq \|\dt^{j_1}\bm{x}\|_{X^{m-j_1}}\|\dt^{j_2}\bm{x}\|_{X^{m-j_2}}
\end{align*}
which imply $\|h_j\|_{L^1}\lesssim 1$. 
By Lemma \ref{lem:EstSolBVP2} with $\alpha=0$ we obtain 
$\|\dt^j\tau'(t)\|_{L^\infty} \lesssim 1$.

(ii) The case $j=m-2$. 
In this case, we have $j_1\leq m-3$ and $j_2\leq m-2$ so that $2\leq m-(j_1+1)$, $1\leq m-(j_2+1)$ and that $3\leq m-j_1$ and $2\leq m-j_2$. 
Therefore, we obtain 
\begin{align*}
\|s^\frac12\dt^{j_1+1}\bm{x}'\cdot\dt^{j_2+1}\bm{x}'\|_{L^1}
&\leq \|\dt^{j_1+1}\bm{x}'\|_{L^2} \|s^\frac12\dt^{j_2+1}\bm{x}'\|_{L^2} \\
&\leq \|\dt^{j_1+1}\bm{x}\|_{X^{m-(j_1+1)}} \|\dt^{j_2+1}\bm{x}\|_{X^{m-(j_2+1)}}
\end{align*}
and 
\begin{align*}
\|s^\frac32\dt^{j_1}\bm{x}''\cdot\dt^{j_2}\bm{x}''\|_{L^1}
&\leq \|s^\frac12\dt^{j_1}\bm{x}''\|_{L^2} \|s\dt^{j_2}\bm{x}''\|_{L^2} \\
&\leq \|\dt^{j_1}\bm{x}\|_{X^{m-j_1}}\|\dt^{j_2}\bm{x}\|_{X^{m-j_2}}
\end{align*}
which imply $\|s^\frac12h_{m-2}\|_{L^1}\lesssim 1$. 
By Lemma \ref{lem:EstSolBVP3} with $\alpha=0$ and $p=2$, we obtain $\|\dt^{m-2}\tau'(t)\|_{L^2} \lesssim 1$.

\medskip
\noindent
{\bf Step 2.} 
We then derive an estimate for $\|\dt^j\tau'\|_{X^{m-1-j}}$. 
To this end, we will evaluate $\|\dt^j\tau'\|_{X^{k+1}}$ inductively on $k$ for $0\leq k\leq m-2-j$. 
By \eqref{Ym1}, we have $\|\dt^j\tau'\|_{X^{k+1}}^2 = \|\dt^j\tau'\|_{L^2}^2 + \|\dt^j\tau''\|_{Y^k}^2$ so that 
it is sufficient to evaluate $\|\dt^j\tau''\|_{Y^k}$. 
Moreover, by \eqref{BVPdtj} we have 
$\|\dt^j\tau''\|_{Y^k} \leq \|(\dt^j\tau)|\bm{x}''|^2\|_{Y^k} + \|h_j\|_{Y^k}$. 
We evaluate the first term $\|(\dt^j\tau)|\bm{x}''|^2\|_{Y^k}$ by using the estimates obtained in the previous Step 1 
and Lemma \ref{lem:CalIneq2} as follows. 

(i) The case $k=0$. 
\[
\|(\dt^j\tau)|\bm{x}''|^2\|_{Y^0} \lesssim \|\dt^j\tau'\|_{L^2}\|\bm{x}\|_{X^4}^2. 
\]

(ii) The case $k=1$. 
Due to the restriction $0\leq k\leq m-2-j$, we have $j\leq m-3$, so that 
\[
\|(\dt^j\tau)|\bm{x}''|^2\|_{Y^1} \lesssim \|\dt^j\tau'\|_{L^\infty} \|\bm{x}\|_{X^4}^2. 
\]

(iii) The case $2\leq k\leq m-2-j$. In this case we have $j\leq m-4$ and $k+2\leq m$, so that 
\[
\|(\dt^j\tau)|\bm{x}''|^2\|_{Y^k} \lesssim \|\dt^j\tau'\|_{L^\infty \cap X^{k-1}} \|\bm{x}\|_{X^{k+2}}^2. 
\]
Therefore, we have 
\[
\|(\dt^j\tau)|\bm{x}''|^2\|_{Y^k} \lesssim
\begin{cases}
 1 &\mbox{for}\quad k=0,1, \\
 1+\|\dt^j\tau'\|_{X^{k-1}} &\mbox{for}\quad 2\leq k\leq m-2-j,
\end{cases}
\]
so that by induction we obtain $\|\dt^j\tau'\|_{X^{k+1}} \lesssim 1+\|h_j\|_{Y^k}$ for $0\leq k\leq m-2-j$. 
Particularly, we get $\|\dt^j\tau'\|_{X^{m-1-j}} \lesssim 1+\|h_j\|_{Y^{m-2-j}}$.

\medskip
\noindent
{\bf Step 3.} 
We finally derive an estimate for $\|h_j\|_{Y^{m-2-j}}$. 
In view of $h_j=\dt^j|\dot{\bm{x}}'|^2-[\dt^j,|\bm{x}''|^2]\tau$, we have 
\[
\|h_j\|_{Y^{m-2-j}} \lesssim \sum_{j_1+j_2=j,j_1\leq j_2}I_1(j_1,j_2;j)
 + \sum_{j_0+j_1+j_2=j, j_0\leq j-1, j_1\leq j_2}I_2(j_0,j_1,j_2;j),
\]
where 
\[
\begin{cases}
 I_1(j_1,j_2;j) = \|\dt^{j_1+1}\bm{x}'\cdot\dt^{j_2+1}\bm{x}'\|_{Y^{m-2-j}}, \\
 I_2(j_0,j_1,j_2;j) = \|(\dt^{j_0}\tau)\dt^{j_1}\bm{x}''\cdot\dt^{j_2}\bm{x}''\|_{Y^{m-2-j}}.
\end{cases}
\]

We evaluate $I_1(j_1,j_2;j)$ by using Lemma \ref{lem:CalIneq1} as follows. 

(i) The case $j\leq m-4$. 
In this case we have $j_1\leq m-5$, so that $4\leq m-(j_1+1)$. 
Therefore, 
\begin{align*}
I_1(j_1,j_2;j)
&\lesssim \|\dt^{j_1+1}\bm{x}\|_{X^{m-1-j \vee 4}} \|\dt^{j_2+1}\bm{x}\|_{X^{m-1-j}} \\
&\leq \|\dt^{j_1+1}\bm{x}\|_{X^{m-(j_1+1)}} \|\dt^{j_2+1}\bm{x}\|_{X^{m-(j_2+1)}}.
\end{align*}

(ii) The case $j=m-2,m-3$. 
In the case $j_2=j$, 
\begin{align*}
I_1(0,j;j)
&\lesssim \|\dt\bm{x}\|_{X^4} \|\dt^{j+1}\bm{x}\|_{X^{m-1-j}}  \\
&\leq \|\dt\bm{x}\|_{X^{m-1}} \|\dt^{j+1}\bm{x}\|_{X^{m-(j+1)}},
\end{align*}
and in the case $j_2\leq j-1$ 
\begin{align*}
I_1(j_1,j_2;j)
&\lesssim \|\dt^{j_1+1}\bm{x}\|_{X^{m-j}} \|\dt^{j_2+1}\bm{x}\|_{X^{m-j}} \\
&\leq \|\dt^{j_1+1}\bm{x}\|_{X^{m-(j_1+1)}} \|\dt^{j_2+1}\bm{x}\|_{X^{m-(j_2+1)}}.
\end{align*}
In any of these cases, we have $I_1(j_1,j_2;j) \lesssim 1$.

We proceed to evaluate $I_2(j_0,j_1,j_2;j)$ by using Lemma \ref{lem:CalIneq2} as follows. 

(i) The case $j\leq m-3$. 
In this case we have $m-2-j\geq1$ and $j_1\leq m-4$, so that 
\begin{align*}
I_2(j_0,j_1,j_2;j)
&\lesssim \|\dt^{j_0}\tau'\|_{L^\infty \cap X^{m-3-j}} \|\dt^{j_1}\bm{x}\|_{X^{m-j \vee 4}} \|\dt^{j_2}\bm{x}\|_{X^{m-j}} \\
&\lesssim \|\dt^{j_0}\tau'\|_{L^\infty \cap X^{m-1-j_0}} \|\dt^{j_1}\bm{x}\|_{X^{m-j_1}} \|\dt^{j_2}\bm{x}\|_{X^{m-j_2}}.
\end{align*}

(ii) The case $j=m-2$. 
In the case $j_2=j$, 
\begin{align*}
I_2(0,0,j;j)
&\lesssim \|\tau'\|_{L^\infty} \|\bm{x}\|_{X^4} \|\dt^j\bm{x}\|_{X^2} \\
&\lesssim \|\tau'\|_{L^\infty} \|\bm{x}\|_{X^m} \|\dt^j\bm{x}\|_{X^{m-j}},
\end{align*}
and in the case $j_2\leq j-1=m-3$, 
\begin{align*}
I_2(j_0,j_1,j_2;j)
&\lesssim \|\dt^{j_0}\tau'\|_{L^\infty} \|\dt^{j_1}\bm{x}\|_{X^3} \|\dt^{j_2}\bm{x}\|_{X^3} \\
&\lesssim \|\dt^{j_0}\tau'\|_{L^\infty} \|\dt^{j_1}\bm{x}\|_{X^{m-j_1}} \|\dt^{j_2}\bm{x}\|_{X^{m-j_2}}.
\end{align*}
In any of these cases, we have $I_2(j_0,j_1,j_2;j) \lesssim 1$. 
Therefore, we have shown that $\|h_j\|_{Y^{m-2-j}} \lesssim 1$.

Summarizing the above calculations, we have proved that 
\[
\begin{cases}
 \|\dt^j\tau'(t)\|_{L^\infty \cap X^{m-1-j}} \lesssim 1 &\mbox{in the case $j\leq m-3$}, \\
 \|\dt^{m-2}\tau'(t)\|_{X^1} \lesssim 1&\mbox{in the case $j=m-2$}
\end{cases}
\]
under the inductive hypothesis \eqref{IndAss1}. 
Therefore, we obtain the desired estimates. 
\end{proof}

In the case $m=4$ we cannot expect that the estimates for the tension $\tau$ obtained in Lemma \ref{lem:EstTau1} hold. 
In this critical case, we obtain weaker estimates for the tension $\tau$, which are given in the following lemma.

\begin{lemma}\label{lem:EstTau2}
Let $M$ be a positive constant and $j$ an integer such that $0\leq j\leq 2$. 
For any $\epsilon>0$, there exists a positive constant $C(\epsilon)=C(M,\epsilon)$ such that if $\bm{x}$ satisfies 
\begin{equation}\label{EstTauAss}
\sum_{l=0}^{j+1} \|\dt^l\bm{x}(t)\|_{X^{4-l}} \leq M, 
\end{equation}
then the solution $\tau$ to the boundary value problem \eqref{BVP} satisfies the following estimates: 
\[
\begin{cases}
 \|\dt^j\tau'(t)\|_{X_\epsilon^{3-j}} \leq C(\epsilon) &\mbox{in the case $j=0,1$}, \\
 \|\dt^2\tau'(t)\|_{X_\epsilon^1} \leq C(\epsilon) &\mbox{in the case $j=2$}.
\end{cases}
\]
In addition to \eqref{EstTauAss} with $j=2$, if we assume $\|\dt\bm{x}'(t)\|_{L^\infty} \leq M$, then we have $\opnorm{ \tau'(t) }_{3,*} \leq C$, where $C=C(M)>0$. 
\end{lemma}

\begin{proof}
As before, we prove this lemma by induction on $j$. 
In the following we will use estimates in Lemmas \ref{lem:embedding} and \ref{lem:CalIneq1} and 
$\|u'\|_{X^k} \lesssim \|u\|_{X^{k+2}}$ without any comment. 

(i) The case $j=0$. 
The estimate $\|\tau'\|_{L^\infty} \lesssim 1$ in the proof of Lemma \ref{lem:EstTau1} is still valid. 
By using the equation for $\tau$, we have $|\tau''|\lesssim |\dot{\bm{x}}'|^2+s|\bm{x}''|^2$, so that 
\begin{align}\label{EstTau''}
\|s^\epsilon \tau''\|_{L^\infty}
&\lesssim \|s^\frac{\epsilon}{2}\dot{\bm{x}}'\|_{L^\infty}^2 + \|s^\frac12\bm{x}''\|_{L^\infty}^2 \\
&\lesssim \|\dot{\bm{x}}\|_{X^3}^2 + \|\bm{x}\|_{X^4}^2. \nonumber
\end{align}
Differentiating the equation for $\tau$ with respect to $s$, 
we have $|\tau'''| \lesssim |\dot{\bm{x}}'\cdot\dot{\bm{x}}''|+|\bm{x}''|^2+s|\bm{x}''\cdot\bm{x}'''|$, so that 
\begin{align}\label{EstTau'''}
\|s^{\frac12+\epsilon}\tau'''\|_{L^2}
&\lesssim \|s^\epsilon\dot{\bm{x}}'\|_{L^\infty} \|s^\frac12\dot{\bm{x}}''\|_{L^2}
 + \|s^\frac12\bm{x}''\|_{L^\infty}( \|s\bm{x}'''\|_{L^2} + \|\bm{x}''\|_{L^2} ) \\
&\lesssim \|\dot{\bm{x}}\|_{X^3}^2 + \|\bm{x}\|_{X^4}^2. \nonumber
\end{align}
Similarly, we have $|\tau''''| \lesssim |\dot{\bm{x}}'\cdot\dot{\bm{x}}'''|+|\dot{\bm{x}}''|^2 + s(|\bm{x}''\cdot\bm{x}''''|+|\bm{x}'''|^2)
 + |\bm{x}''\cdot\bm{x}'''| + |\bm{x}''|^2|\tau''|$, so that 
\begin{align*}
\|s^{\frac32+\epsilon}\tau''''\|_{L^2}
&\lesssim \|s^\epsilon\dot{\bm{x}}'\|_{L^\infty} \|s^\frac32\dot{\bm{x}}'''\|_{L^2}
 + \|s\dot{\bm{x}}''\|_{L^\infty} \|s^\frac12\dot{\bm{x}}''\|_{L^2} \\
&\quad\;
 + \|s^\frac12\bm{x}''\|_{L^\infty}( \|s^2\bm{x}''''\|_{L^2} + \|s\bm{x}'''\|_{L^2} + \|s^\frac12\tau''\|_{L^\infty}\|\bm{x}''\|_{L^2} ) \\
&\quad\;
 + \|s^\frac32\bm{x}''\|_{L^\infty}\|s\bm{x}'''\|_{L^2} \\
&\lesssim \|\dot{\bm{x}}\|_{X^3}^2 + (1+\|s^\frac12\tau''\|_{L^\infty})\|\bm{x}\|_{X^4}^2.
\end{align*}
Therefore, we obtain $\|\tau'(t)\|_{X_\epsilon^3} \lesssim 1$.

(ii) The case $j=1$. 
The estimate $\|\dot{\tau}'\|_{L^\infty} \lesssim 1$ in the proof of Lemma \ref{lem:EstTau1} is still valid. 
By the equation for $\dot{\tau}$, we have 
$|\dot{\tau}''|\lesssim |\dot{\bm{x}}'\cdot\ddot{\bm{x}}'|+s(|\bm{x}''\cdot\dot{\bm{x}}''|+|\bm{x}''|^2)$, so that 
\begin{align}\label{EstdtTau''}
\|s^\epsilon\dot{\tau}''\|_{L^2}
&\lesssim \|s^\epsilon\dot{\bm{x}}'\|_{L^\infty} \|\ddot{\bm{x}}'\|_{L^2}
 + \|s^\frac12\bm{x}''\|_{L^\infty}( \|s^\frac12\dot{\bm{x}}''\|_{L^2} + \|\bm{x}''\|_{L^2} ) \\
&\lesssim \|\dot{\bm{x}}\|_{X^3} \|\ddot{\bm{x}}\|_{X^2} 
 + \|\bm{x}\|_{X^4} ( \|\dot{\bm{x}}\|_{X^3} + \|\bm{x}\|_{X^4} ). \nonumber
\end{align}
Similarly, we have 
$|\dot{\tau}'''| \lesssim |\dot{\bm{x}}'\cdot\ddot{\bm{x}}''| + |\dot{\bm{x}}''\cdot\ddot{\bm{x}}'|
 + s(|\bm{x}''\cdot\dot{\bm{x}}'''| + |\bm{x}'''\cdot\dot{\bm{x}}''| + |\bm{x}''\cdot\bm{x}'''| )
 + |\bm{x}''\cdot\dot{\bm{x}}''| + |\bm{x}''|^2$, so that 
\begin{align*}
\|s^{1+\epsilon}\dot{\tau}'''\|_{L^2}
&\lesssim \|s^\epsilon\dot{\bm{x}}'\|_{L^\infty} \|s\ddot{\bm{x}}''\|_{L^2}
 + \|s\dot{\bm{x}}''\|_{L^\infty} ( \|\ddot{\bm{x}}'\|_{L^2} + \|s\bm{x}'''\|_{L^2} ) \\
&\quad\;
 + \|s^\frac12\bm{x}''\|_{L^\infty} ( \|s^\frac32\dot{\bm{x}}'''\|_{L^2} + \|s\bm{x}'''\|_{L^2}
 + \|s^\frac12\dot{\bm{x}}''\|_{L^2} + \|\bm{x}''\|_{L^2} ) \\
&\lesssim \|\dot{\bm{x}}\|_{X^3} \|\ddot{\bm{x}}\|_{X^2} + \|\bm{x}\|_{X^4} \|\dot{\bm{x}}\|_{X^3} + \|\bm{x}\|_{X^4}^2.
\end{align*}
Therefore, we obtain $\|\dot{\tau}'(t)\|_{X_\epsilon^2} \lesssim 1$.

(iii) The case $j=2$. 
The estimate $\|\ddot{\tau}'\|_{L^2} \lesssim 1$ in the proof of Lemma \ref{lem:EstTau1} is still valid. 
By the equation for $\ddot{\tau}$, we have 
$|\ddot{\tau}''| \lesssim |\dot{\bm{x}}'\cdot \dddot{\bm{x}}'| + |\ddot{\bm{x}}'|^2 
 + s( |\bm{x}''\cdot\ddot{\bm{x}}''| + |\dot{\bm{x}}''|^2 + |\bm{x}''\cdot\dot{\bm{x}}''|)
 + s^\frac12|\bm{x}''|^2$, so that 
\begin{align*}
\|s^{\frac12+\epsilon}\ddot{\tau}''\|_{L^2}
&\lesssim \|s^\epsilon\dot{\bm{x}}'\|_{L^\infty} \|s^\frac12\dddot{\bm{x}}'\|_{L^2}
 + \|s^\frac12\ddot{\bm{x}}'\|_{L^\infty} \|\ddot{\bm{x}}'\|_{L^2} + \|s\dot{\bm{x}}''\|_{L^\infty}\|s^\frac12\dot{\bm{x}}''\|_{L^2} \\
&\quad\;
 + \|s^\frac12\bm{x}''\|_{L^\infty} ( \|s\ddot{\bm{x}}''\|_{L^2} + \|s^\frac12\dot{\bm{x}}''\|_{L^2} + \|\bm{x}''\|_{L^2} ) \\
&\lesssim \|\dot{\bm{x}}\|_{X^3} \|\dddot{\bm{x}}\|_{X^1} + (\|\ddot{\bm{x}}\|_{X^2}+\|\dot{\bm{x}}\|_{X^3}+\|\bm{x}\|_{X^4})^2.
\end{align*}
Moreover, it follows from \eqref{BVPdtj} with $j=2$ and Lemma \ref{lem:EstSolBVP3} that 
$\|s^\epsilon\ddot{\tau}'\|_{L^\infty} \lesssim |a_2| + \|s^\epsilon h_2\|_{L^1}$. 
Here, we see that $\|s^\epsilon h_2\|_{L^1} \leq \epsilon^{-\frac12}\|s^{\frac12(1+\epsilon)}h_2\|_{L^2}$, 
which can be evaluated as above. 
Therefore, we obtain $\|\ddot{\tau}'(t)\|_{X_\epsilon^1} \lesssim 1$.

The only reason to use an additional weight $s^\epsilon$ in the above estimates is a lack of the estimate for $\|\dt\bm{x}'(t)\|_{L^\infty}$; 
we note that $\dt\bm{x}(t) \in X^3$ does not necessarily imply $\dt\bm{x}'(t) \in L^\infty$. 
Therefore, we obtain the later assertion of the lemma. 
The proof is complete. 
\end{proof}

\section{Estimates for initial values}\label{sect:EstID}
In this section we evaluate the initial value $\opnorm{\bm{x}(0)}_m$ in terms of the initial data $\|\bm{x}(0)\|_{X^m}$ and 
$\|\dot{\bm{x}}(0)\|_{X^{m-1}}$. 
Although it is sufficient to evaluate $\dt^j\bm{x}$ only at time $t=0$, we will evaluate them at general time $t$. 
We remind that the operator $\mathscr{A}_\tau$ was defined by \eqref{defA}.

\begin{lemma}\label{lem:EstAtau}
If $\tau|_{s=0}=0$, then we have 
\[
\|\mathscr{A}_\tau\bm{x}\|_{X^m} \lesssim
\begin{cases}
 \min\{ \|\tau'\|_{L^\infty} \|\bm{x}\|_{X^2}, \|\tau'\|_{X^1} \|\bm{x}\|_{X^3} \} &\mbox{for}\quad m=0, \\
 \min\{ \|\tau'\|_{X^{m \vee 2}} \|\bm{x}\|_{X^{m+2}}, \|\tau'\|_{X^m} \|\bm{x}\|_{X^{m+2 \vee 4}} \} &\mbox{for}\quad m=0,1,2,\ldots.
\end{cases}
\]
\end{lemma}

\begin{proof}
We put $\mu(s)=\frac{\tau(s)}{s}=\frac{1}{s}\int_0^s\tau'(\sigma)\mathrm{d}\sigma=(\mathscr{M}\tau')(s)$, where $\mathscr{M}$ is the 
averaging operator defined by \eqref{defM}, and $(A_2u)(s)=-(su'(s))'$. 
Then, we have the identity 
\begin{equation}\label{IdAtau}
\mathscr{A}_\tau\bm{x} = \mu A_2\bm{x}+(\mu-\tau')\bm{x}'.
\end{equation}
Therefore, by Lemmas \ref{lem:EstA2} and \ref{lem:Algebra} and Corollary \ref{cor:WEM1} we see that 
\begin{align*}
\|\mathscr{A}_\tau\bm{x}\|_{L^2}
&\leq \|\mu\|_{L^\infty}\|A_2\bm{x}\|_{L^2} + (\|\mu\|_{L^\infty}+\|\tau'\|_{L^\infty})\|\bm{x}'\|_{L^2} \\
&\lesssim \|\tau'\|_{L^\infty} \|\bm{x}\|_{X^2}, \\
\|\mathscr{A}_\tau\bm{x}\|_{L^2}
&\lesssim \|\mu\|_{X^1}\|A_2\bm{x}\|_{X^1} + (\|\mu\|_{X^1}+\|\tau'\|_{X^1})\|\bm{x}'\|_{X^1} \\
&\lesssim \|\tau'\|_{X^1}\|\bm{x}\|_{X^3}, 
\end{align*}
and that 
\begin{align*}
\|\mathscr{A}_\tau\bm{x}\|_{X^m}
&\lesssim \|\mu\|_{X^{m \vee 2}} \|A_2\bm{x}\|_{X^m} + ( \|\mu\|_{X^{m \vee 2}} + \|\tau'\|_{X^{m \vee 2}} )\|\bm{x}'\|_{X^m} \\
&\lesssim \|\tau'\|_{X^{m \vee 2}} \|\bm{x}\|_{X^{m+2}}, \\
\|\mathscr{A}_\tau\bm{x}\|_{X^m}
&\lesssim \|\mu\|_{X^m} \|A_2\bm{x}\|_{X^{m \vee 2}} + ( \|\mu\|_{X^m} + \|\tau'\|_{X^m} )\|\bm{x}'\|_{X^{m \vee 2}} \\
&\lesssim \|\tau'\|_{X^m} \|\bm{x}\|_{X^{m+2 \vee 4}}.
\end{align*}
Therefore, we obtain the desired estimates. 
\end{proof}

\begin{lemma}\label{lem:EstID}
Let $M$ be a positive constant and $m$ an integer such that $m\geq4$. 
There exists a positive constant $C=C(M,m)$ such that if $(\bm{x},\tau)$ is a solution to \eqref{Eq} and \eqref{BC} satisfying 
$\|\bm{x}(t)\|_{X^m} + \|\dt\bm{x}(t)\|_{X^{m-1}} \leq M$, then we have $\opnorm{\bm{x}(t)}_m \leq C$. 
\end{lemma}

\begin{proof}
We will prove $\|\dt^j\bm{x}(t)\|_{X^{m-j}} \leq C$ inductively for $j=2,3,\ldots,m$.

\medskip
\noindent
{\bf The case $m\geq5$.} 
We first consider the case $m\geq5$. 
By Lemma \ref{lem:EstTau1} with $j=0$, we have $\|\tau'(t)\|_{L^\infty \cap X^{m-1}} \lesssim 1$. 
Now, assuming $2\leq j\leq m$ and 
\[
\sum_{l=0}^{j-1} \|\dt^l\bm{x}(t)\|_{X^{m-l}}
 + \sum_{l=0}^{j-2} \|\dt^l\tau'(t)\|_{L^\infty \cap X^{m-(l+1)}} \lesssim 1,
\]
we will evaluate $(\dt^j\bm{x},\dt^{j-1}\tau')$. 
In the case $2\leq j\leq m-1$, we evaluate $\dt^j\bm{x}$ as 
\begin{align*}
\|\dt^j\bm{x}\|_{X^{m-j}}
&= \|\dt^{j-2}(\mathscr{A}_\tau\bm{x}-\bm{g})\|_{X^{m-j}} \\
&\lesssim \sum_{j_1+j_2=j-2}\|((\dt^{j_1}\tau)\dt^{j_2}\bm{x}')'\|_{X^{m-j}}+1.
\end{align*}
Here, we have $j_1\leq m-3$ so that $2\leq m-(j_1+1)$. 
Therefore, by Lemma \ref{lem:EstAtau} 
\begin{align*}
\|((\dt^{j_1}\tau)\dt^{j_2}\bm{x}')'\|_{X^{m-j}}
&\lesssim \|\dt^{j_1}\tau'\|_{X^{m-j \vee 2}} \|\dt^{j_2}\bm{x}\|_{X^{m-j+2}} \\
&\leq \|\dt^{j_1}\tau'\|_{X^{m-(j_1+1)}} \|\dt^{j_2}\bm{x}\|_{X^{m-j_2}}.
\end{align*}
These estimates give $\|\dt^j\bm{x}(t)\|_{X^{m-j}} \lesssim 1$. 
Now, we can apply Lemma \ref{lem:EstTau1} with $j$ replaced by $j-1$ to obtain 
$\|\dt^{j-1}\tau'(t)\|_{L^\infty \cap X^{m-j}} \lesssim 1$.

In the case $j=m$, we evaluate $\dt^m\bm{x}$ as 
\begin{align*}
\|\dt^m\bm{x}\|_{L^2}
&= \|\dt^{m-2}\mathscr{A}_\tau\bm{x}\|_{X^{m-j}} \\
&\lesssim \|((\dt^{m-2}\tau)\bm{x}')'\|_{L^2} + \sum_{j_1+j_2=m-2, j_1\leq m-3}\|((\dt^{j_1}\tau)\dt^{j_2}\bm{x}')'\|_{L^2}.
\end{align*}
Here, by Lemma \ref{lem:EstAtau} we see that 
$\|((\dt^{m-2}\tau)\bm{x}')'\|_{L^2} \lesssim \|\dt^{m-2}\tau'\|_{X^1}\|\bm{x}\|_{X^3} \lesssim 1$ and that 
\begin{align*}
\|((\dt^{j_1}\tau)\dt^{j_2}\bm{x}')'\|_{L^2}
&\lesssim \|\dt^{j_1}\tau'\|_{X^2} \|\dt^{j_2}\bm{x}\|_{X^2} \\
&\lesssim \|\dt^{j_1}\tau'\|_{X^{m-(j_1+1)}} \|\dt^{j_2}\bm{x}\|_{X^{m-j_2}}.
\end{align*}
These estimates give $\|\dt^m\bm{x}(t)\|_{L^2} \lesssim 1$. 
Therefore, by induction we obtain $\opnorm{\bm{x}(t)}_m \lesssim 1$.

\medskip
\noindent
{\bf The case $m=4$.} 
We then consider the case $m=4$. 
By Lemma \ref{lem:EstTau2} with $j=0$, we have $\|\tau'(t)\|_{X_\epsilon^3} \leq C(\epsilon)$ for $\epsilon>0$. 
Since $X_\epsilon^3 \hookrightarrow X^2$ for $0<\epsilon\leq\frac12$, we have also $\|\tau'(t)\|_{X^2}\lesssim1$. 
By Lemma \ref{lem:EstAtau}, 
\begin{align*}
\|\dt^2\bm{x}\|_{X^2}
&\leq \|\mathscr{A}_\tau\bm{x}\|_{X^2} + 1 \\
&\lesssim \|\tau'\|_{X^2} \|\bm{x}\|_{X^4} + 1.
\end{align*}
Therefore, we obtain $\|\dt^2\bm{x}(t)\|_{X^2} \lesssim 1$. 
Then, by Lemma \ref{lem:EstTau2} with $j=1$, we have $\|\dot{\tau}'(t)\|_{X_\epsilon^2} \leq C(\epsilon)$ for $\epsilon>0$. 
Since $X_\epsilon^2 \hookrightarrow X^1$ for $0<\epsilon\leq\frac12$, we have also $\|\dot{\tau}'(t)\|_{X^1}\lesssim1$. 
By Lemma \ref{lem:EstAtau}, 
\begin{align*}
\|\dt^3\bm{x}\|_{X^1}
&\leq \|\mathscr{A}_\tau\dot{\bm{x}}\|_{X^1} + \|\mathscr{A}_{\dot{\tau}}\bm{x}\|_{X^1} \\
&\lesssim \|\tau'\|_{X^2} \|\dot{\bm{x}}\|_{X^3} + \|\dot{\tau}'\|_{X^1} \|\bm{x}\|_{X^4}.
\end{align*}
Therefore, we obtain $\|\dt^3\bm{x}(t)\|_{X^1} \lesssim 1$. 
Then, by Lemma \ref{lem:EstTau2} with $j=2$, we have $\|\ddot{\tau}'(t)\|_{X_\epsilon^1} \leq C(\epsilon)$ for $\epsilon>0$. 
Since $X_\epsilon^1 \hookrightarrow L^2$ for $0<\epsilon<\frac12$, we have also $\|\ddot{\tau}'(t)\|_{L^2}\lesssim1$. 
By Lemma \ref{lem:EstAtau}, 
\begin{align*}
\|\dt^4\bm{x}\|_{L^2}
&\leq \|\mathscr{A}_\tau\ddot{\bm{x}}\|_{L^2} + 2\|\mathscr{A}_{\dot{\tau}}\dot{\bm{x}}\|_{L^2} + \|\mathscr{A}_{\ddot{\tau}}\bm{x}\|_{L^2} \\
&\lesssim \|\tau'\|_{X^2} \|\ddot{\bm{x}}\|_{X^2} + \|\dot{\tau}'\|_{X^1} \|\dot{\bm{x}}\|_{X^3} + \|\ddot{\tau}'\|_{L^2} \|\bm{x}\|_{X^4}.
\end{align*}
Therefore, we obtain $\|\dt^4\bm{x}(t)\|_{L^2} \lesssim 1$. 
Summarizing these estimates, we get $\opnorm{ \bm{x}(t) }_4 \lesssim 1$.
The proof is complete. 
\end{proof}

\section{A priori estimates for solutions}\label{sect:APE}
In this last section, we prove Theorem \ref{th:APE}.

\begin{lemma}\label{lem:APEx}
For any integer $m\geq4$ and any positive constants $M_1$, $M_2$, and $c_0$, there exists a positive constant $C_2=C_2(M_1,M_2,c_0,m)$ 
such that if $(\bm{x},\tau)$ is a regular solution to \eqref{Eq} and \eqref{BC} satisfying the stability condition \eqref{SolClass0_SC} and 
\begin{equation}\label{EstAss}
\begin{cases}
 \opnorm{ \bm{x}(t) }_{m-1} \leq M_1, \\ 
 \|\dt^{m-1}\bm{x}(t)\|_{X^1}^2 + \|\dt^{m-2}\bm{x}(t)\|_{X^2}^2 \leq M_2, 
\end{cases}
\end{equation}
then we have $\opnorm{ \bm{x}(t) }_m \leq C_2$. 
\end{lemma}

The proof of this lemma is divided into two cases: (i) $m\geq5$ and (ii) $m=4$.

\subsection{Proof of Lemma \ref{lem:APEx} in the case $m\geq5$}
In this subsection, we prove Lemma \ref{lem:APEx} in the case $m\geq5$, so that we suppose $m\geq5$. 
For $3\leq j\leq m+1$ we are going to show 
\begin{equation}\label{IndAss}
\sum_{l=1}^{j-1} \|\dt^{m-l}\bm{x}(t)\|_{X^{l}} \lesssim 1
\end{equation}
by induction on $j$. 
Therefore, assuming that \eqref{IndAss} holds for some $3\leq j\leq m$ we evaluate $\|\dt^{m-j}\bm{x}(t)\|_{X^j}$, 
which can be written as 
\begin{align}\label{NormDeco}
\|\dt^{m-j}\bm{x}(t)\|_{X^j}^2
&= \|\dt^{m-j}\bm{x}(t)\|_{L^2}^2 + \|\dt^{m-j}\bm{x}'(t)\|_{X^{j-2}}^2 + \|s^\frac{j}{2}\ds^j\dt^{m-j}\bm{x}(t)\|_{L^2}^2.
\end{align}
The first term in the right-hand side can be easily evaluated.

We proceed to evaluate the second term. 
By the assumptions and Lemmas \ref{lem:EstSolBVP1}, \ref{lem:EstTau1}, and \ref{lem:EstTau2}, we have $\tau(s,t)\simeq s$ and 
$\opnorm{\tau'(t)}_{m-2,*} \lesssim 1$ in the case $m\geq6$ and $\opnorm{\mu(t)}_{3,*,\epsilon} \leq C_0(\epsilon)$ for any $\epsilon>0$ 
in the case $m=5$. 
We introduce a new quantity $\mu(s,t)$ by 
\begin{equation}\label{defMu}
\mu(s,t) = \frac{\tau(s,t)}{s} = \frac1s \int_0^s \tau'(\sigma,t)\mathrm{d}\sigma = \mathscr{M}(\tau'(\cdot,t))(s),
\end{equation}
where $\mathscr{M}$ is the averaging operator defined by \eqref{defM} and we have used the boundary condition $\tau|_{s=0}=0$. 
Then, by Corollary \ref{cor:WEM1} we have also $\mu(s,t)\simeq 1$ and $\opnorm{\mu(t)}_{m-2,*} \lesssim 1$ in the case $m\geq6$ and 
$\opnorm{\mu(t)}_{3,*,\epsilon} \leq C_0(\epsilon)$ for any $\epsilon\in(0,1)$ in the case $m=5$. 
Integrating the hyperbolic equations for $\bm{x}$ in \eqref{Eq} with respect to $s$ over $[0,s]$ and using the boundary condition 
$\tau|_{s=0}=0$, we obtain $\tau\bm{x}'=\int_0^s(\ddot{\bm{x}}-\bm{g})\mathrm{d}\sigma$, so that 
\begin{equation}\label{ExpX}
\bm{x}' = \mu^{-1}(\mathscr{M}\ddot{\bm{x}}-\bm{g}).
\end{equation}
Roughly speaking, this expression makes us to convert estimates for the time derivatives of $\bm{x}$ 
into those for the spatial derivatives of $\bm{x}$ with less weight of $s$. 
Differentiating this with respect to $t$ and using $\mu(s,t)\simeq1$ and $|\dt\mu(s,t)|\lesssim1$, 
we have $|\dt\bm{x}'| \lesssim |\mathscr{M}(\dt^3\bm{x})| + |\mathscr{M}(\dt^2\bm{x})| + 1$. 
Here, we see that $\|\dt^3\bm{x}\|_{L^\infty} \lesssim \|\dt^3\bm{x}\|_{X^2} \lesssim 1$ and 
$\|\dt^2\bm{x}\|_{L^\infty} \lesssim \opnorm{\bm{x}}_4 \lesssim 1$. 
Therefore, by Corollary \ref{cor:WEM1} we get $\|\dt\bm{x}'(t)\|_{L^\infty} \lesssim 1$. 
By Lemma \ref{lem:EstTau2} again, we obtain $\opnorm{\tau'(t)}_{3,*} \lesssim 1$ so that $\opnorm{\mu(t)}_{3,*} \lesssim 1$. 
In other words, we have $\opnorm{\mu(t)}_{m-2,*} \lesssim 1$ in any cases. 
We will use these estimates in the following without any comment.

Now, we go back to evaluate the second term in the right-hand sided of \eqref{NormDeco}. 
In the case $j=3$, by Lemma \ref{lem:Algebra} and Corollary \ref{cor:WEM1} we see that 
\begin{align*}
\|\dt^{m-3}\bm{x}'\|_{X^1}
&\lesssim \|\dt^{m-3}(\mu^{-1})\|_{X^1} \|\mathscr{M}\ddot{\bm{x}}-\bm{g}\|_{X^2} \\
&\quad\;
 + \sum_{j_1+j_2=m-3, j_1\leq m-4} \|\dt^{j_1}(\mu^{-1})\|_{X^2} \|\mathscr{M}(\dt^{j_2+2}\bm{x})\|_{X^1} \\
&\lesssim \opnorm{ \mu^{-1} }_{m-2,*}
 \Biggl( 1+\|\ddot{\bm{x}}\|_{X^2} + \sum_{j_2=1}^{m-3} \|\dt^{j_2+2}\bm{x}\|_{X^1} \Biggr),
\end{align*}
which together with Lemma \ref{lem:EstCompFunc1} yields $\|\dt^{m-3}\bm{x}'(t)\|_{X^1} \lesssim 1$. 
We then consider the case $4\leq j\leq m$. 
By Lemma \ref{lem:Algebra} and Corollary \ref{cor:WEM1} we see that 
\begin{align*}
\|\dt^{m-j}\bm{x}'\|_{X^{j-2}}
&\lesssim \sum_{j_1+j_2=m-j} \|\dt^{j_1}(\mu^{-1})\|_{X^{j-2}} \|\mathscr{M}(\dt^{j_2+2}\bm{x}) - \dt^{j_2}\bm{g}\|_{X^{j-2}} \\
&\lesssim \opnorm{ \mu^{-1} }_{m-2,*}( 1+\|\dt^{m-(j-2)}\bm{x}\|_{X^{j-2}}+\opnorm{\bm{x}}_{m-1} ),
\end{align*}
which together with Lemma \ref{lem:EstCompFunc1} yields $\|\dt^{m-j}\bm{x}'(t)\|_{X^{j-2}} \lesssim 1$.

It remains to evaluate the last term in \eqref{NormDeco}. 
Applying $\ds^{j-2}\dt^{m-j}$ to the hyperbolic equations for $\bm{x}$, we have 
$\ds^{j-2}\dt^{m-j+2}\bm{x} = \tau\ds^j\dt^{m-j}\bm{x} + [\ds^{j-1},\tau]\dt^{m-j}\bm{x}' + \ds^{j-2}[\dt^{m-j},\mathscr{A}_\tau]\bm{x}$, so that 
\begin{align*}
s|\ds^j\dt^{m-j}\bm{x}|
&\lesssim |\ds^{j-2}\dt^{m-j+2}\bm{x}| + |[\ds^{j-1},\tau]\dt^{m-j}\bm{x}'| + |\ds^{j-2}[\dt^{m-j},\mathscr{A}_\tau]\bm{x}|
\end{align*}
and that 
\begin{align*}
\|s^\frac{j}{2}\ds^j\dt^{m-j}\bm{x}\|_{L^2}
&\lesssim \|\dt^{m-(j-2)}\bm{x}\|_{X^{j-2}} + \|s^\frac{j-2}{2}[\ds^{j-1},\tau]\dt^{m-j}\bm{x}'\|_{L^2} + \|[\dt^{m-j},\mathscr{A}_\tau]\bm{x}\|_{X^{j-2}}. 
\end{align*}
By Lemma \ref{lem:commutator}, the second term in the right-hand side is evaluated as 
\begin{align*}
\|s^\frac{j-2}{2}[\ds^{j-1},\tau]\dt^{m-j}\bm{x}'\|_{L^2}
&\lesssim \|\tau'\|_{X^{j-2\vee2}}\|\dt^{m-j}\bm{x}'\|_{X^{j-2}} \\
&\lesssim \opnorm{\tau'}_{m-2,*}\|\dt^{m-j}\bm{x}'\|_{X^{j-2}}.
\end{align*}
To evaluate the third term in the right-hand side, it is sufficient to consider the case $3\leq j\leq m-1$. 
Then, by Lemma \ref{lem:EstAtau}, we see that 
\begin{align*}
\|[\dt^{m-j},\mathscr{A}_\tau]\bm{x}\|_{X^{j-2}}
&\lesssim \|\dt^{m-j}\tau'\|_{X^{j-2}}\|\bm{x}\|_{X^{j\vee4}} + \sum_{j_1+j_2=m-j, j_2\geq1} \|\dt^{j_1}\tau'\|_{X^{j-2\vee2}} \|\dt^{j_2}\bm{x}\|_{X^j} \\
&\lesssim \opnorm{\tau'}_{m-2,*}\opnorm{\bm{x}}_{m-1}
\end{align*}
These estimates yield $\|s^\frac{j}{2}\ds^j\dt^{m-j}\bm{x}\|_{L^2} \lesssim 1$.

Summarizing the above argument, we see that under the induction hypothesis \eqref{IndAss} it holds that $\|\dt^{m-j}\bm{x}\|_{X^j} \lesssim 1$. 
Therefore, we obtain $\opnorm{ \bm{x}(t) }_{m,*} \leq C_2$. 
Finally, we evaluate $\|\dt^m\bm{x}\|_{L^2}$ as follows. 
Thanks to the estimate $\opnorm{ \bm{x} }_{m,*} \lesssim 1$, we can improve the estimate for $\tau'$ as 
$\opnorm{ \tau' }_{m-1,*} \lesssim 1$ by Lemma \ref{lem:EstTau1}. 
Therefore, by Lemma \ref{lem:EstAtau} we see that 
\begin{align*}
\|\dt^m\bm{x}\|_{L^2}
&= \|\dt^{m-2}(\mathscr{A}_\tau\bm{x})\|_{L^2} \\
&\lesssim \|\tau'\|_{X^2} \|\dt^{m-2}\bm{x}\|_{X^2} 
 + \sum_{j_1+j_2=m-2, j_2\leq m-3} \|\dt^{j_1}\tau'\|_{X^1} \|\dt^{j_2}\bm{x}\|_{X^3} \\
&\lesssim \opnorm{ \tau' }_{m-1,*} \opnorm{\bm{x}}_{m,*},
\end{align*}
which gives $\opnorm{ \bm{x} }_m \lesssim 1$. 
The proof of Lemma \ref{lem:APEx} in the case $m\geq5$ is complete. 
\hfill$\Box$

\subsection{Proof of Lemma \ref{lem:APEx} in the case $m=4$}
In this subsection, we prove Lemma \ref{lem:APEx} in the critical case $m=4$, so that we suppose $m=4$. 
The proof consists of several steps. 
Before going into the proof, we note that by Lemma \ref{lem:EstSolBVP1} we have $\tau(s,t)\simeq s$ and $|\tau'(s,t)|\lesssim1$. 
Particularly, the quantity $\mu$ defined by \eqref{defMu} satisfies $\mu(s,t)\simeq1$.

\medskip
\noindent
{\bf Step 1.} 
Estimate for $\|\bm{x}'(t)\|_{X^2}$. 
We rewrite \eqref{ExpX} as 
\begin{equation}\label{ExpX2}
\mu\bm{x}' = \mathscr{M}\ddot{\bm{x}}-\bm{g}.
\end{equation}
Differentiating this with respect to $s$, we have $\mu\bm{x}''+\mu'\bm{x}'=(\mathscr{M}\ddot{\bm{x}})'$. 
Taking an inner product of this equation with $\bm{x}'$ and using the constraint $|\bm{x}'|^2=1$ together with $\bm{x}'\cdot\bm{x}''=0$, 
we obtain $\mu'=\bm{x}'\cdot(\mathscr{M}\ddot{\bm{x}})'$. 
Therefore, by Corollary \ref{cor:WEM1} we see that 
$\|\mu'\|_{L^2} \leq \|(\mathscr{M}\ddot{\bm{x}})'\|_{L^2} \leq \frac23\|\ddot{\bm{x}}'\|_{L^2}\lesssim1$. 
Then, in view of $\bm{x}''=\mu^{-1}( (\mathscr{M}\ddot{\bm{x}})' - \mu'\bm{x}')$ we obtain 
$\|\bm{x}''\|_{L^2} \lesssim \|(\mathscr{M}\ddot{\bm{x}})'\|_{L^2} + \|\mu'\|_{L^2} \lesssim 1$.

Differentiating \eqref{ExpX2} twice with respect to $s$, we have $\mu\bm{x}'''+2\mu'\bm{x}''+\mu''\bm{x}'=(\mathscr{M}\ddot{\bm{x}})''$. 
Taking an inner product of this equation with $\bm{x}'$ and using $\bm{x}'\cdot\bm{x}'''+|\bm{x}''|^2=0$, we obtain 
$\mu''=\bm{x}'\cdot(\mathscr{M}\ddot{\bm{x}})''+\mu|\bm{x}''|^2$. 
Therefore, by Corollary \ref{cor:WEM1} and Lemma \ref{lem:CalIneqLp1} we see that 
\begin{align*}
\|s\mu''\|_{L^2}
&\lesssim \|s(\mathscr{M}\ddot{\bm{x}})''\|_{L^2} + \|s\bm{x}''\|_{L^\infty}\|\bm{x}''\|_{L^2} \\
&\lesssim \|\ddot{\bm{x}}\|_{X^2} + \|\bm{x}\|_{X^3}\|\bm{x}''\|_{L^2},
\end{align*}
which implies $\|s\mu''\|_{L^2} \lesssim 1$. 
Then, in view of $\bm{x}'''=\mu^{-1}( (\mathscr{M}\ddot{\bm{x}})''-2\mu'\bm{x}''-\mu''\bm{x}')$ we obtain 
$\|s\bm{x}'''\|_{L^2} \lesssim \|s(\mathscr{M}\ddot{\bm{x}})''\|_{L^2} + \|\mu'\|_{L^2}\|s\bm{x}''\|_{L^\infty} + \|s\mu''\|_{L^2} \lesssim 1$. 
Therefore, we obtain 
\begin{equation}\label{EstX'}
\|\bm{x}'(t)\|_{X^2} + \|\mu(t)\|_{X^2} \leq C_2. 
\end{equation}
Particularly, we get $\|s^\frac12\bm{x}''(t)\|_{L^\infty}+\|s^\frac12\mu'(t)\|_{L^\infty} \leq C_2$. 

\medskip
\noindent
{\bf Step 2.} 
Estimate for $\|\dt\bm{x}'(t)\|_{X^1}$. 
Differentiating \eqref{ExpX2} with respect to $t$, we have $\mu\dot{\bm{x}}'+\dot{\mu}\bm{x}'=\mathscr{M}(\dt^3\bm{x})$. 
Taking an inner product of this equation with $\bm{x}'$ and using $\bm{x}'\cdot\dot{\bm{x}}'=0$, 
we obtain $\dot{\mu} = \bm{x}'\cdot\mathscr{M}(\dt^3\bm{x})$. 
Therefore, we see that $\|\dot{\mu}\|_{L^2} \leq \|\mathscr{M}(\dt^3\bm{x})\|_{L^2} \leq 2\|\dt^3\bm{x}\|_{L^2} \lesssim 1$. 
Differentiating \eqref{ExpX2} with respect to $t$ and $s$, we have 
$\mu\dot{\bm{x}}''+\mu'\dot{\bm{x}}'+\dot{\mu}\bm{x}''+\dot{\mu}'\bm{x}'=(\mathscr{M}(\dt^3\bm{x}))'$. 
Taking an inner product of this equation with $\bm{x}'$ and using $\bm{x}'\cdot\dot{\bm{x}}''+\bm{x}''\cdot\dot{\bm{x}}'=0$, 
we obtain $\dot{\mu}'=\bm{x}'\cdot(\mathscr{M}(\dt^3\bm{x}))'+\mu\bm{x}''\cdot\dot{\bm{x}}'$. 
Therefore, we see that 
\begin{align*}
\|s^\frac12\dot{\mu}'\|_{L^2}
&\leq \|s^\frac12(\mathscr{M}(\dt^3\bm{x}))'\|_{L^2} + \|\mu\|_{L^\infty} \|s^\frac12\bm{x}''\|_{L^\infty} \|\dot{\bm{x}}'\|_{L^2} \\
&\lesssim \|\dt^3\bm{x}\|_{X^1} + \|\bm{x}'\|_{X^2}\|\dot{\bm{x}}\|_{X^2}, 
\end{align*}
which implies $\|s^\frac12\dot{\mu}'\|_{L^2} \lesssim 1$. 
Then, in view of $\dot{\bm{x}}''=\mu^{-1}( (\mathscr{M}(\dt^3\bm{x}))' - \mu'\dot{\bm{x}}'-\dot{\mu}\bm{x}''-\dot{\mu}'\bm{x}')$ we obtain 
$\|s^\frac12\dot{\bm{x}}''\|_{L^2} \lesssim \|s^\frac12(\mathscr{M}(\dt^3\bm{x}))'\|_{L^2} + \|s^\frac12\mu\|_{L^\infty}\|\bm{x}''\|_{L^2}
 + \|\dot{\mu}\|_{L^2}\|s^\frac12\bm{x}''\|_{L^\infty} + \|s^\frac12\dot{\mu}'\|_{L^2} \lesssim1$. 
Therefore, we obtain 
\begin{equation}\label{EstdtX'}
\|\dt\bm{x}'(t)\|_{X^1} + \|\dt\mu(t)\|_{X^1} \leq C_2. 
\end{equation}
Particularly, we get $\|s^\epsilon\dt\bm{x}'(t)\|_{L^\infty} + \|s^\epsilon\dt\mu(t)\|_{L^\infty} \leq C_2(\epsilon)$ for any $\epsilon>0$.

\medskip
\noindent
{\bf Step 3.} 
Estimate for $\|\bm{x}(t)\|_{X^4}$. 
We first derive estimates for $\tau$. 
We remind that we have already $\tau(s,t)\simeq s$ and $|\tau'(s,t)|\lesssim1$. 
Therefore, the first lines in \eqref{EstTau''} and \eqref{EstTau'''} are still valid, so that we obtain 
\begin{equation}
\|s^\epsilon\tau''(t)\|_{L^\infty} + \|s^{\frac12+\epsilon}\tau'''(t)\|_{L^2} \leq C_2(\epsilon) \quad\mbox{for}\quad \epsilon>0.
\end{equation}
In view of \eqref{NormDeco} and \eqref{EstX'}, it is sufficient to evaluate $\|s^2\bm{x}''''\|_{L^2}$ to obtain an estimate for $\|\bm{x}\|_{X^4}$. 
Differentiating the hyperbolic equations in \eqref{Eq} twice with respect to $s$, we have 
$s|\bm{x}''''| \lesssim |\ddot{\bm{x}}''|+|\bm{x}'''|+|\tau''\bm{x}''|+|\tau'''|$, so that 
\begin{align*}
\|s^2\bm{x}''''\|_{L^2}
&\lesssim \|s\ddot{\bm{x}}''\|_{L^2} + \|s\bm{x}'''\|_{L^2} + \|s^\frac12\tau''\|_{L^\infty}\|\bm{x}''\|_{L^2} + \|s\tau'''\|_{L^2} \\
&\lesssim \|\ddot{\bm{x}}\|_{X^2} + (1+\|s^\frac12\tau''\|_{L^\infty}) \|\bm{x}'\|_{X^2} + \|s\tau'''\|_{L^2}.
\end{align*}
Therefore, we obtain $\|\bm{x}(t)\|_{X^4} \leq C_2$.

\medskip
\noindent
{\bf Step 4.} 
Estimate for $\|\dt\bm{x}(t)\|_{X^3}$. 
We first derive estimates for $\dot{\tau}$. 
The estimates $|\dot{\tau}(s,t)|\lesssim s$ and $|\dot{\tau}'(s,t)|\lesssim1$ in the proof of Lemma \ref{lem:EstTau1} and 
the fist line in \eqref{EstdtTau''} are still valid, so that we obtain 
\begin{equation}
|\dot{\tau}(s,t)| \leq C_2s, \quad \|\dot{\tau}'(t)\|_{L^\infty} \leq C_2, \quad
\|s^\epsilon\dot{\tau}''(t)\|_{L^2} \leq C_2(\epsilon) \quad\mbox{for}\quad \epsilon>0.
\end{equation}
In view of \eqref{NormDeco} and \eqref{EstdtX'}, it is sufficient to evaluate $\|s^\frac32\dot{\bm{x}}'''\|_{L^2}$ 
to obtain an estimate for $\|\dt\bm{x}\|_{X^3}$. 
Differentiating the hyperbolic equations in \eqref{Eq} with respect to $s$ and $t$, we have 
$s|\dot{\bm{x}}'''| \lesssim |\dt^3\bm{x}'| + |\dot{\bm{x}}''|+|\tau''\dot{\bm{x}}'| + s|\bm{x}'''|+|\bm{x}''|+|\dot{\tau}''|$, so that 
\begin{align*}
\|s^\frac32\dot{\bm{x}}'''\|_{L^2}
&\lesssim \|s^\frac12\dt^3\bm{x}'\|_{L^2} + \|s^\frac12\dot{\bm{x}}''\|_{L^2} + \|s^\frac12\tau''\|_{L^\infty}\|\dot{\bm{x}}'\|_{L^2} \\
&\quad\;
 + \|s\bm{x}'''\|_{L^2} + \|\bm{x}''\|_{L^2} + \|s^\frac12\dot{\tau}''\|_{L^2} \\
&\lesssim \|\dt^3\bm{x}\|_{X^1} + (1+\|s^\frac12\tau''\|_{L^\infty})\|\dot{\bm{x}}'\|_{X^1}
 + \|\bm{x}'\|_{X^2} + \|s^\frac12\dot{\tau}''\|_{L^2}.
\end{align*}
Therefore, we obtain $\|\dt\bm{x}(t)\|_{X^3} \leq C_2$.

\medskip
\noindent
{\bf Step 5.} 
Estimate for $\|\dt^4\bm{x}(t)\|_{L^2}$. 
We have shown that $\opnorm{\bm{x}(t)}_{4,*}\lesssim1$. 
Therefore, by Lemma \ref{lem:EstTau2} we have $\opnorm{\tau'(t)}_{3,*,\epsilon} \leq C_2(\epsilon)$ for any $\epsilon>0$. 
Particularly, $\|\dot{\tau}'(t)\|_{X^1}+\|\ddot{\tau}'(t)\|_{L^2} \lesssim 1$. 
Therefore, by Lemma \ref{lem:EstAtau} we see that 
\begin{align*}
\|\dt^4\bm{x}\|_{L^2}
&= \|\dt^2(\mathscr{A}_\tau\bm{x})\|_{L^2} \\
&\leq \|\mathscr{A}_\tau\ddot{\bm{x}}\|_{L^2} + 2\|\mathscr{A}_{\dot{\tau}}\dot{\bm{x}}\|_{L^2} + \|\mathscr{A}_{\ddot{\tau}}\bm{x}\|_{L^2} \\
&\lesssim \|\tau'\|_{L^\infty} \|\ddot{\bm{x}}\|_{X^2} + \|\dot{\tau}'\|_{X^1} \|\dot{\bm{x}}\|_{X^3}
 + \|\ddot{\tau}'\|_{L^2} \|\bm{x}\|_{X^4},
\end{align*}
which implies $\|\dt^4\bm{x}(t)\|_{L^2} \leq C_2$. 
Summarizing the above estimates, we obtain $\opnorm{ \bm{x}(t) }_4 \leq C_2$. 
The proof of Lemma \ref{lem:APEx} in the case $m=4$ is complete. 
\hfill$\Box$

\subsection{Proof of Theorem \ref{th:APE}}
We are ready prove Theorem \ref{th:APE}, which ensures a priori estimates for the solution $(\bm{x},\tau)$ to the initial boundary value problem 
\eqref{Eq}--\eqref{IC}. 
We are going to show that for any regular solution $(\bm{x},\tau)$ to the problem, if the initial data satisfy \eqref{CondID}, 
then the estimates in \eqref{EstAss} hold in fact for $0\leq t\leq T$ by choosing appropriately the positive constants $M_1$, $M_2$, 
and the positive time $T$. 
In the following, we simply denote the constants $C_0=C(M_0,c_0,m)$, $C_1=C(M_1,c_0,m)$, and $C_2=C(M_2,M_1,c_0,m)$. 
These constants may change from line to line.

Suppose that the initial data $(\bm{x}_0^\mathrm{in},\bm{x}_1^\mathrm{in})$ satisfy \eqref{CondID} and that $(\bm{x},\tau)$ is a regular 
solution to the problem \eqref{Eq}--\eqref{IC}. 
By Lemmas \ref{lem:EstID}, \ref{lem:EstTau1}, and \ref{lem:EstTau2}, we have 
\begin{equation}\label{EstID}
\begin{cases}
 \opnorm{ \bm{x}(0) }_m + \opnorm{ \tau'(0) }_{m-2,*} \leq C_0, \\
 C_0^{-1}s \leq \tau(s,0) \leq C_0s, \quad \sum_{j=1}^{m-3}|\dt^j\tau(s,0)| \leq C_0s.
\end{cases}
\end{equation}
Suppose also that the solution $(\bm{x},\tau)$ satisfies \eqref{EstAss} for $0\leq t\leq T$, where the constants $M_0$, $M_1$, and time $T$ 
will be defined later. 
Then, by Lemmas \ref{lem:APEx}, \ref{lem:EstTau1}, and \ref{lem:EstTau2}, we have 
\[
\begin{cases}
 \opnorm{ \bm{x}(t) }_m \leq C_2, \\
 C_1^{-1}s \leq \tau(s,t) \leq C_1s, \quad \sum_{j=1}^{m-3}|\dt^j\tau(s,t)|\leq C_2s, \quad |\dt^{m-2}\tau(s,t)|\leq C_2s^\frac12, \\
 \opnorm{\tau'(t)}_{m-1,*} \leq C_2 \qquad\ \mbox{in the case $m\geq5$}, \\
 \opnorm{\tau'(t)}_{3,*,\epsilon} \leq C_2 \qquad\mbox{in the case $m=4$ with $\epsilon=\frac14$}
\end{cases}
\]
for $0\leq t\leq T$. 
Here, we note that there is no special reason on the choice $\epsilon=\frac14$ 
and that we can choose $\epsilon$ arbitrarily such that $0<\epsilon<\frac12$. 
We put $\bm{y}=\dt^{m-2}\bm{x}$ and $\nu=\dt^{m-2}\tau$. 
Then, we see that $(\bm{y},\nu)$ satisfies the linearized system \eqref{LEq} and \eqref{LBVP} with $(\bm{f},f,h)$ given by 
\begin{align*}
\bm{f} &= \{ \dt^{m-2}(\tau\bm{x}')-(\dt^{m-2}\tau)\bm{x}'-\tau\dt^{m-2}\bm{x}' \}', \\
f &= -\frac12\{ \dt^{m-2}(\bm{x}'\cdot\bm{x}') - 2\bm{x}'\cdot\dt^{m-2}\bm{x}' \}, \\
h &= \{ \dt^{m-2}(\dot{\bm{x}}'\cdot\dot{\bm{x}}') - 2\dot{\bm{x}}'\cdot\dt^{m-2}\dot{\bm{x}}' \} \\
&\quad\;
 - \{ \dt^{m-2}(\tau\bm{x}''\cdot\bm{x}'') - (\dt^{m-2}\tau)\bm{x}''\cdot\bm{x}'' - 2\tau\bm{x}''\cdot\dt^{m-2}\bm{x}'' \}.
\end{align*}
Therefore, by Proposition \ref{prop:EE} we obtain the energy estimate 
\begin{equation}\label{EI1}
E(t) \leq C_1 \mathrm{e}^{C_2 t}\left( E(0) + S_1(0) + C_2\int_0^t S_2(t')\mathrm{d}t' \right), 
\end{equation}
where $E(t) = \|\dot{\bm{y}}(t)\|_{X^1}^2+\|\bm{y}(t)\|_{X^2}^2 = \|\dt^{m-1}\bm{x}(t)\|_{X^1}^2 + \|\dt^{m-2}\bm{x}(t)\|_{X^2}^2$, 
and $S_1(t)$ and $S_2(t)$ are defined by \eqref{DefS12}.

\begin{lemma}\label{lem:EstS12}
It holds that $E(0)+S_1(0)\leq C_0$ and $S_2(t)\leq C_2$. 
\end{lemma}

\begin{proof}
We first evaluate $S_2(t)$. 
In the case $m\geq5$, by Lemma \ref{lem:EstAtau} we see that 
\begin{align*}
\|\dot{\bm{f}}\|_{L^2}
&\lesssim \sum_{j_1+j_2=m-1, j_1,j_2\leq m-2} \|((\dt^{j_1}\tau)\dt^{j_2}\bm{x}')'\|_{L^2} \\
&\lesssim \|\dt\tau'\|_{X^2} \|\dt^{m-2}\bm{x}\|_{X^2}
 + \sum_{j_1\leq m-2, j_2\leq m-3} \|\dt^{j_1}\tau'\|_{X^1} \|\dt^{j_2}\bm{x}\|_{X^3} \\
&\lesssim \opnorm{ \tau' }_{m-1,*} \opnorm{ \bm{x} }_m.
\end{align*}
In the case $m=4$, we have 
$\|\dot{\bm{f}}\|_{L^2} \lesssim \|((\dt^2\tau)\dt\bm{x}')'\|_{L^2} + \|((\dt\tau)\dt^2\bm{x}')'\|_{L^2}$. 
Here, we note that 
$\|\dt^2\tau'\|_{L^3} \leq \|s^\frac14\dt^2\tau'\|_{L^\infty} \|s^{-\frac14}\|_{L^3} \leq C_2$. 
Therefore, the first term can be evaluated as 
\begin{align*}
\|((\dt^2\tau)\dt\bm{x}')'\|_{L^2}
&\leq \|s^{-\frac12}\dt^2\tau\|_{L^\infty} \|s^\frac12\dt\bm{x}''\|_{L^2} + \|\dt^2\tau'\|_{L^3} \|\dt\bm{x}'\|_{L^6} \\
&\lesssim \|\dt^2\tau'\|_{L^3} \|\dt\bm{x}\|_{X^3},
\end{align*}
where we used Lemmas \ref{lem:embedding} and \ref{lem:CalIneqLp1}. 
The second term can be evaluated by Lemma \ref{lem:EstAtau} as 
$\|((\dt\tau)\dt^2\bm{x}')'\|_{L^2} \lesssim \|\dt\tau'\|_{L^\infty}\|\dt^2\bm{x}\|_{X^2}$. 
In any case, we have $\|\dot{\bm{f}}(t)\|_{L^2} \leq C_2$ for $0\leq t\leq T$.

As for $\dot{f}$, by Lemma \ref{lem:embedding} we see that 
\begin{align*}
\|s^\frac14\dot{f}\|_{L^2}
&\lesssim \sum_{j_1+j_2=m-1,j_1\leq j_2\leq m-2} \|s^\frac14\dt^{j_1}\bm{x}'\|_{L^\infty} \|\dt^{j_2}\bm{x}'\|_{L^2} \\
&\lesssim \sum_{j_1\leq m-3, j_2\leq m-2} \|\dt^{j_1}\bm{x}\|_{X^3}\|\dt^{j_2}\bm{x}\|_{X^2} \\
&\lesssim \opnorm{ \bm{x} }_m^2,
\end{align*}
which yields $\|s^\frac14\dot{f}(t)\|_{L^2} \leq C_2$ for $0\leq t\leq T$. 
Similarly, by the standard Sobolev embedding theorem we have 
$|\dot{f}(1,t)| \lesssim \opnorm{ \bm{x}(t) }_m^2 \leq C_2$ for $0\leq t\leq T$.

As for $h$, we see that 
\begin{align*}
\|s\dot{h}\|_{L^1}
&\lesssim \sum_{j_1+j_2=m-1,j_1\leq j_2\leq m-2} \|\dt^{j_1+1}\bm{x}'\|_{L^2} \|s^\frac12\dt^{j_2+1}\bm{x}'\|_{L^2} \\
&\quad\;
 + \sum_{j_0+j_1+j_2=m-1,j_0\leq m-2, j_1\leq j_2\leq m-2}
  \|s^{-\frac12}\dt^{j_0}\tau\|_{L^\infty} \|s^\frac12\dt^{j_1}\bm{x}''\|_{L^2} \|s\dt^{j_2}\bm{x}''\|_{L^2} \\
&\lesssim \sum_{j_1\leq m-3, j_2\leq m-2} \|\dt^{j_1+1}\bm{x}\|_{X^2} \|\dt^{j_2+1}\bm{x}\|_{X^1} \\
&\quad\;
 + \sum_{j_0,j_2\leq m-2, j_1\leq m-3}
  \|\dt^{j_0}\tau'\|_{L^2} \|\dt^{j_1}\bm{x}\|_{X^3} \|\dt^{j_2}\bm{x}\|_{X^2} \\
&\lesssim \opnorm{ \bm{x} }_m^2 + \sum_{j_0\leq m-2} \|\dt^{j_0}\tau'\|_{L^2}\opnorm{ \bm{x} }_m^2
\end{align*}
and that 
\begin{align*}
\|s^\frac12h\|_{L^2}
&\lesssim \sum_{j_1+j_2=m-2,j_1\leq j_2\leq m-3} \|s^\frac12\dt^{j_1+1}\bm{x}'\|_{L^\infty} \|\dt^{j_2+1}\bm{x}'\|_{L^2} \\
&\quad\;
 + \sum_{j_0+j_1+j_2=m-2,j_0\leq m-3, j_1\leq j_2\leq m-3}
  \|s^{-1}\dt^{j_0}\tau\|_{L^\infty} \|s\dt^{j_1}\bm{x}''\|_{L^\infty} \|s^\frac12\dt^{j_2}\bm{x}''\|_{L^2} \\
&\lesssim \sum_{j_1,j_2 \leq m-3} \|\dt^{j_1+1}\bm{x}\|_{X^2} \|\dt^{j_2+1}\bm{x}\|_{X^2} \\
&\quad\;
 + \sum_{j_0,j_1,j_2\leq m-3} \|\dt^{j_0}\tau'\|_{L^\infty} \|\dt^{j_1}\bm{x}\|_{X^3} \|\dt^{j_2}\bm{x}\|_{X^3} \\
&\lesssim \opnorm{ \bm{x} }_m^2 + \sum_{j_0\leq m-3} \|\dt^{j_0}\tau'\|_{L^\infty}\opnorm{ \bm{x} }_m^2,
\end{align*}
where we used Lemma \ref{lem:CalIneqLp1}. 
Therefore, we obtain $\|s\dot{h}(t)\|_{L^1}+\|s^\frac12h(t)\|_{L^2} \leq C_2$ for $0\leq t\leq T$. 
Summarizing the above estimates, we get $S_2(t) \leq C_2$ for $0\leq t\leq T$.

It remains to evaluate $E(0)$ and $S_1(0)$. 
Since we have \eqref{EstID}, by similar evaluations as above, we obtain $E(0)+S_1(0) \leq C_0$. 
The proof is complete. 
\end{proof}

This lemma and \eqref{EI1} implies $E(t) \leq C_1 \mathrm{e}^{C_2 t}( C_0 + C_2t)$. 
On the other hand, it is easy to see that 
$\opnorm{\bm{x}(t)}_{m-1} \leq \opnorm{\bm{x}(0)}_{m-1}+\int_0^t\opnorm{\bm{x}(t')}_m\mathrm{d}t' \leq C_0 + C_2t$ and that 
$\frac{\tau(s,t)}{s} \geq \frac{\tau(s,0)}{s} - \frac{1}{s}\int_0^t|\dt\tau(s,t')|\mathrm{d}t'\geq 2c_0 - C_2t$. 
Summarizing the above estimates, we have shown 
\[
\begin{cases}
 \|\dt^{m-1}\bm{x}(t)\|_{X^1}^2 + \|\dt^{m-2}\bm{x}(t)\|_{X^2}^2 \leq C_1\mathrm{e}^{C_2 t}( C_0 + C_2t), \\
 \opnorm{ \bm{x}(t) }_{m-1} \leq C_0 + C_2t, \\
 \frac{\tau(s,t)}{s} \geq 2c_0-C_2t.
\end{cases}
\]
Now, we define the constants $M_1$ and $M_2$ by $M_1=2C_0$ and $M_2=4C_0C_1$ and then choose the time $T$ so small that 
$C_2T\leq\min\{C_0,c_0,\log2\}$. 
Then, by the standard argument we see that the solution $(\bm{x},\tau)$ satisfies in fact \eqref{EstAss} for $0\leq t\leq T$ 
and the estimates in Theorem \ref{th:APE} follows from Lemmas \ref{lem:APEx}, \ref{lem:EstTau1}, and \ref{lem:EstTau2}. 
The proof of Theorem \ref{th:APE} is complete. 
\hfill$\Box$

\bigskip
\noindent
\textbf{Data availability}
Data sharing not applicable to this article as no datasets were generated or analysed during the current study. 

\bigskip
\noindent
\textbf{\Large Declarations}

\medskip
\noindent
\textbf{Conflict of interest}
On behalf of all authors, the corresponding author states that there is no conflict of interest.


\bigskip
Tatsuo Iguchi \par
{\sc Department of Mathematics} \par
{\sc Faculty of Science and Technology, Keio University} \par
{\sc 3-14-1 Hiyoshi, Kohoku-ku, Yokohama, 223-8522, Japan} \par
E-mail: \texttt{iguchi@math.keio.ac.jp}

\bigskip
Masahiro Takayama \par
{\sc Department of Mathematics} \par
{\sc Faculty of Science and Technology, Keio University} \par
{\sc 3-14-1 Hiyoshi, Kohoku-ku, Yokohama, 223-8522, Japan} \par
E-mail: \texttt{masahiro@math.keio.ac.jp}


\begin{thebibliography}{99}
\bibitem{IguchiTakayama2023}
Iguchi, T., Takayama, M.: 
Well-posedness of the initial boundary value problem for degenerate hyperbolic systems with a localized term 
and its application to the linearized system for the motion of an inextensible hanging string. 
In preparation.
\end{thebibliography}
\end{document}